\documentclass[10pt]{amsart}
\usepackage{fullpage}
\usepackage[utf8]{inputenc}
\usepackage{microtype}
\usepackage{graphics}
\usepackage{amsmath,amssymb,stmaryrd,array,multirow,dsfont,mathrsfs}

\usepackage{amsthm}
\usepackage{enumerate}
\usepackage{xcolor}
\usepackage{color}
\usepackage{hyperref}
\usepackage{indentfirst}
\usepackage{pifont}
\usepackage[T1]{fontenc}
\usepackage{microtype}
\usepackage{graphics}
\usepackage{amsthm}
\usepackage{extarrows}
\usepackage{stmaryrd}
\graphicspath{{Images/}}
\xdefinecolor{darkgreen}{rgb}{0,0.4,0}
  
\usepackage{tikz}

\newcommand{\beq}{\begin{equation}}
\newcommand{\eeq}{\end{equation}}
\newcommand{\beqs}{\begin{equation*}}
\newcommand{\eeqs}{\end{equation*}}
\newcommand{\beal}{\begin{align}}
\newcommand{\eeal}{\end{align}}
\newcommand{\beals}{\begin{align*}}
\newcommand{\eeals}{\end{align*}}
\newcommand{\ben}{\begin{eqnarray}}
\newcommand{\een}{\end{eqnarray}}
\newcommand{\beno}{\begin{eqnarray*}}
\newcommand{\eeno}{\end{eqnarray*}}
\renewcommand{\Re}{{\rm Re}\,}

\renewcommand{\div}{{\rm div}}

\newcommand{\Lip}{{\rm Lip}\,}
\newcommand{\Id}{{\rm Id}\,}

\newcommand{\Rmnum}[1]{\uppercase\expandafter{\romannumeral #1} }
 \numberwithin{equation}{section}

\allowdisplaybreaks[4]
\newtheorem{thm}{Theorem}[section]
\newtheorem{lem}[thm]{Lemma}
\newtheorem{prop}[thm]{Proposition}
\newtheorem{rmk}[thm]{Remark}

\def\curl{\mathop{\rm curl}\nolimits}

\def \d {\mathrm {d}}
\def\cA{{\mathcal A}}
\def\cB{{\mathcal B}}
\def\cC{{\mathcal C}}

\def\cE{{\mathcal E}}
\def\cF{{\mathcal F}}
\def\cG{{\mathcal G}}
\def\cH{{\mathcal H}}

\def\cJ{{\mathcal J}}

\def\cL{{\mathcal L}}
\def\cM{{\mathcal M}}
\def\cN{{\mathcal N}}
\def\cO{{\mathcal O}}

\def\cS{{\mathcal S}}
\def\cT{{\mathcal T}}

\let\f=\frac
\def \p {\partial}
\def\mR {\mathbb{R}}
\def\pa {\partial^{\alpha}}
\def\ep{\epsilon}

\def \INS {\text{INS}}
\def \diag {\text{diag}}

\def \pt {\partial_{t}}

\def\si {\sigma}
\def\na{\nabla}

\def \oH {\overline{H}}

\def \bn {\textbf{n}}

\def \sym {\text{sym}}

\title{Incompressible and vanishing vertical viscosity limit for the compressible Navier-Stokes system with Dirichlet boundary conditions}
\author{Nader Masmoudi, Changzhen Sun, Chao Wang, Zhifei Zhang} 
\address{NYUAD Research Institute, New York University Abu Dhabi, PO Box 129188, Abu Dhabi,   United Arab Emirates.
 Courant Institute of Mathematical Sciences, New York University, 251 Mercer Street, New York, NY 10012, USA.}
\email{masmoudi@cims.nyu.edu}
\address{Université Marie et Louis Pasteur, CNRS, Laboratoire de Math\'ematiques de Besançon (UMR 6623), F-25000 Besançon, France.}
\email{changzhen.sun@univ-fcomte.fr}

\address{School of Mathematical Sciences, Beijing (Peking) University,
Beijing 100871, People's Republic of China.}
\email{wangchao@math.pku.edu.cn, zfzhang@math.pku.edu.cn }

\date{\today}
\begin{document}
\maketitle
\begin{abstract}
    In this paper, we show the incompressible and vanishing vertical viscosity limits for the strong solutions to the isentropic compressible Navier-Stokes system with anistropic dissipation, in a domain with \textit{Dirichlet} boundary conditions in the general setting of \textit{ill-prepared} initial data. We establish the uniform regularity estimates with respect to the Mach number $\ep$ and the vertical viscosity  $\nu$  so that the solution exists on a uniform time interval $[0,T_0]$ independent of these parameters. 
    The key steps toward this goal are the careful construction of the approximate solution in the presence of both fast oscillations and two kinds of  boundary layers together with 
the stability analysis of the remainder. In the process, it is also shown that the solutions of the compressible systems converge to those of the incompressible system with only horizontal dissipation, after removing the fast waves whose horizontal derivative is bounded in $L_{T_0}^2L_x^2$ by $\min\{1, (\ep/\nu)^{\f14}\}.$ 

\end{abstract}
\section{Introduction and backgrounds}
  \quad We are interested  here in the justification of the  incompressible inviscid limit for the (scaled) isentropic compressible Navier-Stokes system $(\text{CNS})_{\ep,\nu}:$ 
  \beq \label{CNS-O}
\left\{
\begin{array}{l}
 \displaystyle\pt \rho^{\epsilon,\nu} +\div( \rho^{\epsilon, \nu} u^{\epsilon,\nu})=0,\\
 \displaystyle\pt ( \rho^{\epsilon,\nu} u^{\epsilon,\nu})+\div(\rho^{\epsilon,\nu}u^{\epsilon,\nu}\otimes u^{\epsilon,\nu})+
\f{\nabla P^{\ep,\nu}}{\ep^2}-(\mu_1\Delta_{\nu} + \mu_2
\na\div )u^{\epsilon,\nu} 
=0,  \\[5pt]
\rho^{\ep,\nu}|_{t=0}=\rho_0^{\ep,\nu}, \quad u^{\ep,\nu}|_{t=0}=u_0^{\ep,\nu}|_{t=0},
\end{array}
\right. 
\eeq
where $(t,x)\in \mathbb{R}_{+}\times \Omega,$ 
 $\Omega\subset \mathbb{R}^3$
being a smooth bounded domain. The system is supplemented 
with  the Dirichlet boundary condition:
\beq \label{bdyconditions}
u^{\ep,\nu}|_{\p\Omega}=0 \,. 
\eeq
In the system \eqref{CNS-O}, 
$\rho^{\ep,\nu}(t,x)$ and $u^{\ep,\nu}(t,x)$ represent respectively the density and the velocity of the fluid, $P^{\ep,\nu}$ is the  pressure 
which is  a given smooth function of the density. The last term in the second equation is the viscous term where
 $\mu_1 , \mu_2$ are two fixed viscosity constants satisfying $$\mu_1>0, \qquad \mu_1\nu+\mu_2>0$$ and $\nu\in (0,1]$ could be considered as the inverse of the Reynolds number which is assumed to be small. Moreover, here we consider the anistropic dissipation where the vertical dissipation is much smaller than the horizontal one: 
\beqs 
\mu_1\Delta_{\nu}=\mu_1(\Delta_{h}+\nu\p_z^2),\, \quad (\Delta_{h}=\p_{x_1}^2+\p_{x_2}^2)\,.
\eeqs
Finally, the parameter $\ep$ is the scaled Mach number which  is  assumed  small, that is $\ep\in(0,1].$
The inverse of the Reynolds number $\nu$ is related to the Mach number $\ep$ in the sense that 
$\nu=\nu(\ep)\rightarrow 0$ as $\ep\rightarrow 0.$ 

 Taking the smallness of the Mach number into account, the scaled system \eqref{CNS-O} can be derived from the original compressible Navier-Stokes system by introducing suitable change of variable. Indeed, one finds \eqref{CNS-O} by performing the following scaling:
   $$\rho(t,x)=\rho^{\ep,\nu}(\ep t,x),\, u(t,x)=\ep u^{\ep,\nu}(\ep t,x),\,\,\tilde{\mu}_h=\ep \mu_1,\, \,\tilde{\mu}_v= \mu_1\ep \nu, \,\,\tilde{\mu}_2=\mu_2\ep\nu,$$
   where $(\rho, u, \cT)$ solve the following Navier-Stokes system with anistropic dissipation:
   \beq \label{NCNS-O}
 \left\{
\begin{array}{l}
 \displaystyle\pt \rho +\div( \rho u)=0,\\
 \displaystyle\pt ( \rho u)+\div(\rho u\otimes u)+
\nabla P-(\tilde{\mu}_h\Delta_{x,y}+\tilde{\mu}_v\p_z^2+\tilde{\mu}_2\na\div)u=0. \\  
\end{array}
\right.\\
\eeq
We remark that the anistropic dissipation considered here is a common occurrence in geophysical fluids. As a matter of fact, instead of incorporating the conventional dissipation $-\tilde{\mu}\Delta$ of the fluid into the equations, meteorologists frequently depict turbulent diffusion using a viscosity represented as $-\tilde{\mu}_h\Delta_{z,y}-\tilde{\mu}_v\p_z^2$. Here, $\tilde{\mu}_h$ and $\tilde{\mu}_v$ are empirical constants, with $\tilde{\mu}_v$ being significantly smaller than $\tilde{\mu}_h$. For example, in the ocean,  $\tilde{\mu}_v$ ranges from 1 to $10^3 \text{cm}^2/\text{sec} $, while $\tilde{\mu}_h$ varies from $10^5$ to $10^8 \text{cm}^2/\text{sec}$. 
More extensive discussions can be found in 
the book [Chapter 4, \cite{pedlosky2012geophysical}].

In the following, for the clarity of the presentation, we shall consider the fluid domain to be the strip:
$\Omega=\mathbb{T}_{a_1,a_2}^2\times [0,a_3],$ where $\mathbb{T}_{a_1,a_2}^2=\mathbb{T}_{a_1}\times \mathbb{T}_{a_2}$
are two dimensional torus with length $a_1, a_2,$ and $a_1,a_2,a_3>0.$ 
Moreover, we allow the initial data to be ill-prepared in the sense that 
\beq\label{generaldata}
(\div u^{\ep,\nu}, \na P^{\ep,\nu}/\ep)=\cO(1)
\eeq
in some suitable space. In  other words, we do not 
impose any smallness assumption on the velocity and assume only $\rho=\bar{\rho}+\cO(\ep).$

  The incompressible limit and the inviscid limit are two important singular limits in fluid mechanics. When the 
 Mach number is small, the compressible fluid system
can be approximated by the corresponding incompressible ones and this limit process is referred to as the incompressible limit. When the Reynolds number is very high (that is $\nu$ is very small), the viscous fluid could be characterized  approximately by the inviscid system, this limit process is called the inviscid limit. 
\textit{Our aim is to show that as $\ep, \nu=\nu(\ep)$ tends to 0, the solution to the system $(\text{CNS})_{\ep,\nu}$ converge (in some suitable space) 
to the solution of the incompressible Navier-Stokes system with only horizontal dissipation.}

    Let us review first some investigations concerning the incompressible limit (that is $\ep$ tends to $0$ while $\nu$ being fixed) for compressible Navier-Stokes system (CNS)$_{\ep},$ which has been a vast project so far.
   The initial works, credited to Ebin \cite{MR431261}, Klainerman-Majda \cite{MR615627,MR668409}, focus on the study of the local \textit{strong} solution of (CNS)$_{\ep}$ occupied in a domain \textit{without} boundaries ($\mR^3$ or $\mathbb{T}^3$), with well-prepared initial data ($(\div\, u_0^{\ep},\nabla P_0^{\ep}/\ep)=\cO(\ep)$). Later, the same problem is considered by Ukai \cite{MR849223} in the whole space and Gallagher \cite{MR1794519} on 
   the torus for ill-prepared initial data ($(\div u_0^{\ep},\nabla P_0^{\ep}/\ep)=\cO(1)$). 
  The incompressbile limit of \textit{weak} solutions for (CNS)$_{\ep}$ was first investigated by Lions and and the first author \cite{MR1628173,MR1710123}. 
  In general, for ill-prepared data, one can only obtain weak convergence in time. Nevertheless, when the Dirichlet boundary condition \eqref{bdyconditions} is imposed, the authors in \cite{MR1697038} are able to show the local strong convergence by employing the damping effects of acoustic waves due to the strong boundary layers. 
There are also many other related works, one can see for instance \cite{MR1917042,chen-hao-gui-jiang, MR1886005,MR1702718,MR2575476,MR918838,MR3916820,MR3240080}. We refer to the introduction of \cite{sun2022uniform}  for a more thorough review of the literature.  
 Additionally, for comprehensive information, we recommend the reader consulting the well-written survey papers by 
 Alazard \cite{MR2425022}, Danchin \cite{MR2157145}, Feireisl \cite{MR3916821}, Gallagher \cite{MR2167201},
 Jiang-Masmoudi \cite{MR3916820}. 

 The current work is partly motivated by the verification of the incompressible  limit  for \textit{strong} solutions to $\text{(CNS)}_{\epsilon}$ with \textit{Dirichlet} boundary conditions \eqref{bdyconditions} and general data \eqref{generaldata}.
The usual studies on the strong solutions consists in first ensuring that for any $\ep\in (0,1]$, the system  
   $\text{(CNS)}_{\epsilon}$ admits a strong solution
 on a time interval that is independent of the Mach number $\ep,$ and then proving the convergence result in this fixed interval.  
The first step is usually based on some uniform regularity estimates independent of  $\ep.$ The previous studies on the incompressible limit for the strong solutions focus either on the ill-prepared data in absence of boundaries \cite{MR1794519,MR2211706}  or the well-prepared data with boundaries \cite{MR2812710,MR3240080}.
Nevertheless, when a boundary is present, the question of uniform regularity for ill-prepared data is more subtle due to the appearance of boundary layers. Recently, the first two authors together with Rousset \cite{MFS-JMPA} proved uniform 
regularity estimates for (CNS)$_{\ep}$ with the Navier-slip boundary condition:
\beq\label{navier-slipbc}
u^{\ep}\cdot \bn=0, \quad
 \Pi(\mathbb{S} u^{\ep} \bn)+a\Pi u^{\ep}=0\quad
  \text{on }  {\partial{\Omega}},
  \eeq
where $\bn$ is the unit outward normal vector and  $a$ is a constant related to the slip length. In the case of 
Navier-slip boundary condition, the boundary layer effects is weaker due to the absence of the first boundary layer profile and thus the Lipschitz norm of the velocity is expected to be uniformly bounded. Consequently, uniform regularity estimates could be established in \cite{MFS-JMPA} 
 in the framework of conormal spaces which take into accounts the boundary layer effects, without providing detailed descriptions of the solutions.  However, due to the \textit{strong} boundary layers, the case with Dirichlet boundary conditions and general initial data is more subtle and has not been addressed. 

 We now switch to the study of the the vanishing viscosity limit and first put ourselves in a more general setting where both the horizontal and vertical viscosity tend to $0.$ In the domain without boundary or with boundary and Navier-slip boundary condition, one can either start from a global weak solution for the Navier-Stokes system and prove the convergence of this solution to the strong solutions of the limiting inviscid system-- as in \cite{iftimie-planas,Iftimie-sueur,sueur}; or establish uniform high regularity estimates in the conormal spaces and justify the limit process by using the strong compactness. The latter is manageable mainly due to the fact that the boundary layer effect is weaker when the Navier-slip boundary condition is considered.
However, in the domain with boundaries and with the Dirichlet boundary condition, there is a strong (Prandtl) boundary layer whose governed equation is only solvable in the analytic spaces (or slightly more general Gevery space) without the monotonicity 
 assumption \cite{MR3925144,MR2601044,MR4465902}.  The corresponding inviscid limit could thus only be achieved by imposing the spatial analyticity \cite{ MR4097327, MR3851056, MR1617542, MR1617538, MR3614755}. All of the above works concern the incompressible system. However, when dealing with compressible fluids, working in the real analytic space requires controlling infinitely many time derivatives. This, in turn, necessitates imposing an infinite number of compatibility conditions on the data, which is quite restrictive. 
 Nevertheless, for the case with anisotropic dissipation considered in the present work, the situation is better, since the extra horizontal dissipation in the Prandtl equation allows one to prove the well-posedness in the usual Sobolev spaces.  One could thus expect to construct the approximate solution to sufficiently high order and then to prove the stability of the remainder (that is, the remainder is small in $L^{\infty}([0,T],L^2(\Omega))$ for some $T>0$ uniform in the viscosity parameter). This has been justified in \cite{MR3843301} for the incompressible fluid. Nevertheless, as will be explained later, there are some extra obstacles for the \textit{compressible} fluids when considering also the low Mach number limit with general initial data.

For the study of the incompressible and vanishing viscosity limit, 
there are only a few works, we quote \cite{MR1808029} for the investigations in the whole space and the torus without boundaries. 
In \cite{MR3070031}, the inviscid incompressible limit problem is studied for the weak solution of the non-isentropic system and  a convergence result is obtained as long as the viscosity and the Mach number are related $(\nu=\ep^{\alpha},\, 0<\alpha<{10}/{3}).$
All the above-mentioned results are concerned with the weak solution, whose global existence is known thanks to the work of Lions \cite{MR1637634}.
 For the study of the strong solution, Alazard \cite{MR2211706} proved the uniform regularity estimates in the Mach number and the Reynolds number for  the non-isentropic Navier-Stokes system in the whole space $\mR^3,$ which ensures the existence of the solution on a time interval independent of the Mach number and Reynolds number. This result was extended recently to the case with boundaries \cite{sun2022uniform} incoporated with the Navier-slip boundary condition.  However,
 to our best knowledge, there is no corresponding result in the case of the Dirichlet boundary condition.  
 
    The most nontrivial part for the case of Dirichlet boundary condition lies in the uniform (in the Mach number and the Reynolds number) well-posedness in some suitable functional framework. 
    To ensure the propagation of high regularity estimates, it is necessary to have uniform control over the $L_x^{\infty}$ 
  norm, which is not easily attained.  
  Indeed, due to the vanishing vertical viscosity, one could not expect to use the energy estimates and the Sobolev embedding to get a uniform control of $L_t^2L_x^{\infty}$ norm. It is thus unlikely to obtain such uniform regularity estimates by directly working on the original system.
    We will henceforth adopt another methodology, which consists in the construction of the approximate solution and then study the stability of the remainder. More precisely, we aim to construct an approximate solution so that the error between the real solution and the approximate solution is small in some suitable space and its $L_x^{\infty}$ norm could be controlled uniformly, which enables one to show that the remainder exists on a time interval independent on $\ep$ and $\nu.$ Although this strategy is quite classical, it offers the advantage of providing  a more precise description of the solution to $\text{(CNS)}_{\epsilon,\nu}$ in the low Mach number and vanishing vertical viscosity limits. 
     By doing so, we have to deal with the following
    two questions: 
    
    (1) A careful construction of the approximate solution, 

    (2) Stability analysis of the remainder. 

Let us remark that both of these two questions are non-trivial.
   To construct an approximate solution, it is crucial to provide precise descriptions of both interior profiles and boundary layers. As mentioned earlier, previous studies \cite{MR1303036,MR1794519,MR1465607} primarily focused on scenarios where physical boundaries were absent. In those cases, the authors employed a group associated with the linear singular operator $\ep^{-1}\left( \begin{array}{cc}
0 & \div \\
\na & 0
\end{array}\right)$ to filter rapid oscillations in time. This approach could give a very precise description of the main order of the solution (i.e., the first interior profile). However, in our current context, we need to extend this `group method' to address cases with boundaries where the Dirichlet boundary conditions are imposed.
The presence of boundaries introduces complexities in the structure of the expression for the approximate solution and gives rise to new terms in the equation of the first interior profile. More precisely, the horizontal component of the velocity in the first interior profile does not usually vanish on the boundaries, which differs from the Dirichlet boundary condition for the real solution. To make up this  discrepancy in the boundary conditions, a boundary layer has to be introduced in the vicinity of the boundaries.  
As we will see in the next subsection, there are mainly two kinds of boundary layers. The first layer, with a width of $\sqrt{\ep\nu\,}$, needs to be identified in the fast scale $\tau=t/\ep$ to address the mismatch in the horizontal boundary conditions between the compressible (highly oscillatory) parts of the Navier-Stokes equations with and without vertical dissipation. 
The other one--a Prandtl type boundary layer with width $\sqrt{\nu\,}$-- 
    is relatively hidden and must be introduced to compensate  for the mismatch of the boundary conditions of the incompressible Navier-Stokes with and without vertical dissipation. 
    Due to the appearance of multiple scales both in time $(t, t/\ep)$ and space $(z, {z}/{\sqrt{\ep\nu}}, {z}/{\sqrt{\nu}}),$ it is not practical to construct the expansion to the very high order. This makes the stability analysis of the remainder tougher since its normal derivatives cannot be bounded uniformly in the usual Sobolev spaces (say in $L_t^{\infty}L^2$). This differs significantly from previous works \cite{MFS-JMPA, sun2022uniform}, where the Lipschitz norm of the velocity is proven to be uniformly bounded. Here, we must carefully track the weights of each quantity involved in the energy functionals to propagate the uniform regularity of the remainder with respect to both $\ep$ and $\nu.$




 \subsection{Approximate solution} 

 For simplicity, we assume that the pressure $P$ satisfies a polytropic $\gamma$ law: $P(\rho^{\ep,\nu})=\f{1}{\gamma}(\rho^{\ep,\nu})^{\gamma}, \gamma>1.$ Define  
 \beqs 
\sigma^{\ep,\nu}= \f{2}{(\gamma-1)}\bigg((\rho^{\ep,\nu})^{\f{\gamma-1}{2}}-1\bigg)\big/\ep\,,
 \eeqs
 the system \eqref{CNS-O} reduces then to the following system: 
\beq \label{CNS-N}
\left\{
\begin{array}{l}
 (\displaystyle\pt+ u^{\epsilon,\nu}\cdot\na) \sigma^{\ep,\nu}+\f{\gamma-1}{2}\sigma^{\ep,\nu}\div u^{\ep,\nu}+\f{\div u^{\epsilon,\nu}}{\ep}=0,\\[8pt]
 (\displaystyle\pt+ u^{\epsilon,\nu}\cdot\na) u^{\ep,\nu}+\f{\gamma-1}{2}\sigma^{\ep,\nu}\na \sigma^{\ep,\nu}+
\f{\nabla \sigma^{\ep,\nu}}{\ep}-\,\div_{\nu }\mathcal{S}u^{\epsilon,\nu}=\ep f^{\ep,\nu},  
\end{array}
\right.
\eeq
with 
$$\div_{\nu}\,\cS=\mu_1\big(\Delta_{h}+\nu\p_z^2\big)+\mu_2\na\div ,\quad  f^{\ep,\nu}\colon=\f{1}{\ep}\big[(\f{\gamma-1}{2}\ep\sigma^{\ep,\nu}+1)^{-\f{2}{\gamma-1}}-1\big]\div_{\nu}\cS u^{\ep,\nu}.$$
Denote $U^{\ep,\nu}=(\sigma^{\ep,\nu}, u^{\ep,\nu})^t,$ 
the equation \eqref{CNS-N} can be written further as
\beq\label{cns-abs} 
(\pt -{L}/{\ep})U^{\ep,\nu}-\left(\begin{array}{cc}
   0   & 0 \\
    0 & \div_{\nu}\,\cS
\end{array}\right) U^{\ep,\nu}+Q(U^{\ep,\nu}, U^{\ep,\nu})=\left( \begin{array}{c}
   0   \\
   \ep\nu f^{{\ep,\nu}}
\end{array}\right) ,
\eeq
where the penalized linear operator $L$ and the quadratic nonlinear form $Q(\cdot, \cdot)$ are defined as
\beqs 
L=-\left( \begin{array}{cc}
    0 &  \div \\
    \na  & 0 
\end{array}\right),\qquad Q(U^{\ep,\nu}, U^{\ep,\nu})= (u^{\ep,\nu}\cdot\na )U^{\ep,\nu}-\f{\gamma-1}{2}\sigma^{\ep,\nu} LU^{\ep,\nu}\, .
\eeqs
To obtain an approximate solution, it is crucial to possess a conceptual understanding of its intended structure. The complexity arises from the emergence of boundary layers near the boundaries.
Let us first keep $\nu$ fixed and send $\ep$ to $0.$ 
In the fast scale $\tau=t/\ep,$ the linearized system of 
\eqref{cns-abs} reads 
\beqs 
(\p_{\tau}-L)V^{\ep}-\left(\begin{array}{cc}
   0   & 0 \\
    0 & \ep\, \div_{\nu}\,\cS
\end{array}\right)V^{\ep}=0\,, \quad V^{\ep}|_{\p\Omega}=0\, .
\eeqs
When $\ep$ tends to $0,$ the above equation formally tends to $(\p_{\tau}-L)V^0=0$ whose solvability only requires the boundary condition $V^0_3|_{\p\Omega}=0.$
In order to make up this  discrepancy in the boundary conditions, a boundary layer with a width $\sqrt{\ep\nu},$ under the form $U^{B}(\tau, t, x_h, \f{z}{\sqrt{\ep\nu}}),$
needs to be introduced. 

Moreover, when the solution belongs to the kernel of the singular operator $L,$  meaning $\sigma^{\ep,\nu}$ is a constant and $u^{\ep}$ is divergence-free, the system degenerates to the following incompressible Navier-Stokes system with anistropic dissipation (we neglect initially the interactions between the compressible and incompressible components): 
\beqs 
\pt u^{\nu}+ u^{\nu}\cdot\na u^{\nu}-\div_{\nu}\cS u^{\nu}+\na P^{\nu}=0, \quad u^{\nu}|_{\p\Omega}=0\,.
\eeqs
The limit system (i.e. $\nu=0$) requires only the boundary condition $u_3^{\nu}|_{\p\Omega}=0$ and thus a boundary layer of width $\sqrt{\nu}$ of the form $v^p(t,x_h, \f{z}{\sqrt{\nu}})$ has to be introduced during the limit process. 

Based on the above analysis, we could expect that an admissible 
 approximate solution has the following form which consists of three kinds of interior profiles and the boundary layers profiles:
\beq\label{expansion}
\begin{aligned}
U^{a,M}&\sim \sum_{j=0}^{M} \ep^j\bigg[ U_{osc}^{I,j}(\f{t}{\ep},t,x)+ U_{osc}^{B,j}\big(\f{t}{\ep},t,x_h,\f{z}{\sqrt{\ep\nu}}\big)+\tilde{U}_{osc}^{B,j}\big(\f{t}{\ep},t,x_h,\f{a_3-z}{\sqrt{\ep\nu}}\big)\bigg]\\
& + \sum_{k=1}^{M} (\sqrt{\ep\nu})^k \bigg[ S_{osc}^{I,k}(\f{t}{\ep},t,x)+ S_{osc}^{B,k}\big(\f{t}{\ep},t,x_h,\f{z}{\sqrt{\ep\nu}}\big)+\tilde{S}_{osc}^{B,k}\big(\f{t}{\ep},t,x_h,\f{a_3-z}{\sqrt{\ep\nu}}\big)\bigg] \\
& + \sum_{\ell=0}^{M} (\sqrt{\nu})^{\ell}\bigg[
V^{I,\ell}(t,x)
+V^{p,\ell}(t,x_h,\f{z}{\sqrt{\nu}})+
\tilde{V}^{p,\ell}\big(t,x_h,\f{a_3-z}{\sqrt{\nu}}\big)\bigg],
\end{aligned}
\eeq
where $M\geq 1$ is an integer. In the above expansion, $U_{osc}^{I,j}, S_{osc}^{I,k}$ represent the interior oscillating compressible parts and in particular $U_{osc}^{I,0}$  has the form $\cL\big(\f{t}{\ep}\big)(W(t,x)),$ $\cL$ being the semigroup generated by the linear operator $L$ in the space $V_{\sym}^m$ (see \eqref{def-vsym} for the definition).
Moreover,
the profile
$V^{I,0}=\left( \begin{array}{c}
   0   \\
   v^{\INS}\end{array}\right)
$ lies in the kernel of the operator $L$ and $ v^{\INS}$ can be taken as the solution to the incompressible Navier-Stokes system \textit{with horizontal dissipation}. 
Finally, the oscillating boundary layers emanating from the bottom and top have the form
\beqs
U_{osc}^{B,j}=\left( \begin{array}{c} 
0\\
\mathcal{B}_h^j\\
\sqrt{\ep\nu}\mathcal{B}_3^{j+1}\end{array}\right); \quad 
\tilde{U}_{osc}^{B,j}=\left( \begin{array}{c} 
0\\
\mathcal{T}_h^j\\
\sqrt{\ep\nu}\mathcal{T}_3^{j+1}\end{array}\right); \quad 
S_{osc}^{B,j}=\left( \begin{array}{c} 
0\\
\tilde{\mathcal{B}}_h^j\\[2pt]
\sqrt{\ep\nu}\tilde{\mathcal{B}}_3^{j+1}\end{array}\right); \quad 
\tilde{S}_{osc}^{B,j}=\left( \begin{array}{c} 
0\\
\tilde{\mathcal{T}}_h^j\\[2pt]
\sqrt{\ep\nu}\tilde{\mathcal{T}}_3^{j+1}\end{array}\right)
\eeqs
and the (compressible) Prandtl type boundary layers 
$V^{p,\ell}=\left( \begin{array}{c} 0\\v_{h}^{p,\ell}\\[3pt]\sqrt{\nu}{v}_{3}^{p,\ell+1}\end{array}\right), \tilde{V}^{p,\ell}=\left( \begin{array}{c} 0\\\tilde{v}_{h}^{p,\ell}\\[3pt]\sqrt{\nu}{\tilde{v}}_{3}^{p,\ell+1}\end{array}\right).$
Let us remark that in \eqref{expansion}, we need do expansion also in terms of $\ep$ due to the appearance of the fast time variable $t/\ep.$ 
Finally, as we will see in Section 2, the boundary layers $U_{osc}^{B,j}$ do not in general vanish initially. Consequently, for the 
general initial data that do not depend on the fast variable $\f{z}{\sqrt{\ep\nu}},\, \f{z}{\sqrt{\nu}}$, we will  just impose $U^{I,0}\colon= U_{osc}^{I,0}+\left( \begin{array}{c}
   0   \\
   v^{\INS}
\end{array}\right)\big|_{t=0}=U^{\ep,\nu}|_{t=0},$ so that 
\beq\label{GIC}U_0^R=U^{\ep,\nu}|_{t=0}-U^a|_{t=0}=-(U^a|_{t=0}-U^{I,0}|_{t=0})\, ,
\eeq
and henceforth
\beqs 
\|U_0^R\|_{L^2(\Omega)}+\sqrt{\ep\nu}
\|\p_z U_0^R\|_{L^2(\Omega)}\lesssim ({\ep\nu\,})^{1/4}.
\eeqs
As a result, it is indeed not very useful to construct $U_{osc}^{B,j},\, V_{osc}^{B,j} $ for $j\geq 1,$ since the above estimate prevents the remainder $U^R=U^{\ep}-U^a$ from behaving better than $({\ep\nu\,})^{1/4}$ in $L_t^{\infty}L_x^2\, .$
\subsection{Main results}
Before stating the main results, we first introduce the conormal Sobolev space. 
Define the vector fields
$$Z_0=\ep\pt, \quad Z_1=\p_{x_1}, \quad  Z_2=\p_{x_2}, \quad Z_3=\phi(z)\p_z$$
where $\phi(z)=z(a_3-z)$ which vanishes on the boundaries. 
We introduce the following Sobolev conormal spaces: for $p\in [1,+\infty], q\in [2,+\infty]$
\beq\label{def-conormal}
L_t^p W_{co}^{m,q}=\big\{f\,
\big|\,Z^{\alpha}f\in L^{p}\big([0,t],L^q(\Omega)\big), \forall \,\alpha\in\mathbb{N}^4,\,|\alpha|\leq m\big\},
\eeq
equipped with the norm
\beqs
\|f\|_{L_t^p W_{co}^{m,q}}=\sum_{|\alpha|\leq m}\|Z^{\alpha} f\|_{L^{p}([0,t],L^q(\Omega))}\,.
\eeqs
It is also useful to define also the space with only tangential regularity
\beq\label{def-tanconormal}
L_t^p W_{tan}^{m,q}=\big\{f\,
\big|\,Z^{\alpha}f\in L^{p}\big([0,t],L^q(\Omega)\big), \forall \,\alpha\in\mathbb{N}^4,\,|\alpha|\leq m,\alpha_3=0\big\}
\eeq
equipped with the norm 
\beqs
\|f\|_{L_t^p W_{tan}^{m,q}}=\sum_{|\alpha|\leq m,\alpha_3=0}\|Z^{\alpha} f\|_{L^{p}([0,t],L^q(\Omega))}\,.
\eeqs
We denote for simplicity $\|f\|_{L_t^p H_{co}^{m}}=\|f\|_{L_t^p W_{co}^{m,2}}.$ Moreover, we use the following notations 
\begin{align*}
L_t^p \underline{H}_{co}^{m}=\big\{f\,
\big|\,Z^{\alpha}f\in L^{p}\big([0,t],L^2(\Omega)\big), \forall \,\alpha\in\mathbb{N}^4,\,|\alpha|\leq m, \alpha_0\leq m-1\big\},\, \\
L_t^p \underline{H}_{tan}^{m}=\big\{f\,
\big|\,Z^{\alpha}f\in L^{p}\big([0,t],L^2(\Omega)\big), \forall \,\alpha\in\mathbb{N}^4,\,|\alpha|\leq m, \alpha_0\leq m-1, \alpha_3=0\big\}.\, 
\end{align*}

To measure pointwise regularity at a given time $t$ (in particular also with $t=0$), we shall use the semi-norms
 \beq
 \label{normfixt}\|f(t)\|_{H_{co}^m}=
 \sum_{|\alpha|\leq m}\|(Z^{\alpha} f)(t)\|_{L^2(\Omega)}.
 \eeq

The main results of the current work is summarized in the following two theorems. The first one concerns the existence of the approximate solution.

\begin{thm}[Existence of the approximate solution]\label{thm-appsol}
Suppose that $a_1, a_2,a_3$ is such that the small divisor estimate \eqref{no-resonantasp} holds true for some $r_0.$
There exist  ${T}_1>0$ which is independent of $\ep,\nu,$ and 
   an approximate solution $U^a$ (defined by \eqref{app sol}) such that $
   (\ep\pa_t)^iU^a\in L_{T_1}^{\infty}H_{co}^{6-i}$ with $i\in[0, 2]$, 
and satisfies the approximate equation
    \beq\label{eq-ua} 
\big(\pt-L^{\ep}/\ep\big) U^a -\left(\begin{array}{cc}
   0   & 0 \\
    0 & \div_{\nu}\,\cS
\end{array}\right) U^{a}+ Q(U^a, U^a)=R^a, \quad u^a|_{\p\Omega}=0,
\eeq
where the error satisfies 
\beq\label{es-error}
\|R^a\|_{L_{T_1}^1\underline{H}_{co}^3}
+\sqrt{\ep\nu}
\|\p_z R^a\|_{L_{T_1}^1H_{co}^2}
\lesssim \ep+({\ep \nu\,})^{\f14}+\nu^{\f34}:= \eta\,.
\eeq
Moreover, we have for any $2\leq p<+\infty$
\begin{align*}
    \|U^a-\left(\begin{array}{c}
   0    \\
    v^{\INS} 
\end{array}\right)-U_{osc}^{I,0} \|_{L^2([0,T_1],L^p(\Omega))}\lesssim \ep+\nu^{\f{1}{2p}}\xlongrightarrow{\ep,\nu\rightarrow 0} 0.
\end{align*}
Here $v^{INS}$ solves the incompressible Naiver-Stokes system with only horizontal dissipation \eqref{eq-vINS}, $U_{osc}^{I,0}$ is the first interior oscillating part defined in \eqref{form-UI0} and satisfies
\beqs 
\|\nabla_h U_{osc}^{I,0}\|_{L^2([0,T_1]\times \Omega)}\lesssim  \min\{1,(\ep/\nu)^{1/4} \}\xlongrightarrow{\ep,\nu\rightarrow 0} 0, \quad \text{ if } (\ep/\nu)\xlongrightarrow{\ep,\nu\rightarrow 0}  0.
\eeqs
\end{thm} 


\medskip

The next result concerns the stability of the remainder $$U^R=U^{\ep,\nu}-U^a:=(\sigma^R, u^R).$$ 
Before stating the main result, we first derive the equation satisfied by $U^R.$
Substituting the equation 
\eqref{eq-ua} from \eqref{cns-abs}, we find that $U^R$ solves
\beq\label{eq-UR}
\left\{\begin{array}{l}
(\pt-L/\ep)U^R-\left(\begin{array}{c}
   0    \\
    \f{1}{\rho}\div_{\nu}\cS u^R
\end{array}\right)+Q(U^a+U^R, U^R)+Q(U^R,U^a)= -R^a,\\[10pt]
u^R|_{z=0}=u^R|_{z=a_3}=0;\quad U^R|_{t=0}=U_0^R,
\end{array}\right.
\eeq
where $R^a$ is defined in \eqref{eq-ua} and 
\begin{align*}
\f{1}{\rho}=\f{1}{\rho}\big(\ep(\sigma^R+\sigma^a)\big)=\bigg[\f{\gamma-1}{2}\ep(\sigma^{a}+\sigma^R)+1\bigg]^{-\f{2}{\gamma-1}}
\end{align*}
and the quadratic form
\beq \label{def-quadratic}
Q(U,V)=(u\cdot \na) V+\f{\gamma-1}{2}\sigma L V, \quad ( U=(\sigma, u)^t).
\eeq
\begin{thm}[Stability of the remainder]\label{thm-remainder}
 
Suppose that the initial data $U^R(0):=U^{\ep,\nu}(0)-U^a(0)$
belongs to the space $\cN_0$  defined in \eqref{def-cNT} and satisfies the 
compatibility condition  
 \beq\label{comp-cond}
  u^R(0)|_{\p\Omega}= \pt u^R(0)|_{\p\Omega}=0. 
\eeq
    Let also $\nu=\ep^{\kappa},$ with $\kappa\in (0,3).$
Then there exists a uniform time $T_0>0,$ such that
there exists a unique solution $U^R$ to the system \eqref{eq-UR} in $C([0,{T}_0],H_{co}^{2})$ and 
\beqs 
\|U^R\|_{L_{T_0}^{\infty}L^2}+\nu^{\f12}\|\na U^R \|_{L_{T_0}^{2}L^2}\lesssim \eta \xlongrightarrow{\ep,\nu\rightarrow 0} 0
\eeqs
    where $\eta= \ep+(\ep\nu)^{\f14}+\nu^{\f34}.$ 
\end{thm}
\begin{rmk}
We remark that the constraint we imposed between $\nu$ and $\epsilon,$ specifically $\nu=\epsilon^{\kappa}$, where $\kappa\in(0,3)$, serves to establish uniform estimates for the remainder, directly related to the error estimate for $R^a-$the source term comprising the approximate solution. As we progress to higher orders in constructing the approximate solution, this correlation becomes progressively less constraining. It's worth noting that the methodology outlined in Theorem \ref{thm-appsol} for constructing the approximate solution is applicable across any values of $\nu$ and $\epsilon$. However, in cases where $\kappa \in [3,+\infty)$, where the vanishing viscosity limit prevails, an alternative approach may be more feasible. In this scenario, one could expand the original system \eqref{CNS-O} in terms of $\nu$ and analyze the stability of the remainder. Finally, one justifies the incompressible limit for the compressible Navier-Stokes system with only horizontal dissipation. We will give more details in Section \ref{sec-remark-othercase}.
\end{rmk}
\begin{rmk}
In addition to addressing other technical issues, we have made significant efforts to minimize the regularity requirements imposed on 
$U^R$ in order to establish the uniform well-posedness of the remainder equation \eqref{eq-UR}. Specifically, our arguments allow us to involve only two time derivatives, necessitating just two compatibility conditions \eqref{comp-cond}, which is quite less restrictive than those in previous works \cite{MFS-JMPA, MFS-free, sun2022uniform} where five compatibility conditions are assumed. 
\end{rmk}

 Let us observe that, due to the presence of the source term $R^a$ and the linear terms $u^a \cdot \nabla U^R$ and $U^R\cdot\nabla U^a,$ the time derivative of the remainder $\partial_t U^R$ is not uniformly bounded initially, so that we could not in general expect the uniform control of $L_t^{\infty}L^2$ norm of $\pt u^R$ and thus the $L_t^{\infty}H^2$ norm of $U^R.$
   Consequently,  to establish uniform well-posedness of the remainder, one still needs to work in conormal spaces \eqref{def-conormal} that account for the effects of boundary layers.

\textbf{Organization of the paper:} 
We construct the approximate solution in Section 2 and prove
Theorem \ref{thm-appsol}, based on the estimates 
for the approximate solutions that are presented in the appendix. 
We study the stability of the remainder in Sections 3-5 and prove Theorem \ref{thm-remainder}. Finally,
in the appendix, we compute the filtered interior profile and obtain various estimates on the profiles included in the the approximate solutions, which are needed in the proof of Theorem \ref{thm-appsol}.

\section{Construction of the approximate solution}
  In this section, we aim to construct the approximate solution in the form of \eqref{expansion} by
combining the filtering method \cite{MR1303036} and multi-scale analysis. To deal with the fast oscillations in time, we use the group associated to the singular linear operator $L$ 
 with suitable boundary conditions
 to 
 filter out the high-oscillating acoustic waves, which corresponds to the first interior profile. 
We then construct (an infinite collection of) oscillating boundary layers to address the inconsistency in the boundary conditions for the horizontal components of velocities between the real solution and the initial interior profile. The kernel of the penalized  operator 
$L$ corresponds to the solution to the incompressible Navier-Stokes equation with horizontal dissipation. However, we have to use 
the slightly compressible Prandtl layer (see \eqref{prandtl-bo}, \eqref{prandtl-to}) to compensate for the horizontal boundary conditions of the  incompressible component of the velocity, in order to take into account the interactions of internal compressible part and the Prandtl type layers. 

Here, we introduce a new function space which involves only two-order time derivative:
\begin{align}\label{def-2td}
    L_t^p {\overline{H}}_{co}^{m}=\big\{f\,
\big|\,Z^{\alpha}f\in L^{p}\big([0,t],L^2(\Omega)\big), \forall \,\alpha\in\mathbb{N}^4,\,|\alpha|\leq m, \alpha_0\leq 2 \big\} , (\forall\, m\geq 2).
\end{align}

We note that this space is introduced to involve as few time derivatives as possible. While this is not particularly important in this section, it becomes crucial for ensuring the stability of the remainder in the next section. By minimizing the number of time derivatives required, we can reduce the compatibility conditions imposed on the initial data of the remainder.


\subsection{First two interior equations}
    In this subsection, we aim to recover the first interior equation $U^{I,0}=\colon U_{osc}^{I,0}+\left( \begin{array}{c}
   0   \\
   v^{\INS}
\end{array}\right).$  

Plugging the ansatz \eqref{expansion} into \eqref{cns-abs} and looking at the $\ep^{-1}$ 
order of the interior equation, we find
\beq\label{eq-int0} 
(\p_\tau-L) U^{I,0}=0, \quad U^{I,0}|_{\tau=0}= W (t,x), \quad u_3^{I,0}|_{\p\Omega}=0,
\eeq
where we denote $\tau=t/\ep.$ To solve \eqref{eq-int0}, it suffices to find the initial condition $W(t,x)$ 
which could be achieved by imposing the sublinear condition on the second interior equation $U_{osc}^{I,1}.$ 

Before detailing this, we need to find the equation satisfied by the $U_{osc}^{I,1}.$ The $\ep^0$ order of the interior equation reads
\beq\label{second-inter}
(\p_\tau-L) U_{osc}^{I,1}+ \big(\pt U^{I,0}+ Q(U^{I,0}, U^{I,0})-\mu_1\Delta_h U^{I,0}\big)=0. \quad 
\eeq
Moreover, one interior profile $S_{osc}^{I,1}$ needs to be introduced to make up the contributions of the boundary layer 
$\cB_3^1$ and $\cT_3^1$ (introduced later) on the boundaries.  
More precisely, we will find $S_{osc}^{I,1}$ in  a way such that \footnote{The simplest way would be choosing $(\p_\tau-L) S_{osc}^{I,1}=0,$ however, by Laplace-Fourier transform, $S_{osc}^{I,1}$ cannot be bounded uniformly in $L_{t,x}^2$ and $L_{t,x}^{\infty}$ which gives some trouble for the analysis of the remainder. Moreover, by constructing $S_{osc}^{I,1}$ in the way of \eqref{artificial-bd}, we make one damping term appear in the filtered oscillating part which is the key to prove the strong convergence when $(\ep/\nu)\xrightarrow{\ep\rightarrow 0} 0.$  }
\beq\label{artificial-bd}
(\p_\tau-L) S_{osc}^{I,1}=-\mathcal{L}(\tau)\overline{S}, \quad S_{osc,3}^{I,1}|_{z=0}=-\cB_3^1|_{z=0}, \quad S_{osc,3}^{I,1}|_{z=a_3}=-\cT_3^1|_{z=a_3}.
\eeq
Hereafter, we denote $\cL(\tau)$ as the semigroup generated by $L$ in the space $V_{sym}^{s+3}(\Omega)$(see \eqref{def-vsym}),  incorporating the boundary condition that the last component of the acted element vanishes on the boundary.
By the explicit construction performed later,  one can ensure that $S_{osc}^{I,1}$ is uniformly bounded in $L^{\infty}_{\tau}\big(L^{\infty}([0,T],{H}^{s})\big)$ as long as $W\in L^{\infty}([0,T],H^{s})$ for any $s\geq 0.$

Let $V=U_{osc}^{I,1}+\sqrt{\f{\nu}{\ep}}S_{osc}^{I,1}.$ 
By \eqref{second-inter} and \eqref{artificial-bd}, $V$ can be solved as
\beqs 
V=
\cL(\tau)V|_{\tau=0}-\cL(\tau)\int_0^{\tau} \cL(-\tau')\bigg(\pt U^{I,0}+ Q(U^{I,0}, U^{I,0})-\left(\begin{array}{cc}
   0   & 0 \\
    0 & \mu_1\Delta_h+\mu_2\na\div 
\end{array}\right)U^{I,0}\bigg)+\sqrt{\f{\nu}{\ep}}\overline{S} \,\,\d \tau'.
\eeqs
To ensure that the expansion \eqref{expansion} makes sense, we need $\ep 
V $ (and thus $\ep \cL(-\tau) V$) tends to 0 as $\ep$ tends to 0 in $L^2(\Omega)$  for any $t>0,$ this is equivalent to imposing the following sublinear condition
\beq
\lim_{\tau\rightarrow +\infty}\f{1}{\tau} \int_0^{\tau} \cL(-\tau')\bigg(\pt U^{I,0}+ Q(U^{I,0}, U^{I,0})-\left(\begin{array}{cc}
   0   & 0 \\
    0 & \mu_1\Delta_h +\mu_2\na\div 
\end{array}\right) U^{I,0}\bigg)(\tau')+\sqrt{\f{\nu}{\ep}}\overline{S} \,\,\d \tau'=0.
\eeq
Recalling the formula $U^{I,0}=\cL(\tau)W,$ we find that the filtered profile $W$ solves 
\beq\label{eq-W-1}
 \pt W + \overline{Q}(W,W) -\overline{\Delta} W 
+ \sqrt{\f{\nu}{\ep}}\overline{S}=0
\eeq
where 
 \beq\label{res-dissipation}
\overline{\Delta} W=\lim_{\tau\rightarrow +\infty}\f{1}{\tau} \int_0^{\tau} \cL(-\tau') \left(\begin{array}{cc}
   0   & 0 \\
    0 & {\mu_1}\Delta_h +\mu_2\na\div 
\end{array}\right) \cL(\tau')W \,\d \tau' \, ,
\eeq
and the quadratic term
\beq \label{res-nonlinear}
\overline{Q}(W,W)= \lim_{\tau\rightarrow +\infty}\f{1}{\tau} \int_0^{\tau} \cL(-\tau') Q\big(\cL(\tau')W, \cL(\tau')W\big) \,\d \tau'
\eeq
is the mean value of the almost periodic (in $\tau$) function $\cL(-\tau')Q(U^{I,0}, U^{I,0})(\tau, t, x)$ valued in $L^{\infty}([0,T],H^s(\Omega)).$ We refer to \big[Section 2, \cite{MR1733696}\big] for some basic properties of almost periodic functions. 
The derivations of $\overline{\Delta} W$ and 
$\overline{Q}(W,W)$ will be presented the Appendix \ref{app-B}.

Inserting \eqref{artificial-bd}, \eqref{eq-W-1} into \eqref{second-inter}, 
we find that 
\beqs 
(\p_\tau-L)V+ Q(U^{I,0},U^{I,0})-\cL(\tau)\overline{Q}(W,W)-\left(\begin{array}{cc}
   0   & 0 \\
    0 & \mu_1\Delta_h 
\end{array}\right) U^{I,0} + \f{\mu_1}{2}\cL(\tau)\Delta_h W=0.
\eeqs
Assuming for simplicity $V|_{\tau=0}=0,$ then 
\beq\label{def-V}
\begin{aligned}
V(\tau,t,x)=\cL(\tau)\int_0^{\tau}&\bigg(-\cL(-\tau') Q\big(
U^{I,0},U^{I,0}\big)
+\overline{Q}(W,W)\bigg)\\
& +\bigg(\cL(-\tau')\left(\begin{array}{cc}
   0   & 0 \\
    0 &  {\mu_1}\Delta_h +\mu_2\na\div
\end{array}\right) U^{I,0} -\overline{\Delta} W 
\bigg)\, \d \tau'. 
\end{aligned}
\eeq
As will be shown afterwards $U^{I,0}$ belongs to $L^{\infty}_{\tau}\big(L^{\infty}([0,T],H^{s})\big)$ 
\beqs 
\cL(-\tau) Q\big(U^{I,0},U^{I,0}\big) \in L^{\infty}_{\tau}\big(L^{\infty}([0,T],H^{s-2})\big). 
\eeqs
It then follows from \big[Lemma 2.4, \cite{MR1733696}\big] that
$\big\|\ep V\big(\f{t}{\ep}, t,\cdot\big)\big\|_{H_{co}^{s-2}}\rightarrow 0.$
Moreover, if some suitable small divisor estimates hold, the decay rates in the above convergence could be made quantitative.
To be precise, if the domain is non-resonant, that is, there exists $r_0,$ such that for any $\alpha,\beta,\gamma\in\{+,-\},$  and eigenvalues $\lambda_{k}^{\pm}=\pm i|k|=\pm i|(\f{2\pi}{a_1}k'_1, \f{2\pi}{a_2}k'_2,  \f{\pi}{a_3}k_3')|,\, (k'\in \mathbb{Z}^3),$ it holds that 
\beq\label{no-resonantasp} 
\begin{aligned}
 \text{ either }  \quad \omega_{klm}^{\alpha,\beta,\gamma}\colon=\lambda_{m}^{\alpha}-  \lambda_{k}^{\beta}-\lambda_{l}^{\gamma}= 0, \\
  \text{ or }  \quad |\omega_{klm}^{\alpha,\beta,\gamma}|\gtrsim (1+|k|)^{-r_0}(1+|l|)^{-r_0},
\end{aligned}
\eeq
then one could have $\big\| \ep V\big(\f{t}{\ep}, t,\cdot\big)\big\|_{H_{co}^{s-2-r_0}}\lesssim \ep.$  This is proven in Lemma \ref{lem-V}. To be precise,  Lemma \ref{lem-V} gives that $\big\| \ep V\big(\f{t}{\ep}, t,\cdot\big)\big\|_{H_{co}^{s-2-r_0}}\lesssim \ep.$

Thanks to the simple geometry of $\Omega,$ the semigroup $\cL(\tau)$ can be expressed explicitly, by using the eigenfunctions 
$(N_{k}^{\alpha})_{k\in \mathbb{Z}^3 
,\alpha\in \{+,-\}}$ of $L$ 
which form a family of complete sets in $L^2(\Omega)\times \{\na q\big| q\in H^1, \p_z q|_{\p\Omega}=0\}.$ 
We leave these elementary computations to the Appendix \ref{app-A}.
Denote 
$\mathbb{Q}=\left( \begin{array}{cc}
 \Id  & 0 \\
  0 & -\na (-\Delta_N)^{-1}\div \end{array}\right)$ the projection of a four-elements vector field to its compressible part. 
Then by letting $$W(t)=\mathbb{Q}W+ (\Id - \mathbb{Q})W=\sum_{k'\in\mathbb{Z}^3}
\sum_{\alpha\in\{+,-\}} b_k^{\alpha}(t)N_{k}^{\alpha}+\left( \begin{array}{c}
   0   \\
   v^{\INS} \end{array}\right),$$
  (we refer to \eqref{rela-k'-k} for the relation between $k'$ and $k$), 
the equation of $W$ is thus equivalent to infinite collections of coupled equations for $b_k^{\alpha}(t),$ which turn out to be damped Burgers equations (see $(\text{DB})_m^{\alpha}$ in Appendix B.2).
Moreover, the first oscillating interior profile 
\beq\label{form-UI0}
U_{osc}^{I,0}:=\cL\bigg(\f{t}{\ep}\bigg) \mathbb{Q}W=\sum_{k'\in\mathbb{Z}^3}
\sum_{\alpha\in\{+,-\}} b_k^{\alpha}(t)N_{k}^{\alpha} e^{i\alpha|k|\tau}.
\eeq
In the following, we will assume that $W(0,\cdot)\in H^{s}$ with $s\geq 11+r_0,$ that is, 
$$\sum_{k'\in \mathbb{Z}^3}\sum_{\alpha\in\{+,-\}}|b_{k}^{\alpha}(0)|^2 (1+|k|^2)^{s}<+\infty.$$ 
Then by Lemma \ref{well-posedness-W} and Lemma \ref{lem-UI0}, there exists $T_2>0$ which is independent of $\ep,\nu$ such that 
\begin{align*}
   \sup_{0\leq t\leq T_2} \|U^{I,0}\big(\f{t}{\ep},t,\cdot\big)\|_{H^s}\lesssim \|W(0)\|_{H^s}, \quad \sup_{0\leq t\leq T_2} \|\ep V\big(\f{t}{\ep},t,\cdot\big)\|_{H^{s-r_0-2}}\lesssim \ep (1+T_2) \big(\|W(0)\|_{H^s}^2+\|W(0)\|_{H^s}^3\big).
\end{align*}

\subsection{Corrections for the fast variable}
In this subsection, we find boundary layer corrections for the interior part with fast oscillations specifically
$U_{osc}^{I,0}.$ 
In the process, the profile $\sqrt{\ep\nu}S_{osc}^{I,1}$ will also be determined. 
\subsubsection{The first boundary layers}
Let us focus only on the boundary layer corrections around the bottom, those for the upper boundary are almost identical, just replacing $z$ by $a_3-z.$

Plugging the ansatz \eqref{expansion} into \eqref{cns-abs} and match the $\ep^{-1}$ order of the boundary layer profile with the size $\sqrt{\ep\nu},$ we find that $\cB_h^0$ solves the heat equation
\beq\label{eq-bl0} 
\big(\pt -\nu\mu_1 \p_z^2\big)\cB_h^0 = 0.
\eeq
To make up the boundary condition of the horizontal velocity component of the first oscillating interior profile  $U_{osc}^{I,0},$ we need to impose
\beq\label{bl0-bd}
\cB_h^0|_{z=0}=-U_{osc,h}^{I,0}|_{z=0}.
\eeq
  For the choice of the initial condition of $\cB_h^0,$  we have some freedom. The simplest choice would be  the zero initial condition. However, if we do so, $\cB_h^0$ could be solved explicitly by using the Green function for the heat equation:
\begin{align*}
  \cB_h^0&= -2\nu\mu_1 \int_0^{t}\f{1}{(4\pi\nu\mu_1(t-t'))^{\f{1}{2}}}\p_{z}\big(e^{-\f{z^2}{4\pi\nu\mu_1(t-s)}}\big)U_{osc}^{I,0}|_{z=0}\big(\f{t'}{\ep},t', x_h\big)\, \d t'\,.
\end{align*}
It is direct to verify that
 \beqs 
\|\cB_h^0\|_{L_{T_2}^{\infty}L_{x}^2}\lesssim \nu^{\f{1}{4}}\, ,
  \eeqs
which means more or less that the size of the boundary layer degenerates to $\sqrt{\nu},$ and would not be sufficient for us to find some damping mechanism of the interior profile $U_{osc}^{I,0}$ when $\nu \geq \ep\,.$
Nevertheless, thanks to the explicit formula \eqref{form-UI0}, we can  construct $\cB_h^0$ in the following way 
\beq \label{def-cB_h^0}
\cB_h^0 : = -\sum_{k'\in \mathbb{Z}^3} \sum_{\alpha\in\{+,-\}} \f{c_{*}}{2}\f{k_h}{|k|}e^{ik_h\cdot x_h} b_k^{\alpha}(t)  e^{i\alpha |k|\tau}e^{-\f{z}{\sqrt{\ep\nu}}\sqrt{\f{|k|}{2\mu_1}}(1+\alpha i)}
\eeq
where $\tau=\f{t}{\ep}, c_{*}=\f{1}{\sqrt{a_1a_2a_3}}.$
It satisfies the heat equation \eqref{eq-bl0} up to a small remainder
\beq\label{Bh0-exact}
\big(\pt -\nu\mu_1\p_z^2\big)\cB_h^0={\cB'}_h^0
\eeq
where ${\cB'}_h^0$ is defined by replacing $b_k^{\alpha}$ in the definition of $\cB_h^0$ by $\pt b_k^{\alpha}:$
\beq \label{defcB'h}
{\cB'}_h^0:= -\sum_{k'\in \mathbb{Z}^3} \sum_{\alpha\in\{+,-\}} \f{c_{*}}{2} \f{k_h}{|k|} \pt b_k^{\alpha}(t) e^{i\alpha |k|\tau}e^{-\f{z}{\sqrt{\ep\nu}}\sqrt{\f{|k|}{2\mu_1}}(1+\alpha i)}.
\eeq
  Finally, $\cB_{3}^1$ is searched by the divergence-free condition owing to the fact $\sigma^{B,0}=0:$
  \beq\label{bo-bl-3}
\div_h \cB_h^0+ \p_{\zeta} \cB_3^1 =0, \quad \cB_3^1 \xlongrightarrow{\zeta\rightarrow+\infty} 0 , \quad \big(\zeta=\f{z}{\sqrt{\ep\nu}}\big).
  \eeq
  Consequently, 
\beqs
\cB_{3}^1|_{\zeta=0}=\int_{0}^{+\infty} \div_h \cB_h^0 \,\d\zeta=-\sqrt{\f{\mu_1}{2}}\f{c_{*}}{2}\sum_{k'\in \mathbb{Z}^3}\sum_{\alpha\in\{+,-\}}   \f{|k_h|^2}{|k|^{\f{3}{2}}} b_k^{\alpha}(t) \alpha(\alpha i+1) e^{ik_h\cdot x_h} e^{i\alpha |k|\tau}. 
\eeqs
It follows from \eqref{prop-bhp0}  and \eqref{cBh-Linfty}  that 
\beq\label{es-bh'0}
\|{\cB'}_h^0\|_{L_{T_2}^1\oH_{co}^{s-5}}+ (\ep\nu)^{\f{1}{2}} \|\p_z{\cB'}_h^0\|_{L_{T_2}^1\oH_{co}^{s-6}}\lesssim (\ep\nu)^{\f{1}{4}}(\|W(0)\|_{H^s}+\|W(0)\|_{H^s}^2),
\eeq
\beq\label{prop-bh0}
\begin{aligned}
\big\|\cB_h^0
&\|_{L_{T_2}^{\infty}\oH_{co}^{s-9/4}}+\big\|\cB_3^1
\|_{L_{T_2}^{\infty}\oH_{co}^{s-11/4}}+(\ep\nu)^{\f{1}{2}}\big(\big\|\p_z \cB_h^0
\|_{L_{T_2}^{\infty}H_{co}^{s-13/4}}+\big\|\p_z \cB_3^1
\|_{L_{T_2}^{\infty}H_{co}^{s-15/4}}\big)\\
&\lesssim (\ep\nu)^{\f{1}{4}}\|W(0)\|_{H^s}.
\end{aligned}
\eeq
In a similar manner, we can construct the boundary layer associated to the upper boundary as
\begin{align*}
\cT_h^0= -\sum_{k'\in \mathbb{Z}^3} \sum_{\alpha\in\{+,-\}} (-1)^{|k'_3|} \f{c_{*}}{2} \f{k_h}{|k|}e^{ik_h\cdot x_h}b_k^{\alpha}(t)e^{i\alpha |k|\tau}e^{-\f{a_3-z}{\sqrt{\ep\nu}}\sqrt{\f{|k|}{2\mu_1}}(1+\alpha i)}
\end{align*}
and 
\beqs
\cT_{3}^1|_{\zeta'=0}=
-\sqrt{\f{\mu_1}{2}}\f{c_{*}}{2} \sum_{k'\in \mathbb{Z}^3}\sum_{\alpha\in\{+,-\}} (-1)^{|k'_3|+1} \f{|k_h|^2}{|k|^{\f{3}{2}}} b_k^{\alpha}(t) \alpha(\alpha i+1) e^{ik_h\cdot x_h} e^{i\alpha |k|\tau},
\eeqs
which satisfies the relation
\beq\label{up-bl-3}
\div_h \cT_h^0- \p_{\zeta'} \cT_3^1 =0, \quad \cT_3^1 \xlongrightarrow{\zeta'\rightarrow+\infty} 0 , \quad \big(\zeta'=\f{a_3-z}{\sqrt{\ep\nu}}\big).
\eeq
Moreover, they satisfy the same estimates as in \eqref{prop-bh0} and \eqref{es-bh'0}. 

\subsubsection{Determine the damping term $\overline{S}$}
As $\cB_{3}^1$ and $\cT_{3}^1$ do not vanish on the boundaries, we need to use an interior profile  of size $\sqrt{\ep\nu}$ to compensate them. To be precise, we would like to find $S_{osc}^{I,1}$ and  $\overline{S}$ such that \eqref{artificial-bd} holds. It turns out to be equivalent to studying the spectrum of the linearized operator 
\beq\label{defLepnu}
\tilde{L}^{\ep,\nu}=-\left( \begin{array}{cc}
    0 &  \div \\
    \na  & 0 
    \end{array} 
    \right)+\left(\begin{array}{cc}
        0 & 0 \\
        0 & 
      \ep \nu  \mu_1  \p_z^2  
    \end{array}\right) 
\eeq
in the space $L^2(\Omega)\times \{\na q\big| q\in H^1, \p_z q|_{\p\Omega}=0\}.$ 
Let $$\lambda_{k}^{\ep,\nu,\pm}=\pm i |k|+\sqrt{\ep\nu} \lambda_{k,1}^{\pm}+\ep\nu \lambda_{k,2}^{\pm}+\cdots $$
and 
\beqs 
\Psi_{k}^{\ep,\nu,\pm}= (\Psi_{k,0}^{I,\pm}+\Psi_{k,0}^{B,\pm})+ \sqrt{\ep\nu}  (\Psi_{k,1}^{I,\pm}+\Psi_{k,1}^{B,\pm})+\cdots
\eeqs
be the eigenpairs of operator $\tilde{L}^{\ep,\nu},$ where 
$(\pm i |k|, \Psi_{k,0}^{I,\pm}=N_k^{\pm})$ are the eigenpairs of $L.$ The first boundary layer correction $\Psi_{k,0}^{B,\pm}$ can be found in the same way as done in the previous subsection  
\begin{align*}
\Psi_{k,0}^{B,\pm}&= \left( \begin{array}{c}
     0  \\
   \f{c_{*}}{2} \f{k_h}{|k|} e^{ik_h\cdot x_h}e^{-\f{z}{\sqrt{\ep\nu}}\sqrt{\f{|k|}{2\mu_1}}(1\pm i)}   \\
     0 
\end{array} \right)+ \left( \begin{array}{c}
     0  \\
    (-1)^{|k'_3|} \f{c_{*}}{2}\f{k_h}{|k|} e^{ik_h\cdot x_h}e^{-\f{a_3-z}{\sqrt{\ep\nu}}\sqrt{\f{|k|}{2\mu_1}}(1\pm i)}   \\
     0 
\end{array} \right).  
\end{align*} 
Moreover, the last component of $\Psi_{k,1}^{B,\pm},$ denoted as $\underline{\Psi}_{k,1,3}^{B,\pm}+\overline{\Psi}_{k,1,3}^{B,\pm}$ is recovered by the divergence free condition similar to \eqref{bo-bl-3}, \eqref{up-bl-3}. We only need the boundary conditions
\beqs 
\underline{\Psi}_{k,1,3}^{B,\pm}|_{z=0}= -\sqrt{\f{\mu_1}{2}}\f{c_{*}}{2}  \f{|k_h|^2}{|k|^{\f{3}{2}}}  (i\pm 1) e^{ik_h\cdot x_h}, \quad \overline{\Psi}_{k,1,3}^{B,\pm}|_{z=a_3}=(-1)^{|k'_3|+1}\underline{\Psi}_{k,1,3}^{B,\pm}|_{z=0}\, .
\eeqs
By matching the $\sqrt{\ep\nu}$ order of the interior part of the identity 
\beqs 
L^{\ep,\nu} \Psi_{k}^{\ep,\nu,\pm}=\lambda_{k}^{\ep,\nu,\pm} \Psi_{k}^{\ep,\nu,\pm}, 
\eeqs
one finds that $\Psi^{I,\pm}_{k,j}=\big(\sigma^{I,\pm}_{k,j}, u^{I,\pm}_{k,j}=\na\phi^{I,\pm}_{k,j}\big)$ solves
\beq\label{sec-int-spec-0}  
\left\{ \begin{array}{c}
  -\div  u^{I,\pm}_{k,1}= \pm i|k| \sigma_{k,1}^{I,\pm} + \lambda_{k,1}^{\pm} \sigma_{k,0}^{I,\pm},\\[3pt]
   -\na \sigma_{k,1}^{I,\pm} = \pm i|k|  u^{I,\pm}_{k,1}  + \lambda_{k,1}^{\pm} u_{k,0}^{I,\pm},
\end{array} \right.
\eeq
or equivalently, 
\beq\label{sec-int-spec} 
\left\{ \begin{array}{c}
  -\Delta\phi_{k,1}^{I,\pm} = |k|^2  \phi^{I,\pm}_{k,1}  \mp 2i|k|\lambda_{k,1}^{\pm} \phi_{k,0}^{I,\pm},  \\[3pt]
  -\Delta \sigma^{I,\pm}_{k,1}=|k|^2 \sigma_{k,1}^{I,\pm} \mp 2i|k|\lambda_{k,1}^{\pm} \sigma_{k,0}^{I,\pm}.
\end{array} \right.
\eeq
The boundary conditions of $\p_z\phi_{k,1}^{I,\pm}$ need to be imposed to cancel $\underline{\Psi}_{k,1,3}^{B,\pm}|_{z=0}$ and 
$\overline{\Psi}_{k,1,3}^{B,\pm}|_{z=a_3}:$
\beq \label{bdc}
\p_z\phi_{k,1}^{I,\pm}|_{z=0}=-\underline{\Psi}_{k,1,3}^{B,\pm}|_{z=0}, \quad  \p_z\phi_{k,1}^{I,\pm}|_{z=a_3}=-\overline{\Psi}_{k,1,3}^{B,\pm}|_{z=a_3}=(-1)^{|k'_3|+1} \p_z\phi_{k,1}^{I,\pm}|_{z=0}\, . 
\eeq 
Taking the $L^2(\Omega)$ 
inner product of \eqref{sec-int-spec} with $\phi_{k,0}^{I,\pm}=-\f{c_{*}i}{ 2
|k|}\cos (k_3 z)e^{ik_h\cdot x_h},$ we obtain after a few computations that
\beq\label{deflambdak1}
\lambda_{k,1}^{\pm}=-\f{2(1\pm i)}{a_3}\sqrt{\f{\mu_1}{2}}  \f{|k_h|^2}{|k|^{{3}/{2}}}.  %
\eeq
Note that the real part of $\lambda_{k,1}^{\pm}$ is negative once $|k_h|\neq 0,$ this damping mechanism is essential in order to show the strong convergence of the compressible part of the velocity.  
Inserting the expression of  $\lambda_{k,1}^{\pm}$ into  \eqref{sec-int-spec}, we find that $f_k^{\pm}=e^{-ik_h\cdot x_h} \phi_{k,1}^{\pm}$ solves the ODE
\beqs 
\left\{ \begin{array}{l}
  (\p_z^2+k_3^2)f_{k}^{\pm}=\pm  c_{*}\lambda_{k,1}^{\pm}\cos(k_3 z), \\[3pt]
  \p_z f_{k}^{\pm}|_{z=0} = \mp\f{a_3}{2} \f{c_{*}}{2}\lambda_{k,1}^{\pm} ,\\ 
  \p_z f_{k}^{\pm}|_{z=a_3}=(-1)^{|k'_3|+1} \p_z f_{k}^{\pm}|_{z=0} ,
\end{array} \right.
\eeqs
and is given by
\beqs 
f_{k}^{\pm}=\f{\pm c_{*}\lambda_{k,1}^{\pm}}{2k_3} \sin(k_3 z) \big(z-\f{a_3}{2}\big). 
\eeqs
It is straightforward to verify that 
\beq\label{psi1-I}
\begin{aligned}
\Psi_{k,1}^{I,\pm}&= \left( \begin{array}{c}
     \mp i|k|\phi_{k,1}^{I,\pm}- \lambda_{k,0}^{I,\pm}\phi_{k,1}^{I,\pm} \\
    \na \phi_{k,1}^{I,\pm}
\end{array} \right)
\\
&= \left( \begin{array}{c}
  \f{c_{*} }{|k|}i\lambda_{k,1}^{\pm}  e^{ik_h\cdot x_h}\big(\cos(k_3 z)-\f{|k|^2}{k_3}\sin(k_3 z)\big(z-\f{a_3}{2}\big)\big)\\
     i k_h f_k^{\pm} e^{ik_h\cdot x_h}  \\
      \p_z f_k^{\pm} e^{ik_h\cdot x_h} 
\end{array} \right).
\end{aligned} 
\eeq
Define now 
\beq \label{defS-barS}
S_{osc}^{I,1}=\sum_{k'\in \mathbb{Z}^3} \sum_{\alpha\in\{+,-\}}  b_k^{\alpha}(t) \Psi_{k,1}^{I,\pm}e^{i\alpha|k|\tau} , \quad \overline{S}= \sum_{k'\in \mathbb{Z}^3}\sum_{\alpha\in\{+,-\}}  (-\lambda_{k,1}^{\alpha})b_k^{\alpha}(t)
 N_k^{\alpha} , 
\eeq
it then stems from \eqref{sec-int-spec-0} that 
\beqs 
(\p_{\tau} -L)  S_{osc}^{I,1}= \sum_{k'\in \mathbb{Z}^3}\sum_{\alpha\in\{+,-\}}  \lambda_{k,1}^{\alpha}b_k^{\alpha}(t) 
 N_k^{\alpha} e^{i\alpha|k|\tau} =-\cL(\tau) \overline{S}. 
\eeqs
By \eqref{SoscI1}, it holds that 
 \beq\label{SoscI1-main}
\|S_{osc}^{I,1}\|_{L_{T_2}^{2}\oH_{co}^{s-\f52}}\lesssim   \|S_{osc}^{I,1}\|_{L_{T_2}^{2}H^{s-\f12}}\lesssim \|U_{osc}^{I,0}\|_{L_{T_2}^{2}H^{s}}\lesssim \|W(0)\|_{H^s}.
\eeq
Moreover, it follows from \eqref{bdc} and the expression \eqref{psi1-I} that
\beqs 
S_{osc,3}^{I,1}|_{z=0}=-\cB_3^1|_{z=0}, \quad S_{osc,3}^{I,1}|_{z=a_3}=-\cT_3^1|_{z=a_3}, \quad S_{osc,1,2}^{I,1}|_{z=0, z=a_3}=0 ,
\eeqs
and thus in particular \eqref{artificial-bd} holds true.
Let us remark that since the tangential part of the velocity $S_{osc,1,2}^{I,1}$ vanish on the boundaries, it is not necessary to introduce the boundary layer of  $S_{osc}^{B,1}$ to correct.

\subsection{Corrections for the mean flow--slow variable}
\subsubsection{Prandtl layer correction}
As studied in Subsection 2.1, the incompressible part of the first interior equation $U^{I,0}$ is the solution to the incompressible Navier-Stokes equations with only tangential dissipation: 
\beqs 
\pt v^{\INS}+ \mathbb{P}\big(v^{\INS}\cdot\na v^{\INS}-\mu_1\Delta_h v^{\INS}\big) =0, \quad \div \,v^{\INS}=0, \quad v_3^{\INS}|_{\p\Omega}=0\,. 
\eeqs
Note that the local well-posedness of the above system in the usual Sobolev spaces $H^s(\Omega)$
is standard. 

To compensate the non-vanishing boundary condition for the horizontal variable, it is natural to introduce the  Prandtl layers $v_h^p(t,x_h,\theta)=v_h^p(t,x_h,\f{z}{\sqrt{\nu}})$ and 
$\tilde{v}_h^p(t,x_h,\theta')=\tilde{v}_h^p(t,x_h,\f{a_3-z}{\sqrt{\nu}})$ which solve
\beq\label{prandtl-bo}
\left\{
\begin{array}{l}
    \pt v_h^{p}+(\underline{u_h^{I,0}}+v_h^p)\na_h v_h^p+v_h^p\cdot\na_h \underline{u_h^{I,0}}+\bigg( -
     v_3^{p,1}|_{\theta=0}+v_3^{p,1}+\theta \underline{\p_{z} u_3^{I,0}}\bigg)\p_{\theta}v_h^p-\mu_1(\Delta_h+\p_{\theta}^2)v_h^p=0,\\
   v_h^{p,1}|_{\theta=0}=- \underline{v_h^{\INS}}, \quad  -\big(\frac{\gamma-1}{2} \ep \underline{\sigma_{osc}^{I,0}}+1\big)\p_\theta v_3^{p,1}=\ep v_h^p\cdot\na_h \underline{\sigma^{I,0}_{osc}}+\big(\frac{\gamma-1}{2} \ep \underline{\sigma_{osc}^{I,0}}+1\big)\div_h v_h^p, 
\end{array}
\right.
\eeq
and 
\beq\label{prandtl-to}
\left\{
\begin{array}{l}
    \pt \tilde{v}_h^{p}+(\overline{{u}_h^{I,0}}+v_h^p)\cdot\na_h \tilde{v}_h^p+\tilde{v}_h^p\cdot\na_h \overline{{u}_h^{I,0}}+\bigg( 
    -\tilde{v}_3^{p,1}|_{\theta'=0}+\tilde{v}_3^{p,1}+\theta' \overline{
    \p_z u_3^{I,0}}\bigg)\p_{\theta}\tilde{v}_h^p-\mu_1(\Delta_h+\p_{\theta}^2)\tilde{v}_h^p=0,\\
 \tilde{v}_h^{p,1}|_{\theta'=0}= -\overline{\tilde{v}_h^{\INS}}, \quad
\big(\frac{\gamma-1}{2} \ep \overline{\sigma_{osc}^{I,0}}+1\big)\p_{\theta'} \tilde{v}_3^{p,1}=\ep \tilde{v}^p\cdot\na_h \overline{\sigma^{I,0}_{osc}}+\big(\frac{\gamma-1}{2} \ep \overline{\sigma_{osc}^{I,0}}+1\big)\div_h \tilde{v}_h^p,
\end{array}
\right.
\eeq
where we denote $\underline{f}$ and $\overline{f}$ the trace of a function $f$ on $\{z=0\}$ and  $\{z=a_3\}.$ 
Note also we write $v^p=v^{p,0}$ sometimes for notational convenience. 

Note that in the above equations, we have taken into account the interactions between the first order interior oscillating part and the boundary layer $v_h^p,$ through the term 
$\underline{u_h^{I,0}}\cdot \na v_h^p=(\underline{u_{osc}^{I,0}}+\underline{v^{\INS,0}})_h\cdot\na_h v^p$ in the equation as well as the relation
 $$\big(\frac{\gamma-1}{2} \ep \underline{\sigma_{osc}^{I,0}}+1\big)^{\f{\gamma-1}{2}}\p_\theta v_3^{p,1}=-\ep v^p\p_h \underline{\sigma^{I,0}_{osc}}-\big(\frac{\gamma-1}{2} \ep \underline{\sigma_{osc}^{I,0}}+1\big)^{\f{\gamma-1}{2}}\div_h v_h^p.$$ 
 Let us notice that the above is an  `almost divergence-free' relation in the sense that $$v^{p,1}_3|_{\theta=0}=\int_0^{+\infty}\div_h v_h^p \,\d \theta+\cO(\ep).$$ 
Moreover, due to the presence of both horizontal and vertical dissipations, the local well-posedness of the above systems in the  
weighted (in $\theta$) Sobolev conormal spaces is not an issue.  See Proposition \ref{prop-well-vhp}.


\subsubsection{Second interior profile.}
In general, the third component of the boundary layer $\sqrt{\nu} v_3^{p,1}$ and 
$\sqrt{\nu}\tilde{v}_3^{p,1}$ does not vanish on the boundaries, we need to lift them up. 
We thus introduce the unknown $\sqrt{\nu}V^{I,1}=\sqrt{\nu}(\chi^{I,1}, v^{I,1})^t$ 
which satisfies the following equation:
\beq\label{eq-ve1} 
\left\{
\begin{array}{l}
(\pt-L/\ep) V^{I,1}+
Q\big(
V^{I,1}, U^{I,0}  \big)+Q\big( U^{I,0}, V^{I,1})-\left(\begin{array}{cc}
   0   & 0 \\
    0 & \mu_1\Delta_h 
\end{array}\right)V^{I,1}=0, \\ 
V^{I,1}|_{t=0}=0, \\[5pt]
 v_3^{I,1}|_{\p\Omega}=-(v_3^{p,1}+\tilde{v}_3^{p,1})|_{\p\Omega}, 
\end{array}
\right.
\eeq
 where $U^{I,0}=U_{osc}^{I,0}+(0, v^{\INS})^t$ and 
 \beqs 
Q(V^{I,1}, U^{I,0})+Q\big( U^{I,0}, V^{I,1}) =\left( \begin{array}{c}
v^{e,1}\cdot\na\sigma^{I,0}+u^{I,0}\cdot\na\chi^{I,1}+\f{\gamma-1}{2}(\chi^{I,1}\div u^{I,0}+\sigma^{I,0}\div v^{I,1})\\[4pt]
     v^{I,1}\cdot\na u^{I,0}+ u^{I,0}\cdot\na v^{I,1} +\f{\gamma-1}{2}(\chi^{I,1}\cdot\na\sigma^{I,0}+\sigma^{I,0}\cdot\na\chi^{I,1})
\end{array}\right).
 \eeqs
 We refer to Proposition \ref{prop-well-vI1} for 
 a statement of the uniform well-posedness for \eqref{eq-ve1}. Henceforth, there exists $T_3>0$ such that
 $v^{\INS}, V^{p,0}, V^{I,1}$ all exist on the interval $[0,T_3].$ We will denote $T_1=\min\{T_2,T_3\}$ in the following so that all the profiles constructed in section 2.1-2.3 exist  on 
 $[0,T_1].$


\subsection{Further corrections for the boundary conditions} 
We expect that the quantity
\beqs 
\tilde{U}^a=U_{osc}^{I,0}+\ep U_{osc}^{I,1}+ \sqrt{\ep\nu} S_{osc}^{I,1}+
{U}_{osc}^{B,0}+\tilde{U}_{osc}^{B,0}\\
+\bigg[ \left( \begin{array}{c} 0 \\ v^{\INS} 
  \end{array}\right)
+\left( \begin{array}{c} 
0\\
v^{p,0}+\tilde{v}^{p,0}
\end{array}
\right)+ \sqrt{\nu} V^{I,1}\bigg]
\eeqs
 will serve as the approximate solution. However, 
 the velocity does not vanish on the boundaries:
 
 $\bullet$ For the mean flow,  the horizontal velocity component  
 $\sqrt{\nu}v_h^{I,1}$
 does not vanish on the boundaries. 


$\bullet$ For the oscillating part, the horizontal velocity component $\ep U_{osc,h}^{I,1}$
does not vanish on the boundaries. 

 $\bullet$ The introduction of $\cB_h^{0}$ (the horizontal part of $U^{B,0}$) to compensate the lower boundary condition of $U_{osc,h}^{I,0}$ has non-trivial contribution on the upper boundary $\cB_h^0|_{z=a_3}.$  Similar fact holds for ${\cT}_h^{0}$ on the lower boundary ${\cT}_h^0|_{z=0}.$

  $\bullet$  The vertical component $\sqrt{\ep\nu}\big(S_3+U_{osc,3}^{B,0}+\tilde{U}_{osc,3}^{B,0}\big):=  \sqrt{\ep\nu} F_3 $ vanish on the boundaries up to an error with exponential decay
  \beqs 
  \begin{aligned}
   F_3|_{z=a_3}&=\cB_{3}^{1}|_{z=a_3}
   =\int_{\f{a_3}{\sqrt{\ep\nu}}}^{+\infty} \div_h \cB_h^0 \,\d\zeta\\
   &=-\sqrt{\f{\mu_1}{2}}c_{*}\sum_{k'\in \mathbb{Z}^3}\sum_{\alpha\in\{+,-\}}   \f{|k_h|^2}{|k|^{\f{3}{2}}} b_k^{\alpha}(t) \alpha(\alpha i+1) e^{ik_h\cdot x_h} e^{i\alpha |k|\tau}e^{-\f{a_3}{\sqrt{\ep\nu}}\sqrt{\f{|k|}{2\mu_1}}(1+\alpha i)}.
  \end{aligned}
  \eeqs
Similarly, 
\beqs 
  \begin{aligned}
   F_3|_{z=0}&=\cT_{3}^{1}|_{z=0}
   =-\int_{\f{a_3}{\sqrt{\ep\nu}}}^{+\infty} \div_h \cT_h^0 \,\d\zeta'\\
  &=-\sqrt{\f{\mu_1}{2}}c_{*}\sum_{k'\in \mathbb{Z}^3}\sum_{\alpha\in\{+,-\}} (-1)^{|k_3'|+1}  \f{|k_h|^2}{|k|^{\f{3}{2}}} b_k^{\alpha}(t) \alpha(\alpha i+1) e^{ik_h\cdot x_h} e^{i\alpha |k|\tau}e^{-\f{a_3}{\sqrt{\ep\nu}}\sqrt{\f{|k|}{2\mu_1}}(1+\alpha i)}.
  \end{aligned}
  \eeqs 
It is thus necessary to introduce several further corrections.
 Let $(0, \mathfrak{R}_j)^t, (\div \mathfrak{R}_j=0, j=1,\cdots 4)$ 
 be the corresponding correctors of the above four terms. We begin with the construction of $\mathfrak{R}_1.$
 
\underline{Construction of $\mathfrak{R}_1.$}
Up to an error with exponential decay  $e^{-\f{a_3}{\sqrt{\nu}}},$ we can choose simply the following boundary layer correction $\sqrt{\nu} v^{p,1}=\sqrt{\nu}(v_h^{p,1}, \sqrt{\nu}v_3^{p,2})^t$ to make up the contribution  of $v_h^{I,1}|_{\p\Omega}:$ 
\beqs 
v_h^{p,1}=-\bigg(e^{-\f{z}{\sqrt{\nu}}} v_h^{I,1}|_{z=0}+e^{-\f{a_3-z}{\sqrt{\nu}}} v_h^{I,1}|_{z=a_3}\bigg), \, v_3^{p,2}= \bigg(e^{-\f{a_3-z}{\sqrt{\nu}}} \div_h v_h^{I,1}|_{z=a_3}-e^{-\f{z}{\sqrt{\nu}}} \div_h v_h^{I,1}|_{z=0}\bigg).
\eeqs
It is direct to see that $\div v_h^{p,1}=0$ and
\beq\label{trace-lowvar}
(v_h^{p,1}+v_h^{I,1})|_{z=0}=- e^{-\f{a_3}{\sqrt{\nu}}} v_h^{I,1}|_{z=a_3}, \quad (v_h^{p,1}+v_h^{I,1})|_{z=a_3}=-
e^{-\f{a_3}{\sqrt{\nu}}} v_h^{I,1}|_{z=0}.
\eeq
Now, it remains to find two divergence free vectors $\mathfrak{R}_{11}$ and $\mathfrak{R}_{12}$ to compensate the boundary effects of $\sqrt{\nu}( v_h^{p,1}+v_h^{I,1}, 0)^t$ and $\nu (0, v_3^{p,2}).$

Let $g(z)={z^2}(z-a_3)/{a_3^2}, h(z)=z^2(3a_3-2z)/a_3^3$ two smooth functions on $[0,a_3]$ with properties
 \begin{align*}
     g(0)=g(a_3)=g'(0)=0,\quad g'(a_3)=1, \\
     h(0)=h'(a_3)=h'(0)=0,\quad h(a_3)=1.
 \end{align*}
The corrector $\mathfrak{R}_{11}$ and $\mathfrak{R}_{12}$ can be chosen as
 \beqs 
\mathfrak{R}_{11}=\sqrt{\nu}e^{-\f{a_3}{\sqrt{\nu}}} \left( \begin{array}{c} 
g'(z) v_h^{I,1}|_{z=0}+g'(a_3-z)v_h^{I,1}|_{z=a_3}\\[3pt] 
-g(z)\div_h v_h^{I,1}|_{z=0}+ g(a_3-z)\div_h v_h^{I,1}|_{z=a_3}
 \end{array}\right)\, ,
 \eeqs
 \beqs 
\mathfrak{R}_{12}=\nu\left( \begin{array}{c} 
h'(a_3-z)\tilde{v}_h^{p,2}|_{z=0} +h'(z)\tilde{v}_h^{p,2}|_{z=a_3}\\[5pt] 
h(a_3-z) v_3^{p,2}|_{z=0}+h(z)v_h^{p,2}|_{z=a_3}
 \end{array}\right),   
 \eeqs
 where we denote for convenience $v_3^{p,2}=\div_h \tilde{v}_h^{p,2}.$ 
 
Define $\mathfrak{R}_1=v^{p,1}+\mathfrak{R}_{11}+\mathfrak{R}_{12},$ then by construction,  $\mathfrak{R}_1+(v_h^{I,1},0)^t=0 \text{ on } \p\Omega$ and 
\beq\label{es-R1}   
 \|\mathfrak{R}_1\|_{L_{T_1}^{2}\oH_{co}^{s-6}}+ \|(\pt ,\nu^{\f{1}{2}}\p_z)\mathfrak{R}_1\|_{L_{T_1}^{2}\oH_{co}^{s-8}}\lesssim \nu^{\f{3}{4}}.
\eeq
$\bullet$ Construction of $\mathfrak{R}_2.$ This one could be constructed similarly as $\mathfrak{R}_1.$ 
Define $u_{osc}^{B,1}=\ep(\cB_h^{1}, \sqrt{\ep\nu}\cB_3^{2})^t,$ where 
\beqs 
\cB_h^{1}=-\bigg(e^{-\f{z}{\sqrt{\ep\nu}}} U_{osc,h}^{I,1}|_{z=0}+e^{-\f{a_3-z}{\sqrt{\ep\nu}}} U_{osc,h}^{I,1}|_{z=a_3}\bigg), \, \cB_3^{2}= \bigg(e^{-\f{a_3-z}{\sqrt{\ep\nu}}} \div_h U_{osc,h}^{I,1}|_{z=a_3}-e^{-\f{z}{\sqrt{\ep\nu}}} \div_h U_{osc,h}^{I,1}|_{z=0}\bigg).
\eeqs
To lift up  the boundary condition of $\ep (\cB_h^{1}+U_{osc,h}^{I,1},0)^t$ and $\ep\sqrt{\ep\nu} (0,\cB_3^2),$
we define $\mathfrak{R}_{21}, \mathfrak{R}_{22}$ in the same way as that of $\mathfrak{R}_{11}, \mathfrak{R}_{12}.$  Then $\mathfrak{R}_2= u_{osc}^{B,1}+\mathfrak{R}_{21}+\mathfrak{R}_{22}$ compensates the boundary effects of $U_{osc}^{I,1}$ and 
\beq\label{es-R2}
 \|\mathfrak{R}_2\|_{L_{T_1}^{\infty}\oH_{co}^{s-5-r_0}}+ \|(\pt , (\ep\nu)^{\f{1}{2}}\p_z) \mathfrak{R}_2\|_{L_{T_1}^{2}\oH_{co}^{s-6-r_0}}\lesssim (\ep\nu)^{\f{1}{4}}
 \|W(0)\|_{H^s}.
\eeq
$\bullet$  The construction of $\mathfrak{R}_3, \mathfrak{R}_4$ are easier and can be defined in the same way as that of $\mathfrak{R}_{11}, \mathfrak{R}_{12}.$ 
For instance, 
$\mathfrak{R}_4$ can be taken as
\beqs 
\mathfrak{R}_4=\sqrt{\ep\nu}\left( \begin{array}{c} 
h'(a_3-z)\int_{\f{a_3}{\sqrt{\ep\nu}}}^{+\infty} \cB_h^0 \d \zeta +h'(z)\int_{\f{a_3}{\sqrt{\ep\nu}}}^{+\infty} \cT_h^0 \d \zeta'\\[5pt] 
h(a_3-z) F_3|_{z=0}+h(z)F_3|_{z=a_3}
 \end{array}\right).
 \eeqs
It can be verified that
\beq\label{es-R34}   
 \|(\mathfrak{R}_3, \mathfrak{R}_4)\|_{L_{T_1}^{\infty}\oH_{co}^{s-3}}+ \|(\pt ,(\ep\nu)^{\f{1}{2}}\p_z)(\mathfrak{R}_3, \mathfrak{R}_4)\|_{L_{T_1}^{2}\oH_{co}^{s-5}}
\lesssim  e^{-\f{a_3}{2\sqrt{\ep\nu}}}.
\eeq

\subsection{The approximate solution}
 In this subsection, we derive the equation satisfied by the approximate solution 
 \beq\label{app sol}
 U^a=U^{I,0}+\ep U_{osc}^{I,1}+\sqrt{\ep\nu}S_{osc}^{I,1}+ {U}_{osc}^{B,0}+\tilde{U}_{osc}^{B,0}\
 +V^a
+\left( \begin{array}{c} 0 \\
\sum_{j=1}^4\mathfrak{R}_j \end{array}\right)
\eeq
where $U^{I,0}=U_{osc}^{I,1}+\left(
\begin{array}{c} 0\\
v^{\INS}\end{array}\right)$ and 
\beq\label{defva}
V^a= \left( \begin{array}{c} 
0\\
v^{p,0}+\tilde{v}^{p,0}
\end{array}
\right)+ \sqrt{\nu}\, V^{I,1}\, .
\eeq
Since it is assumed that $s\geq 11+r_0,$ it holds  by the constructions in the previous subsections that $U^a\in L_{T_1}^{\infty}\oH_{co}^6.$

Let us now compute the equation satisfied by $U^a.$
First, it follows by construction \eqref{eq-int0}-
\eqref{artificial-bd} that
\beq\label{eq-oscint}
\begin{aligned}
&\bigg(\pt- L^{\ep,\nu}/\ep
\bigg)\big(U^{I,0}+\ep U^{I,1}_{osc}+\sqrt{\ep\nu}S_{osc}^{I,1}\big)+Q(U^{I,0},U^{I,0}) \\
&=-\nu\left(\begin{array}{cc}
   0   & 0 \\
    0 & \mu_1\p_z^2
\end{array}\right)U^{I,0} +\bigg[\pt\big(\ep U_{osc}^{I,1}+\sqrt{\ep\nu}S_{osc}^{I,1}\big)\bigg]\big(\f{t}{\ep}, t\big)-\left(\begin{array}{cc}0   & 0 \\ 0 & \div_{\nu}\cS \end{array}\right)\big(\ep U_{osc}^{I,1}+\sqrt{\ep\nu}S_{osc}^{I,1}\big)\\
&:= R_0^a +R_1^a+R_2^a\, ,
\end{aligned}
\eeq
where 
$L^{\ep,\nu}$ is the linear operator defined in \eqref{defLepnu}.
The remainder $R_1^a$ is bounded by
\beq\label{esR2a}
\|R_0^a\|_{L_{T_1}^{\infty}
\oH_{co}^{s-4}}+\|\p_z R_0^a\|_{L_{T_1}^{\infty}\oH_{co}^{s-5}} \lesssim \nu\,,
\eeq
while $R_1^a, R_2^a$ being bounded under the `non-resonant' assumption as
\beq\label{esR1a}
\|(R_1^a, R_2^a)\|_{L^1_{T_1} \oH_{co}^{s-5-r_0}}+\|\p_z (R_1^a, R_2^a)\|_{L_{T_1}^{1}\oH_{co}^{s-6-r_0}} \lesssim \ep +\sqrt{\ep\nu}. 
\eeq
Second, by \eqref{Bh0-exact}, we derive that
\beq\label{eq-oscbd1}
(\pt-L^{\ep,\nu}/\ep)U_{osc}^{B,0}
= \left(\begin{array}{c} 0   \\
    {\cB'}_h^0-\mu_1\Delta_h \cB_h^0\\[3pt]
   \nu\mu_1\int_{\f{z}{\sqrt{\ep\nu}}}^{+\infty}\div_h {\cB'}_h^0 \d\zeta -\mu_1\Delta_h \cB_3^1
   \end{array}\right) := R_{31}^a.
\eeq
Notice that by  \eqref{Bh0-exact}, \eqref{bo-bl-3}, it holds that
\beqs 
\big(\pt-\nu\mu_1\p_z^2\big)\cB_{3}^1=\nu\mu_1\int_{\f{z}{\sqrt{\ep\nu}}}^{+\infty}\div_h {\cB'}_h^0 \d\zeta.
\eeqs
We get in a similar manner that 
\beq\label{eq-oscbd2}
(\pt-L^{\ep,\nu}/\ep)\tilde{U}_{osc}^{B,0}= \left(\begin{array}{c}
   0   \\
   {\cT'}_h^0-\nu\mu_1\Delta_h \cT_h^0\\[3pt]
   -\nu\mu_1\int_{\f{a_3-z}{\sqrt{\ep\nu}}}^{+\infty}\div_h {\cT'}_h^0 \d\zeta -\nu\mu_1\Delta_h \cT_3^1
   \end{array}\right) := R_{32}^a.
\eeq
Define  $R_3^a:= R_{31}^a+R_{32}^a.$ It follows from the estimates \eqref{prop-bh0}, \eqref{prop-bhp0} 
that
\beq\label{esR3a} 
\|R_3^a\|_{L^1_{T_1}\oH_{co}^{s-4}}+(\ep\nu)^{\f{1}{2}} \|\p_z R_3^a\|_{L_{T_1}^1\oH_{co}^{s-5}}\lesssim (\ep\nu)^{\f{1}{4}}.
\eeq
Next, from \eqref{prandtl-bo} and \eqref{prandtl-to} one deduces 
that 
\beqs 
\big(\pt-L^{\ep,\nu}) \left( \begin{array}{c} 
0\\
v^{p,0}+\tilde{v}^{p,0}
\end{array}
\right)
+\cN^p=\left(\begin{array}{c}
      0    \\
0\\
\sqrt{\nu}(\pt-\Delta_{\nu})(v_3^{p,1}+\tilde{v}_3^{p,1})
\end{array}\right)+\left(\begin{array}{c} 0\\
\ep\nu\mu_2\nabla \cN_0^p\end{array}\right):=R_{41}^a
\eeqs
where 
\beq \label{defnp}
\cN^p= \left( \begin{array}{c}
     \cN_0^p  \\
     \cN_h(v_h^p)+\cN_h(\tilde{v}_h^p)\\
     0
\end{array} \right)
\eeq
with $ \cN_h(v_h^p)$
and $\cN_h(\tilde{v}_h^p)$  the nonlinear terms appearing in \eqref{prandtl-bo}, \eqref{prandtl-to} and 
\beqs 
 \cN_0^p=\f{\gamma-1}{2}\bigg(\underline{\sigma^{I,0}}\,\div v^{p,0} + \overline{\sigma^{I,0}}\,\div \tilde{v}^{p,0}\bigg)+{v}_h^p\cdot\na_h \underline{\sigma^{I,0}}+\tilde{{v}}_h^p\cdot\na_h \overline{\sigma^{I,0}}.
\eeqs
This, together with \eqref{eq-ve1}, yields 
\beq\label{eq-vp} 
(\pt-L^{\ep,\nu})
V^{a}+
\cN^p+\cN^I
= {R_4^a},
\eeq
where 
\beqs 
\cN^I=Q\big(\sqrt{\nu}
V^{I,1}, U^{I,0}  \big)+Q\big( U^{I,0}, \sqrt{\nu}V^{I,1} 
\big)
\eeqs
and 
$$ {R_4^a}:=R_{41}^a-\left(\begin{array}{c}
      0    \\
    \nu (\mu_1\p_z^2+\mu_2\na\div) v^{I,1}
\end{array}\right).$$
Similar to the estimate of $R_3^a,$ it holds that
\beq\label{esR4a} 
\|{R_4^a}\|_{L_{T_1}^{\infty}\oH_{co}^{s-7}}+\nu^{\f{1}{2}}\|  \p_z{R_4^a}\|_{L_{T_1}^{\infty}H_{co}^{s-8}}\lesssim \nu^{\f{3}{4}}.
\eeq

Now, collecting \eqref{eq-oscint}-\eqref{eq-vp}, 
we find that the approximate solution $U^a$ is governed by
\beq\label{eq-ua-1}
(\pt-L/\ep)U^a+Q(U^a,U^a)-\left(\begin{array}{cc}
   0   & 0 \\
    0 & \div_{\nu}\cS
\end{array}\right)U^a= R^a:=\sum_{j=1}^6 R_j^a,
\eeq
where $R_1^a, R_2^a$ are defined in \eqref{eq-int0},
$R_4^a=(0,{R_4^a}')^t,$
\beqs 
R_5^a= \left(\begin{array}{c}
0\\
\big(\pt-\nu\mu_1\Delta\big)
\sum_{j=1}^4 \mathfrak{R_j}
\end{array}\right), \quad R_6^a=Q(U^a,U^a)-Q(U^{I,0}, U^{I,0})-\cN^p-\cN^e . 
\eeqs

In view of the estimates \eqref{es-R1}-\eqref{es-R34}, 
it holds that
\beq\label{esR5a} 
\|{R_5^a}\|_{L_{T_1}^{2}\oH_{co}^{s-8-r_0}}+(\ep\nu)^{\f{1}{2}}\|  \p_z{R_5^a}\|_{L_{T_1}^{2}\oH_{co}^{s-9-r_0}}\lesssim \nu^{\f{3}{4}}+(\ep\nu)^{\f{1}{4}}. 
\eeq
Moreover, we find after some algebraic computations (see the next subsection) that 
\beq\label{esR6a}  
\|{R_6^a}\|_{L_{T_1}^{1}\oH_{co}^{s-6-r_0}}+(\ep\nu)^{\f{1}{2}}\|  \p_z{R_6^a}\|_{L_{T_1}^{1}\oH_{co}^{s-7-r_0}}\lesssim \nu^{\f{3}{4}}+(\ep\nu)^{\f{1}{4}}+\ep.
\eeq
Collecting \eqref{esR2a}, \eqref{esR1a}, \eqref{esR3a}, \eqref{esR4a}-\eqref{esR6a} we have that the error satisfies 
\beqs 
\|{R^a}\|_{L_{T_1}^{1}\oH_{co}^{s-8-r_0}}+(\ep\nu)^{\f{1}{2}}\|  \p_z{R^a}\|_{L_{T_1}^{1}\oH_{co}^{s-9-r_0}}\lesssim \nu^{\f{3}{4}}+(\ep\nu)^{\f{1}{4}}+\ep,
\eeqs
which gives \eqref{es-error} since $s\geq 11+r_0.$
\subsubsection{Verification of \eqref{esR6a}}
It is just a matter to expand the nonlinear terms. 
Define
\beqs 
\mathscr{R} := U^a-U^{I,0}- V^a
=\ep U_{osc}^{I,1}+\sqrt{\ep\nu}S_{osc}^{I,1}+\left( \begin{array}{c}
    0 \\
\cB^0+\cT^0+\sum_{j=1}^4 \mathfrak{R}_j
\end{array}\right)
\eeqs
which is bounded by
\beqs 
\|
\mathscr{R}\|_{L_{T_1}^1\oH_{co}^{s-5-r_0}}+(\ep\nu)^{\f{1}{2}}\| \p_z\mathscr{R}\|_{L_{T_1}^1\oH_{co}^{s-6-r_0}}\lesssim \ep+(\ep\nu)^{\f{1}{4}}+\nu^{\f{3}{4}}.
\eeqs
Define also $V^B=(0, v^{p,0}+\tilde{v}^{p,0})^t,$
after some  computations, we find that $R_6^a=R_{6,1}^a+R_{6,2}^a$ with
\beqs 
\begin{aligned}
   R_{6,1}^a &= Q(U^{I,0}+V^B, V^B)+Q(V^B, U^{I,0}+V^B )-\cN^p+\left( \begin{array}{c}
        0  \\
        \sqrt{\nu}v_3^{I,1}\p_z v_h^{p,0} \\
        0
   \end{array}\right)   ,\\
   R_{6,2}^a&= Q(\mathscr{R}, U^a-\mathscr{R})+Q(U^a, \mathscr{R})+Q(V^B,\sqrt{\nu}V^{I,1})+Q(\sqrt{\nu}V^{I,1}, V^B+\sqrt{\nu}V^{I,1})-\left( \begin{array}{c}
        0  \\
        \sqrt{\nu}v_3^{I,1}\p_z v_h^{p,0} \\
        0
   \end{array}\right).   
\end{aligned}
\eeqs
It is direct to verify  that 
\beqs 
\|
R_{6,2}^a\|_{L_{T_1}^1\oH_{co}^{s-6-r_0}}+(\ep\nu)^{\f{1}{2}}\| \p_z R_{6,2}^a\|_{L_{T_1}^1\oH_{co}^{s-7-r_0}}\lesssim \ep+(\ep\nu)^{\f{1}{4}}+\nu^{\f{3}{4}}.
\eeqs
For $R_{6,1}^a,$ it is convenient to consider separately the contributions from the bottom and upper boundary layers. Therefore, we write $R_{6,1}^a=R_{6,1}^{a,bot}+ R_{6,1}^{a,up}.$ Let us only give the exact form of $R_{6,1}^{a,bot},$ those of $R_{6,1}^{a,up}$ being almost identical, just replacing $v^p$ by $\tilde{v}^p$ and $\underline{U^{I,0}}$ by $\overline{U^{I,0}}.$ 

Denote $R_{6,1}^{a}=R_{6,1}^{a,bot}+ R_{6,1}^{a,up}, \quad R_{6,1}^{a,bot}=((R_{6,1}^{a,bot})_0, (R_{6,1}^{a,bot})_h, (R_{6,1}^{a,bot})_3).$
By recalling the 
the definition  of $\cN^p$ in \eqref{defnp} and 
of the quadratic form
\beqs 
Q(U,V)=(u\cdot \na) V+\f{\gamma-1}{2}\sigma L V, \quad  (U=(\sigma, u)^t),
\eeqs
that
\beqs
\begin{aligned}
(R_{6,1}^{a,bot})_0&=
v_h^{p,0}\cdot\na_h(\sigma^{I,0}-\underline{\sigma^{I,0}})
+\sqrt{\nu}v_3^{p,1}
\p_z\sigma^{I,0}+\f{\gamma-1}{2}(\sigma^{I,0}-\underline{\sigma^{I,0}})\,\div v^{p,0}, 
\end{aligned}
\eeqs
\beqs
(R_{6,1}^{a,up})_h=v_h^p\cdot \na_h(u^{I,0}-\underline{u^{I,0}})_h+\sqrt{\nu}v_3^{p,1}\p_z u^{I,0}+\big(u^{I,0}-\underline{u^{I,0}}\big)_h \cdot\na_h {v}^p+(u_3^{I,0}-z\underline{\p_z u_3^{I,0}}-\sqrt{\nu}(v_3^{e,1}-\underline{v_3^{e,1}}))\p_z u_h^p,
\eeqs
and 
\beqs
(R_{6,1}^{a,up})_3=v_h^p\cdot \na_h(u_3^{I,0}-\underline{u_3^{I,0}})+\sqrt{\nu}v_3^{p,1}\p_z u_3^{I,0}+\sqrt{\nu} (u^{I,0}+v^p)_h\cdot\na_h v_3^{P,1}+
\big(u_3^{I,0}-\underline{u_3^{I,0}}+\sqrt{\nu}v_3^{p,1})\p_z v_3^{p,1}.
\eeqs
 It can be verified that
\beqs 
\|R_{6,1}^a\|_{L_{T_1}^{\infty}H_{co}^{s-1}}+ \nu^{\f{1}{2}}\|\p_z R_{6,1}^a\|_{L_{T_1}^{\infty}H_{co}^{s-2}}\lesssim \nu^{\f{3}{4}}.
\eeqs

\section{Energy functions and main strategy}

From this section to Section 5, we focus on the stability analysis of the remainder $U^R$ 
and provide the proof of Theorem \ref{thm-remainder}. In the current section, we introduce the definitions of key functionals and outline the main strategies that will be employed in the proof of Theorem \ref{thm-remainder}.
\subsection{Energy functionals}
%
%
%
%


For any $T\geq 0,$  define
\beq \label{def-cNT}
\mathcal{N}_{T}=\mathcal{N}_{T}(\sigma^{R},u^{R}):=\{
(\sigma^{R},u^{R})
\in L^{\infty}([0,T],L^2(\Omega))\big|\, \mathcal{N}_{T}(\sigma^{R},u^{R})<+\infty\},
\eeq
where 
$$\mathcal{N}_{T}(\sigma^{R},u^{R})=\mathcal{E}_T(\sigma^{R},u^{R})+\mathcal{A}_{T}(\sigma^{R},u^{R}),$$ 
with the energy norm $\mathcal{E}_{T}$
and $L_{x}^{p}\,(p=3,6,+\infty)$ type norm $\mathcal{A}_{T}$ defined respectively by
\begin{equation}\label{def-cET}
\begin{aligned}
&\qquad \mathcal{E}_{T}(\sigma^{R},u^{R})\\
&:=\eta^{-1} \Big(\|(\sigma^R, u^R)\|_{L_{T}^{\infty}\underline{H}_{tan}^3}+ \|(\nu^{\f12}\p_z u^R, \nabla_h u^R,\div u^R)\|_{L_T^{2}\underline{H}_{tan}^3} +
\|(\nabla\sigma^{R},\div u^R)\|_{L_T^2H_{co}^2\cap L_T^{\infty}\underline{H}_{co}^1}\\
&\qquad\qquad  + \|\phi\, \omega_h^R\|_{L_T^{\infty}H_{co}^2}+\|\phi(\na_h,\nu^{\f12}\p_z)\omega_h^R\|_{L_T^{2}H_{co}^2}\Big)\\
&+ \eta^{-1}\Big(\ep\|(\na \sigma^R,\div u^R)\|_{L_T^{\infty}\underline{H}_{co}^{2}}+ \ep\|\na\div u^R\|_{L_T^{2}\underline{H}_{co}^{2}\cap L_T^{\infty}\underline{H}_{co}^{1}} +\ep^{\f12}\|\na\div u^R\|_{L_T^2H_{tan}^2}+\tau^2\|\na^2\sigma^R\|_{L_T^2H_{co}^1}\\
&\qquad \qquad +\min\{(\ep^2\nu)^{\f13}, (\ep\nu)^{\f12}\}\|\na u^R\|_{L_T^{\infty}H_{co}^{1}}+\ep\nu \|\na^2 u^R\|_{L_T^{\infty}H_{co}^{1}\cap L_T^{2}\underline{H}_{co}^{2}}\Big),
\end{aligned}
\eeq
and
\beq\label{def-cAT}
\begin{aligned}
& \mathcal{A}_{T}(\sigma^{R},u^{R})\\
&:=\eta^{-1}\Big(
\|
\div u^R\|_{ L_T^2L_x^{\infty}\cap L^2_TW_{co}^{1,6}
}+ \min\{1, (\nu^3/\ep^7)^{\f14} \} \|\na\sigma^R\|_{L_T^2L_x^{\infty}}
+\tau \|\na\sigma^R\|_{L^2_TW_{co}^{1,3}}
\\
&\qquad +\|(\sigma^R, u_3^R)\|_{L_T^{\infty}L_x^{\infty}\cap L_T^2W_{co}^{1,\infty}\cap L_T^2W_{co}^{2,6}}+\|(\ep^{\f14}Z_{tan},\ep^{\f12+\delta}Z_3)\div u^R\|_{L_T^2L_x^{\infty}}+\ep\|(Z\sigma^R, \div u^R)\|_{L_{T}^{\infty}L_x^{\infty}}  \Big)\\
 &\,+\eta^{-1}\nu^{\frac{1}{4}}\Big(\|(Z_{tan}, \min\{\ep^{\f23}\nu^{\f{5}{6}}, \ep^{\f12}\nu\}^{\f12}\p_z)u^R\|_{L_T^2W_{tan}^{1,\infty}}+
\|\ep\nu\p_z^2 u^R\|_{L_T^2L_x^{\infty}}
+\|Z_3 u_h^R\|_{L_T^2(W_{tan}^{1,\infty}\cap W_{tan}^{2,3})}\Big)
 \\
 &\, +\eta^{-1}\Big(\|(\min\{(\ep^2\nu)^{\f16},(\ep\nu)^{\f14}\} u_h^R, (\ep\nu)^{\f12}Z u_h^R, 
 (\ep^{\f{3+4\delta}{2}}+(\ep\nu)^{\f34+\delta})\na u_h^R)\|_{L^{\infty}_{T}L^{\infty}_x\cap L_T^{\infty}W_{co}^{1,3}}\Big), 
\end{aligned}
\end{equation}
where $\omega_h$ is the horizontal component of the vorticity $\omega^R=\curl u^R$ and 
$\tau=\min\{1,(\nu^5/\ep^7)^{\f{1}{12}}, (\nu^2/\ep^3)^{\f14}\}$ and $\delta>0$ is an arbitrary small number that will be chosen later. 

\medskip

Under the above definitions, Theorem \ref{thm-remainder} will be the consequence of the following uniform regularity estimates: 
    \begin{thm}[Uniform estimates]\label{thm1}
Assume that the initial data $(u_0^R,\sigma_0^R)$ satisfies the compatibility condition \eqref{comp-cond} and is such that 
$\cN_0(u^R,\sigma^R)<+\infty.$
Assume also that $\sigma_0^R$ is such that 
$$\rho_0(x)=\bigg(\f{\gamma-1}{2}\ep(\sigma_0^{a}+\sigma_0^R)+1\bigg)^{\f{2}{\gamma-1}}\in [1/2, 2],\quad  \forall \, x\in \Omega.$$
There exist $\ep_0\in (0, 1]$ and $T_{0}>0,$ such that, for any $$(\ep,\nu)\in A_0:=\big\{0<\ep\leq \ep_0, \nu=\ep^{\kappa}, \text{ with } \kappa\in (0,3)\big\},$$
the system \eqref{eq-UR} admits a unique solution $(\si^{R},u^{R})$ 
which satisfies
 \beq\label{epsigmaLinfty}
 \rho(t,x) \in [1/4,4], \qquad \forall \, (t,x)\in [0,T_0]\times \Omega 
 \eeq
 and 
 \beq\label{bddcN}
 \sup_{(\ep,\nu)\in A_0 
 }\mathcal{N}_{T_0}(\si^{R},u^{R})< +\infty.
 \eeq
\end{thm}
Let $T\in(0,T_1],$ where $T_1$ is the existence time for the approximate solution $U^a$ constructed in the previous section. 
For short, we will denote $(\mathcal{N}_{T}, \cE_T, \cA_T)=(\mathcal{N}_{T}, \cE_T, \cA_T)(\sigma^R, u^R). $ For notational convenience, we also define the following quantities $\cM_T^a=\cM_T^a(U^a, R^a)$ related to the approximate solution $U^a$ and the source term $R^a:$
\beq\label{def-cM}
\begin{aligned}
\cM_T^a
&=\|U^a\|_{L_T^{\infty}H_{co}^5}+ \eta^{-1} \|\ep\nu \na^2 u^a\|_{L_T^2H_{co}^3}+\|(\div u^a,\na\sigma^a)\|_{L_T^{\infty}W_{co}^{3,\infty}}+\|U^a\|_{L_T^{\infty}W_{co}^{4,\infty}} \\
&+\eta^{-1}\big(\|R^a\|_{L_T^1\underline{H}_{co}^3}+\|(\ep\nu)^{\f12}\na R^a\|_{L_T^1H_{co}^2}+\|\ep R^a, (\ep\nu)^{\f12}\min\{1, \sqrt{\ep/\nu}\} \p_z R^a\|_{L_T^2H_{co}^2}\big).
\end{aligned}
\eeq
Note that by the construction in the previous section, it holds that
$\cM_T^a(U^a)<+\infty.$
In order to prove Theorem \ref{thm1}, it suffices to prove the following a-priori estimates 
\begin{prop}\label{prop-unifromes}
Let $T\in(0,T_1].$ Assume that $U^R=(\sigma^R, u^R)$ is a solution to the system  \eqref{eq-UR}
in $L^{\infty}([0,T],\underline{H}_{co}^3),$  and 
\eqref{epsigmaLinfty} holds true for any $(t,x)\in [0,T]\times \Omega.$ Then it holds that  for any $(\ep,\nu)\in A_0,$ ie, $\nu=\ep^{\kappa}$ with $\kappa\in(0,3),$ there is a constant $d(\kappa)>0$ and  polynomial $\Lambda_1, \Lambda_2$ such that
\begin{align*}
  \mathcal{N}_{T}^2
  \lesssim  \mathcal{N}_{0}^2+\Lambda_1(\cM_T^a)
  + (T+\ep)^{d({\kappa})}\Lambda_2\big(\cM_T^a,
  \cA_{T}
  \big) \big(\cN_T^2+\cN_T^4\big).
\end{align*}
\end{prop}
%
%
%

\medskip

 We will establish the relevant estimates for the energy norms and the $L_x^p$ norms ($p= 3, 6, +\infty$) included in $\cE_T$ and $\cA_T$ in the following sections. Before diving into the details, we provide some useful remarks regarding the quantities in $\cE_T$ and $\cA_T$. The first two lines of \eqref{def-cET} consist of quantities that are uniform in $\ep$ and $\nu$, while the last two lines of \eqref{def-cET} contain non-uniform estimates for quantities with higher normal derivatives.
It is useful to note that by its definition, 
\begin{align}\label{energynorm-uh}
    \|u_h^R\|_{L_T^{\infty}H_{co}^2\cap L_T^2\underline{H}_{co}^3}+\|Z_3 u_h^R\|_{L_T^{\infty}H_{tan}^2}+\nu^{\f12}\|\na u_h^R\|_{L_T^2H_{co}^2}\lesssim \eta \,\cE_T.
\end{align}
The norm $\cA_T$ includes several $L_x^p$ type norms, which are essential for closing the estimates. Although it involves numerous quantities, most can be derived easily using Sobolev embedding and the energy norms within $\cE_T$. This also explains the weights appearing in each term. 
Let us remark that by the definition \eqref{def-cAT}, it holds that 
\begin{align}\label{pzuh-dinfty}
 \ep^2 \|\na u_h^R\|_{L_T^{\infty}(L_x^{\infty}\cap W_{co}^{1,3})}\lesssim \ep^2\big(\ep^{-\f{3+4\delta}{2}}+(\ep\nu)^{-(\f{3}{4}+\delta)}\big)\eta \cN_T\lesssim \big(\eta+\ep^{\min\{\f14,\f{3-\kappa}{4}\}}\big)\cN_T
\end{align}
if we choose $\delta=\min\{1/8, \f{3-\kappa}{4(\kappa+1)} \}.$ Moreover, it can be checked that
\beq\label{rewrite-nasigmaR}
\begin{aligned}
&\|\na\sigma^R\|_{L_T^2(L_x^{\infty}\cap W_{co}^{1,3})}+\|\div u^R\|_{L_T^2W_{co}^{1,6}}\\
&\lesssim \max\{1, (\ep^7/\nu^5)^{\f{1}{12}}, (\ep^3/\nu^2)^{\f14}, (\ep^7/\nu^3)^{\f14}\}\eta \cN_T\lesssim (\eta+ \ep^{\f14})\cN_T,
\end{aligned}
\eeq
\begin{align}\label{rewrite-pzu3R}
&\|\p_z u_3^R\|_{L_T^2L_x^{\infty}}\lesssim \|(\div u^R, \div_h u_h^R)\|_{L_T^2L_x^{\infty}}\lesssim \nu^{-\f14}\eta\cN_T\lesssim (\ep^{\f14}+\nu^{\f12})\cN_T 
\end{align}
for any $(\ep,\nu)\in A_0.$ These estimates will be frequently used in the following subsections.

\subsection{Main strategy}\label{secmainstrategy}
 To prove Proposition \ref{prop-unifromes},  the main step is to give estimate of $\mathcal{E}_{T}$. Before we state the main strategy, we rewrite the system \eqref{eq-UR} into the following form
    \beq \label{CNS-remainder}
\left\{
\begin{array}{l}
 \big(\displaystyle\pt+ (u^{a}+u^R)\cdot\na\big) \sigma^{R}+u^R\cdot\na\sigma^a+\f{\rho^{\f{\gamma-1}{2}}\div u^{R}}{\ep}+R^a_{\sigma}=0,\\[8pt]
 \big(\displaystyle\pt+ (u^{a}+u^R)\cdot\na\big) u^{R}+u^R\cdot\na u^a+
\f{\rho^{\f{\gamma-1}{2}}\nabla \sigma^{R}}{\ep}-\f{1}{\rho}\,\div_{\nu }\mathcal{S}u^{R}+R^a_{u}=\big(\f{1}{\rho}-1\big)\div_{\nu} S u^a, \\ 
u^R|_{\p\Omega}=0,
\end{array}
\right.
\eeq
where ${\rho}=\big(\f{\gamma-1}{2}\ep(\sigma^{a}+\sigma^R)+1\big)^{\f{2}{\gamma-1}}.$ 

 The main steps to prove the uniform regularity estimates are the following: 

\textbf{
Step 1. Uniform energy estimates involving only tangential derivatives.}  In this step, we prove the boundedness of $\eta^{-1}(\|U^R\|_{L_t^{\infty}\underline{H}_{tan}^3}+\|(\nu^{\f12}\p_z u^R,\div u^{R}, \na_h u^R)\|_{L_t^{2}\underline{H}_{tan}^3}).$ Since $\ep\pt, \p_{x_1}, \p_{x_2}$ commute with the operator $(\div, \na)/\ep,$ it could be obtained by direct energy estimates. Let us also remark that the uniform boundedness of $\|(u^R, \na_{h} u^R, \div u^R)\|_{L_t^2L_x^{\infty}}$ is needed here in order to propagate the regularity, which is available as long as 
 $\eta\leq \nu^{\f14}.$
 
\textbf{
Step 2. Uniform energy estimates for $(\na\sigma^R, \div u^R)$.} We prove in this step the boundedness of
$\|(\na\sigma^R, \div u^R)\|_{L^2H_{\tan}^2}.$ The estimates for $\div u^R$ is the consequence of the first step. To get the estimate for $\na\sigma^R,$ we rewrite the equations \eqref{eq-UR} to get a new equation for $\na \sigma^R$ \eqref{eq-pzsigma} and work on this equation to get the desired estimates.


\textbf{Step 3.  Estimates for the vorticity.} 

    \textbf{ Step 3.1. Control of the vorticity away from the boundaries.} We control uniformly in this step 
    $\eta^{-1}(\|\phi \,\omega^R\|_{L_t^{\infty}H_{co}^{2}}+\nu^{\f12}\|\phi \na \omega^R\|_{L_t^{2}H_{co}^{2}}),$ where $\phi(z)=z(a_3-z).$ This can be anticipated since there is no oscillation in the vorticity equation. The interaction between the vorticity and the compressible (oscillating) part occurs through the boundary condition 
    \beqs 
    \nu\p_z\omega^R|_{\p\Omega}={\na\sigma^R/\ep}|_{\p\Omega}+\cdots.
    \eeqs
  
 \textbf{ Step 3.2. Control of the vorticity with weight.} In this substep, we control $$ \min\{(\ep^2\nu)^{\f13}, (\ep\nu)^{\f12}\}\eta^{-1}\big(\|\omega^R\|_{L_t^{\infty}H_{co}^{1}}+\|\na_{\nu} \omega^R\|_{L_t^{2}H_{co}^{1}}\big)+\ep\nu \eta^{-1}\big(\|\omega^R\|_{L_t^{\infty}\underline{H}_{co}^{2}}+\|\na_{\nu} \omega^R\|_{L_t^2\underline{H}_{co}^{2}}\big).$$ 
 A weight like $(\ep^2\nu)^{\f13}$ need to be added due to the boundary condition \eqref{bdrycond-omegah}.

\textbf{Step 4. Some auxilary estimates.} In this step, we control $$\ep\eta^{-1} \big(\|(\na \sigma^R,\div u^R)\|_{L_t^{\infty}\underline{H}_{co}^{2}}+ \|\na\div u^R\|_{L_t^{2}\underline{H}_{co}^{2}}\big), \,\,  $$
which is useful in the next step. 

\textbf{Step 5.  Recovering high order $Z_3$ derivatives for $(\na \sigma^R, \div u ^R)$.} In this step, we use the equation to recover 
$\eta^{-1}\|(\na\sigma^R, \div u ^R)\|_{L_t^{\infty}\underline{H}_{co}^{1}\cap L_t^{2}H_{co}^{2} }.$


\section{estimate of $\cE_T$  }
In this section, we aim to prove the following a-priori estimates 
for the energy norm $\cE_T$ defined in \eqref{def-cET}. In the following, we will denote by $\Lambda$ a polynomial that may vary from line to line.
\begin{prop}\label{prop:est-E-T}
    Under the same assumption as in Proposition \ref{prop-unifromes}, there is some constant $d_1(\kappa)>0$ such that
\begin{align*}
  \mathcal{E}_{T}^2
  \lesssim  \mathcal{N}_{0}^2+\Lambda(\cM_T^a)
  + (T+\ep)^{d_1({\kappa})}\Lambda\big(\cM_T^a,
  \cA_{T}
  \big) \big(\cN_T^2+\cN_T^4\big).
\end{align*}
\end{prop}
The proof of this proposition is divided into several subsections, corresponding to Steps 1–5 listed in Section \ref{secmainstrategy}.

\subsection{Energy estimates for tangential derivatives}
\begin{lem}\label{lem-hightan}
    Let $(\sigma^R, u^R)$ be a solution to the system  \eqref{CNS-remainder}
in $L^{\infty}([0,T],\underline{H}_{co}^3)$  and 
\eqref{epsigmaLinfty} holds true for any $(t,x)\in [0,T]\times \Omega.$ Then it holds that  for any $(\ep,\nu)\in A_0,$ any $t\in [0,T]$
\beq\label{ES-hightan}
\begin{aligned}
  &  \|(\sigma^R, u^R)\|_{L_t^{\infty}\underline{H}_{tan}^3}^2+\|(\na_h u^R, \nu^{\f12}\p_z u^R, \div u^R)\|_{L_t^{2}\underline{H}_{tan}^3}^2\\
    &\lesssim 
\|(\sigma^R, u^R)(0)\|_{\underline{H}_{tan}^3}^2+ \eta^{2} \Lambda\big(\cM_T^a\big)
 +\eta^2\big(T^{\f12}+\ep^{\min\{\f{1}{20},\f{\kappa}{3}\}}\big)\Lambda\big(\cM_T^a,\cA_T\big)\big(\cN_T^2+\cN^4_T\big).
\end{aligned}
\eeq
\end{lem}
\begin{proof} 
{\underline{ The basic energy estimate.}} Taking the scalar product of equation \eqref{CNS-remainder} and $\rho^{-\f{\gamma-1}{2}}(\sigma^R, u^R),$ integrating over $Q_t=[0,t]\times\Omega$  and using the boundary condition $u^R|_{\p\Omega}=u^a|_{\p\Omega}=0,$ we obtain from the standard integration by parts that 
\beq\label{EI-zero}
\begin{aligned}
    &\f{1}{2}\int_{\Omega} \rho^{-\f{\gamma-1}{2}}\big(|\sigma^R|^2+|u^R|^2\big)(t)\, \d x+\iint_{Q_t} \rho^{-\f{\gamma+1}{2}} \big(\mu_1\nu |\p_z u^R|^2+\mu_1|\na_h u^R|^2+\mu_2|\div u^R|^2\big)\,\d x \d s \\
&=\f{1}{2}\int_{\Omega} \rho^{-\f{\gamma-1}{2}}\big(|\sigma^R|^2+|u^R|^2\big)(0)\, \d x\\
&\quad+\f12\iint_{Q_t} \big(|\sigma^R|^2+|u^R|^2\big)\big(\pt \rho^{-\f{\gamma-1}{2}}+\div \big(\rho^{-\f{\gamma-1}{2}}(u^a+u^R)\big)\big) \, \d x\d s\\
&\quad -\iint_{Q_t}\mu_1 \big( \nu\p_z(\rho^{-\f{\gamma+1}{2}})\p_z u^R+\na_h(\rho^{-\f{\gamma+1}{2}})\cdot\na_h u^R \big)\cdot u^R+\mu_2\na(\rho^{-\f{\gamma+1}{2}})\cdot u^R\, \div u^R \, \d x\d s \\
& \quad -\iint_{Q_t}\sigma^R\, R_{\sigma}^a+u^R\cdot R_{u}^a+ \big(\f{1}{\rho}-1\big)\div_{\nu} S u^a \cdot u^R \, \d x\d s =\colon I^0_1+I^0_2+I^0_3+I^0_4.
\end{aligned}
\eeq
Thanks to the equation $\pt \rho+ \div(\rho (u^a+u^R))=0,$ it holds that
$$\pt \rho^{-\f{\gamma-1}{2}}+\div \big(\rho^{-\f{\gamma-1}{2}}(u^a+u^R)\big)=\f{\gamma+1}{2}\rho^{-\f{\gamma-1}{2}} \div (u^a+u^R).$$
The term $I_2^0$ can thus be bounded by 
\begin{align*}
    |I_2^0|\lesssim T^{\f12}\|(\sigma^R, u^R)\|_{L_t^{\infty}L^2}^2 (\|\div u^a\|_{L_t^2L^{\infty}}+\|\div u^R\|_{L_t^2L^{\infty}})\lesssim T^{\f12}\big( \Lambda\big(\cM_T^a\big)+\eta \cN_T\big)\eta^2 \cN^2_T.
\end{align*}
Moreover, we can control the third term $I_3^0$ simply by 
\begin{align*}
    |I_3^0|&\lesssim  \|\na \rho \|_{L_t^2L^{\infty}} \|u^R\|_{L_t^{\infty}L^2} \big(\nu^{\f12}\|\p_z u^R\|_{L_t^2L^2}+\|(\div u^R, \na_h u^R)\|_{L_t^2L^2}\big)\\
    &\lesssim \ep \big(\Lambda\big(\cM_T^a\big)+\eta\cN_T\big)\eta^2\cN_T^2,
\end{align*}
where we have used the fact $\|\na \rho \|_{L_t^2L^{\infty}} \lesssim \ep \|\na(\sigma^a,\sigma^R)\|_{L_t^2L^{\infty}}\lesssim \ep \big(
\cM_T^a+\eta\cN_T\big).$ Finally, for the last term $I_4^0,$ one has
\begin{align*}
    |I_4^0|\lesssim \|(\sigma^R, u^R)\|_{L_t^{\infty}L^2} \bigg(\|(R^a_{\sigma}, R_u^a)\|_{L_t^1L^2}+\|(\sigma^a, \sigma^R)\|_{L_{t,x}^{\infty}} \|\ep \div_{\nu} S u^a\|_{L_t^2L^2}\bigg).
\end{align*}
Plugging the previous three estimates into \eqref{EI-zero} and using Young's inequality, we readily get \eqref{ES-hightan} with $\underline{H}_{\tan}^3$ replaced by $L^2.$ 

{\underline{The high order energy estimates. }}Denote 
$\dot{U}^R=(\dot{\sigma}^R, \dot{u}^R)=Z^{\alpha} (\sigma^R, u^R),$ where $\alpha_0\leq 2,\alpha_3=0.$ Then $\dot{U}^R$ solves the equation
\begin{align*}
  &  \big(\displaystyle\pt+ (u^{a}+u^R)\cdot\na\big) \dot{U}^R+\dot{u}^R\cdot\na U^a+\ep^{-1}\rho^{\f{\gamma-1}{2}}L\dot{U}^R 
  -\left(\begin{array}{c}
       0  \\
       \f{1}{\rho} \div_{\nu} S \dot{u}^R 
  \end{array}\right)+Z^{\alpha}R^a\\
  &= Z^{\alpha}\left(\begin{array}{c}
       0  \\
       (\f{1}{\rho}-1)\div_{\nu} S {u}^a
  \end{array}\right)+\sum_{k=1}^4\cC_k^{\alpha} 
 + \left(\begin{array}{c}   0  \\\,[Z^{\alpha}, {1}/{\rho} ]\div_{\nu}S u^R \end{array}\right),
\end{align*}
where 
\begin{align*}
    \cC_1^{\alpha}=[Z^{\alpha}, u^a]\cdot\na U^R, \quad \cC_2^{\alpha}=[Z^{\alpha}, u^R]\cdot\na U^R,\quad \cC_3^{\alpha}=\sum_{j=1}^3[Z^{\alpha}, \p_j U^a] u_j^R, \,\quad  \cC_4^{\alpha}=\ep^{-1}[Z^{\alpha},\rho^{\f{\gamma-1}{2}}]L U^R .  
\end{align*}
Performing the similar energy estimates as we just did for $L_x^2$ estimate, we find the following energy inequality
\beq
\begin{aligned}\label{EI-hightan}
   & \|\dot{U}^R\|_{L_t^{\infty}L^2}^2+ \|\big(\nu^{\f12}\p_z\dot{u}^R, \na_h \dot{u}^R, \div \dot{u}^R \big)\|_{L_t^2L^2}^2\lesssim  \|\dot{U}^R(0)\|_{L^2}^2+\|(R^a, ({1}/{\rho}-1)\,\div_{\nu}S u^a)\|_{L_t^1\underline{H}_{tan}^3}^2\\
&+\sum_{j=1}^4\|\cC_j^{\alpha}\|_{L_t^1L^2}^2+\bigg|\iint_{Q_t} ([Z^{\alpha}, {1}/{\rho}]\div_{\nu}S  u^R)\cdot \rho^{-\f{\gamma-1}{2}} Z^{\alpha} u^R\bigg| 
+\eta^2(T^{\f12}+\ep)\big(\Lambda\big(\cM_T^a\big)\cN_T^2+\eta\cN_T^3\big).
\end{aligned}
\eeq
Next, we give the estimates of the right hand of \eqref{EI-hightan}.

$\bullet$ For $\|({1}/{\rho}-1)\div_{\nu}S u^a\|_{L_t^1\underline{H}_{tan}^3}$, we have
\begin{align*}
\|({1}/{\rho}-1)\div_{\nu}S u^a\|_{L_t^1\underline{H}_{tan}^3} &\leq \big(\Lambda(\|\sigma^a\|_{L_t^2W_{tan}^{3,\infty}})+\|\sigma^R\|_{L_t^2W_{tan}^{1,\infty}}\big)\|\ep \,\div_{\nu}S u^a\|_{L_t^2\underline{H}_{tan}^3}\\
&\qquad +T^{\f12}
\Lambda\big(\|\ep\sigma^a\|_{L^{\infty}_tW_{co}^{2,\infty}}+\|\ep Z_{tan}\sigma^R\|_{L_{t,x}^{\infty}}\big)
\|\sigma^R\|_{L_t^{\infty}\underline{H}_{tan}^3}\|\ep \,\div_{\nu}S u^a\|_{L_t^2{W}_{tan}^{1,\infty}}\\
 &\lesssim \eta \big(\Lambda(\cM_T^a) +(T^{\f12}+\eta)\Lambda\big(\cM_T^a,\cA_T)\cN_T\big),
 \end{align*}
 where we use $\Lambda(\cdot)$ to denote a polynomial function with degree no greater than $3$ that may change line to line. 
 
 \medskip
 
Let us now deal with the commutator terms $\cC_1^{\alpha}-\cC_4^{\alpha}.$

$\bullet$ For the term $\cC_1^{\alpha},$ we write 
\begin{align*}
\cC_1^{\alpha}=\sum_{\beta<\alpha} Z_{tan}^{\alpha-\beta} u_h^a \, Z_{tan}^{\beta}\na_h U^R+ Z_{tan}^{\alpha-\beta} (u_3^a/\phi) Z_3 Z_{tan}^{\beta}U^R,
\end{align*}
where the weight function $\phi=\phi(z)=z(a_3-z).$ Since $u_3^a|_{\p\Omega}=0,$ it follows from the Fundamental Theorem of Calculus  that 
$\|u_3^a(t)/\phi\|_{W_{co}^{3,\infty}}\lesssim \|\p_z u_3^a(t)\|_{W_{co}^{3,\infty}}.$ Therefore, we can control $\cC_1^{\alpha}$ roughly by
\begin{align}\label{es-cc1}
   \| \cC_1^{\alpha}\|_{L_t^1L^2}\lesssim T^{\f12} \|(u_h^a,\p_z u_3^a)\|_{L_t^{\infty}W_{co}^{3,\infty}} \|(\na_h , Z_3)U^R\|_{L_t^2H_{tan}^2}\lesssim T^{\f12}\eta\Lambda\big(\cM_T^a\big)\cN_T.
\end{align}

$\bullet$ For $C_2^{\alpha},$ we count the derivative acting on each element and control them in different ways. 
On the one hand, by utilizing the estimates \eqref{rewrite-nasigmaR}, \eqref{rewrite-pzu3R}, 
we can estimate $[Z^{\alpha}, u^R]\cdot \na (\sigma^R, u_3^R)$ as 
\beq\label{com-C2-1}
\begin{aligned}
\big\| [Z^{\alpha}, u^R]\cdot \na (\sigma^R, u_3^R)\big\|_{L_t^1L^2}&\lesssim \|\na (\sigma^R, u_3^R)\|_{L_t^2{H}_{tan}^2} \|u^R\|_{L_t^2W_{tan}^{2,\infty}}+\|\na (\sigma^R, u_3^R)\|_{L_t^2L^{\infty}}\|u^R\|_{L_t^2\underline{H}_{tan}^3}\\
&\lesssim  
(\ep^{\f14}+\nu^{\f12})\eta \cN^2_T.
\end{aligned}
\eeq
On the other hand, for the term $[Z^{\alpha}, u^R]\cdot \na u_j^R$ $(j=1,2),$ we can first find in the same way as the above estimate that 
\beq\label{com-C2-2}
\|[Z^{\alpha}, u_h^R]\cdot \na_h u_j^R\|_{L_t^1L^2}\lesssim  \|u_h^R\|_{L_t^2\underline{H}_{tan}^3}\|u_h^R\|_{L_t^2W_{tan}^{2,\infty}}
\lesssim (\ep^{\f14}+\nu^{\f12})\eta \cN^2_T.
\eeq
 We then use the identity $[Z^{\alpha}, u_3^R] \p_3 u_j^R= [Z^{\alpha}, u_3^R/\phi] Z_3 u_j^R$ to get that
 \beq\label{com-C2-3}
\begin{aligned}
 \| [Z^{\alpha}, u_3^R] \p_3 u_j^R\|_{L_t^1L^2}&\lesssim  \min\{ \|Z_3 u_j^R\|_{L_t^2{W}_{tan}^{2,3}}\|u_3^R/\phi\|_{L_t^2W_{tan}^{1,6}}, \|Z_3 u_j^R\|_{L_t^2{H}_{tan}^{2}}\|u_3^R/\phi\|_{L_t^2W_{tan}^{1,\infty}}\}\\
 &\qquad +\|Z_3 u_j^R\|_{L_t^2W_{tan}^{1,\infty}}\|u_3^R/\phi\|_{L_t^2\underline{H}_{tan}^3}.
\end{aligned}
\eeq
Let us estimate the quantities appearing in the right hand side. First,
it holds by the Hardy inequality that
\begin{align}
\|u_3^R/\phi\|_{L_t^2\underline{H}_{tan}^3}\lesssim \|\p_z u_3^R\|_{L_t^2\underline{H}_{tan}^3}\lesssim \|(\div u^R, \na_h u_h^R)\|_{L_t^2\underline{H}_{tan}^3}\lesssim \eta \cN_T, \notag\\
   \|u_3^R/\phi\|_{L_t^2W_{tan}^{1,6}} \lesssim \|\p_z u_3^R\|_{L_t^2W_{tan}^{1,6}}\lesssim \|(\div u^R, \div_h u_h^R)\|_{L_t^2W_{tan}^{1,6}}\lesssim \nu^{-\f16}\eta \cN_T, \label{ineq-319}\\
    \|u_3^R/\phi\|_{L_t^2W_{tan}^{1,\infty}} \lesssim  \|(\div u^R, \div_h u_h^R)\|_{L_t^2W_{tan}^{1,\infty}}\lesssim (\ep^{-\f12}+\nu^{-\f14})\eta\cN_T. \label{ineq-320}
\end{align}
Note that in the second inequality, we have used 
the interpolation 
\begin{align*}
    \|u_h^R\|_{L_t^2W_{tan}^{2,6}}\lesssim \|u_h^R\|_{L_t^2H_{tan}^2}^{\f13}\|u_h^R\|_{L_t^2W_{tan}^{2,\infty}}^{\f23}\lesssim \nu^{-\f16} \eta\cN_T.
\end{align*}
Plugging the previous  estimates into \eqref{com-C2-3} 
we find that for any $(\ep,\nu)\in A_0,$
\begin{align*}
     \| [Z^{\alpha}, u_3^R] \p_3 u_j^R\|_{L_t^1L^2}\lesssim 
     \min\{ \nu^{-\f{5}{12}}, \ep^{-\f12}+\nu^{-\f14}\} \eta^2 \cN_T^2\lesssim 
     (\ep^{\f{1}{20}}+\nu^{\f13})\eta \cN_T^2, 
\end{align*}
which, combined with \eqref{com-C2-1},\eqref{com-C2-2}, lead to that 
\begin{align}\label{es-cc2}
     \| \cC_2^{\alpha}\|_{L_t^1L^2}\lesssim \ep^{\min\{\f{1}{20},\f{\kappa}{3}\}}\eta \cN^2_T. 
\end{align}
We thus finish the estimate of $\cC_2^{\alpha}.$ 

$\bullet$ For $C_3^{\alpha},$
we have
\begin{align}
     \| \cC_3^{\alpha}\|_{L_t^1L^2}\lesssim T^{\f12}\|\na_h U^a, \div u^a\|_{L_t^{2}W_{tan}^{3,\infty}}\|u^R\|_{L_t^{\infty}\underline{H}_{tan}^3}\lesssim T^{\f12}\eta \Lambda(\cM_T^a)\cN_T.
\end{align}

$\bullet$ For $C_4^{\alpha},$ thanks to \eqref{rewrite-nasigmaR}, we get
\beq\label{es-cc4} 
\begin{aligned}
&\|\cC_4^{\alpha}\|_{L_t^1L^2}\lesssim T^{\f12}\|\sigma^a\|_{L_{t}^{\infty}W_{co}^{2,\infty}\cap L_t^{\infty}H_{co}^3}\|(\na\sigma^R,\div u^R)\|_{L_t^2H_{co}^2\cap L_t^2L^{\infty}}\\  &+\Lambda\big(\|\ep\sigma^a\|_{L^{\infty}_tW_{co}^{2,\infty}}+\|\ep Z\sigma^R\|_{L_{t,x}^{\infty}}\big) \\
&\qquad \cdot\big(\|\sigma^R\|_{L_t^2W_{co}^{1,\infty}\cap L_t^2W_{co}^{2,6}\cap L_t^2\underline{H}_{co}^3} \|(\na\sigma^R,\div u^R)\|_{L_t^2H_{co}^2\cap L_t^2W_{co}^{1,3}\cap L_t^2L^{\infty}}\big)\\
  &\lesssim \big(T^{\f12}+\eta
  +\ep^{\f14}\big)  \eta\Lambda\big(\cM_T^a,\cA_T\big)(\cN_T+\cN_T^2). 
\end{aligned}
\eeq

$\bullet$ For$\bigg|\iint_{Q_t} ([Z^{\alpha}, {1}/{\rho}]\div_{\nu}S  u^R)\cdot \rho^{-\f{\gamma-1}{2}} Z^{\alpha} u^R\bigg|$, we write
\beno
[Z^{\alpha}, {1}/{\rho}]\div_{\nu}S  u^R=Z^{\alpha}(1/\rho)\div_{\nu}S  u^R+[Z^{\alpha}, 1/\rho, \div_{\nu}S u^R].
\eeno
First, since 
\begin{align*}
\big\|Z^{\alpha}( {1}/{\rho})\div_{\nu}S  u^R\big\|_{L_t^1L^2}&\lesssim T^{\f12} \|\ep\, \div_{\nu} S u^R\|_{L_t^2L^{2}} \Lambda(\|\sigma^a\|_{L_t^{\infty}W_{co}^{3,\infty}})\\
&\quad + \|\ep\, \div_{\nu} S u^R\|_{L_t^2L^{\infty}}\|\sigma^R\|_{L_t^2\underline{H}_{tan}^3} \Lambda\big(\|\ep\sigma^a\|_{L^{\infty}_tW_{co}^{2,\infty}}+\|\ep Z_{tan}\sigma^R\|_{L_{t,x}^{\infty}}\big)\\
&\lesssim \big(T^{\f12}+ \nu^{-\f14}\eta\big)\eta \Lambda\big(\cM_T^a, \cA_T\big) \cN_T,
\end{align*}
\begin{align}\label{integ-1}
   \bigg|\iint_{Q_t} Z^{\alpha} ({1}/{\rho})\div_{\nu}S  u^R\cdot \, \rho^{-\f{\gamma-1}{2}}\,Z^{\alpha}u^R\bigg|\lesssim \big(T^{\f12}+ 
   \ep^{\f14}+\nu^{\f12}\big)\eta^2 \Lambda\big(\cM_T^a, \cA_T\big) \cN_T^2.
\end{align}
Second, denoting for short $\dot{v}^R=\rho^{-\f{\gamma-1}{2}}\,Z^{\alpha}u^R,$
we can integrate by parts in space to get that 
\begin{align*}
\bigg|\iint_{Q_t} &[Z^{\alpha}, {1}/{\rho},\div_{\nu}S  u^R]\cdot 
\,\dot{v}^R\bigg| \lesssim \nu^{\f12}
\|\p_z \dot{v}^R\|_{L_t^2L^2}
\nu^{\f12}\|[Z^{\alpha}, 1/\rho,\p_z u]\|_{L_t^2L^2}+
\|\na_h \dot{v}^R\|_{L_t^2L^2}\|[Z^{\alpha}, 1/\rho,\nabla_h u^R]\|_{L_t^2L^2}\\
&+\|\div \dot{v}^R\|_{L_t^2L^2}
\|[Z^{\alpha}, 1/\rho, \div u^R]\|_{L_t^2L^2}+\|\dot{v}^R\|_{L_t^{\infty}L^2}\|[Z^{\alpha}, \na(1/\rho), (\big(\na_{h},\nu\p_z) u^R;\div u^R)]
\|_{L_t^1L^2}.
\end{align*}
Denote $f=\nu^{\f12}\p_z u^R$ or $\na_h u^R$ or 
$\div u^R,$ it holds by counting the derivatives hitting on $1/\rho$ and $f$ that 
\begin{align*}
   \|[Z^{\alpha}, 1/\rho, f]\|_{L_t^2L^2}&\lesssim \ep \big(\|Z_{tan}\sigma^R\|_{L_{t,x}^{\infty}}+\Lambda(\|\sigma^a\|_{L_t^{\infty}W_{co}^{2,\infty}})\big)\|f\|_{L_t^2H_{tan}^2}\\
& +\|\ep f\|_{L_t^2W_{tan}^{1,\infty}}\|\sigma^R\|_{L_t^{\infty}\underline{H}_{tan}^2}\Lambda\big(
\|\ep Z_{tan}(\sigma^a, \sigma^R)\|_{L_{t,x}^{\infty}}
\big).
\end{align*}
In light of the definition of $\cE_T$  and
$\cA_T$ in \eqref{def-cET}, \eqref{def-cAT}, we find that  
\begin{align*}
   & \|[Z^{\alpha}, 1/\rho, f]\|_{L_t^2L^2}\lesssim 
   \eta\Lambda\big(\cM_T^a, \cA_T\big)(\ep^{\f12}   \cN_T+\ep^{\f12}\nu^{-\f14}\eta \cN_T^2)\lesssim \ep^{\f12}\eta \Lambda\big(\cM_T^a, \cA_T\big) 
   (\cN_T+\cN_T^2).
\end{align*}
Similarly, for $g=\na_{h}u^R,$ or $\nu\p_z u^R,$ or $\div u^R,$ if $\alpha_0\leq 1,$  
we use the following estimate
\beq
\begin{aligned}\label{commutator-g}
  & \|[Z^{\alpha}, \na(1/\rho), g]\|_{L_t^1L^2}\\
   &\lesssim \ep \big(\Lambda(\|\na\sigma^a\|_{L_t^{\infty}W_{co}^{2,\infty}})+\|\na\sigma^R\|_{L_t^2L^{\infty}}\big)\|g\|_{L_t^2H_{tan}^2}+  \|\na \sigma^R\|_{L_{t}^2W_{co}^{1,3}}\|\ep g\|_{L_t^2W_{tan}^{2,6}}\\
& \qquad+\|\ep g\|_{L_t^2W_{tan}^{1,\infty}}\big(\|\na \sigma^R \|_{L_t^{2}\underline{H}_{tan}^2}\Lambda\big(
\|\ep Z_{tan} (\sigma^a, \sigma^R)\|_{L_{t,x}^{\infty}}\big)
+ \|\sigma^R\|_{L_t^{\infty}\underline{H}_{tan}^2}\|\ep \na\sigma^R\|_{L_t^2L^{\infty}}\big)
\end{aligned}
\eeq
to get that
\begin{align*}
    \big\|[Z^{\alpha}, \na(1/\rho), \big(\na_{\nu} u^R,\div u^R)]\big\|_{L_t^1L^2}\lesssim (\ep^{\f13}+\eta)\eta \Lambda\big(\cM_T^a,\cA_T) (\cN_T+\cN_T^2).
\end{align*}
Note that we have abused slightly the notation, since when 
$g=\na_h\sigma^R,$ the second term in \eqref{commutator-g} should be replaced by 
$\| \|\na_h \sigma^R\|_{L_{t}^2W_{co}^{1,3}}\|\ep \na_h u^R\|_{L_t^2W_{tan}^{2,6}}\|\lesssim (\ep\nu^{-\f12}\eta)\eta\cN_T^2\lesssim \ep^{\f12}\eta\cN_T^2.$

However, when $\alpha_0=2,$ the previous arguments is not 
feasible since we do not have the estimate of 
$\|\p_z^2 (\ep\pt)^2\na_{h} u^R\|_{L_t^2L^2}.$ 
We thus instead deal with the problematic term $\ep\nu\iint_{Q_t} \na_h\p_z\sigma^R Z_0^2\p_z u^R \cdot \dot{v}^R$ as
\begin{align*}
\big|\ep\nu\iint_0 \na_{h}\p_z\sigma^R \cdot Z_0^2\p_z u^R\, \dot{v}^R\big|&\lesssim \nu^{\f12}\|\na\dot{v}^R\|_{L^2_tL^2} \nu^{\f12}\|Z_0^2\p_z u^R\|_{L^2_tL^2}\|\ep\na_h\p_z\sigma^R\|_{L_t^{\infty}L^{3}}\\
&\lesssim  
\ep^{\f12}(\eta+\ep^{\f14})\eta^2\cN_T^3\, .
\end{align*}
Consequently, we find that 
\begin{align*}
    \bigg|\iint_{Q_t} [Z^{\alpha}, {1}/{\rho},\div_{\nu}S  u^R]\cdot 
\,\dot{v}^R\bigg| \lesssim (\ep^{\f13}+\eta) \eta^2 \Lambda\big(\cM_T^a,\cA_T) (\cN_T^2+\cN_T^3),
\end{align*}
which, combined with \eqref{integ-1}, gives that, for any $\nu=\ep^{\kappa}$ with $0<\kappa<3,$
\begin{align}\label{es-integ}
    \bigg|\iint_{Q_t} [Z^{\alpha}, {1}/{\rho}]\div_{\nu}S  u^R\cdot Z^{\alpha} u^R\bigg| \lesssim (T^{\f12}+\ep^{\min\{\f14,\f{\kappa}{2}\}}) \eta^2 \Lambda\big(\cM_T^a,\cA_T) (\cN_T^2+\cN_T^3).
\end{align}
Plugging \eqref{es-cc1}, \eqref{es-cc2}-\eqref{es-cc4}, \eqref{es-integ}
into \eqref{EI-hightan}, we find 
\eqref{ES-hightan}.

\end{proof}

\subsection{Some estimates for the compressible part $(\na\sigma^R, \div u^R)$} In this step, we aim to prove some estimates for the compressible part $(\na\sigma^R, \div u^R),$ including the uniform estimates for the tangential derivatives: 
$\|(\na\sigma^R, \div u^R)\|_{L_t^2H_{tan}^2},$
and the estimate for $\ep^{\f12} \|\p_z\div u^R\|_{L_t^2H_{tan}^2}.$  

As the estimate of $\|\div u^R\|_{L_t^2H_{tan}^2}$ has been obtained in \eqref{ES-hightan},  let us first focus  on the estimate of $\|\p_z \sigma^R\|_{L_t^2H_{tan}^2}.$
We use the equation \eqref{CNS-remainder} to find the following equation for $\p_z \sigma^R$
\begin{align}\label{eq-pzsigma}
  \ep^2 \tilde{\mu}(\pt +u\cdot \na )\p_z \sigma^R+ \rho^{\gamma} \p_z\sigma^R=-\rho^{\f{\gamma+1}{2}} \ep \pt u_3^R+G_u+G_{\sigma},
\end{align}
where 
$\tilde{\mu}=\mu_2+\mu_1\nu, u=u^R+u^a$ and 
\beq\label{defGusigma}
\begin{aligned}
  &  G_u=-\rho^{\f{\gamma+1}{2}}\big(\ep u\cdot \na u_3^R+\ep u^R\cdot \na u_3^a+\ep R_u^a\big)+\ep\rho^{\f{\gamma-1}{2}}\big(\mu_1\nu\p_z\div_h u_h^R-\mu_1\Delta_h u_3^R-(\rho-1)\div_{\nu}\cS u^a\big), \\
   & G_{\sigma}=- \ep^2 \tilde{\mu}\big(\p_z u\cdot\na \sigma^R+\p_z(u^R\cdot\na\sigma^a)+\p_z R_{\sigma}^a\big).
\end{aligned}
\eeq
\begin{lem}\label{lem-pzsig-tan}
Under the same assumption as in Lemma \ref{lem-hightan}, we have that for any $(\ep,\nu)\in A_0,$ any $t\in[0,T],$
\beq
\begin{aligned}\label{es-nasigma-tan}
     &\quad  \ep^2 \|\p_z\sigma^R\|_{L_t^{\infty}H_{tan}^2}^2+\|\p_z\sigma^R\|_{L_t^2H_{tan}^2}^2\\
       &\lesssim  \ep^2 \|\sigma^R(0)\|_{H_{tan}^2}^2
       +\eta^2\Lambda(\cM_T^a)+(T^{\f12}
       +\ep^{\min\{\f{1}{20},\f{\kappa}{3},\f{3-\kappa}{4}\}}+\eta)\eta^2\Lambda(\cM_T^a)(\cN_T^2+\cN_T^4).
    \end{aligned}
    \eeq
\end{lem}
\begin{proof}
Let $\alpha \in \mathbb{N}^4,$ with $\alpha_3=0.$ 
Taking 
$Z^{\alpha}(|\alpha|\leq 2)$ on the equation \eqref{eq-pzsigma}, and multiplying the resultant equation by $Z^{\alpha}\p_z\sigma^R,$ integrating over $Q_t=[0,t]\times \Omega,$ we find, after suitable integration by parts, the following energy inequality
\beq\label{EI-sigma}
\begin{aligned}
  &  \ep^2 \tilde{\mu} \int_{\Omega} |\p_z Z^{\alpha}\sigma^R|^2(t) \, \d x +\iint_{Q_t} \rho^{\gamma}|\p_z Z^{\alpha}\sigma^R|^2 \, \d x \d s\\
    & \lesssim  \ep^2 \tilde{\mu} \int_{\Omega} |\p_z Z^{\alpha}\sigma^R|^2(0) \, \d x+\ep \|\div(u^a+u^R)\|_{L^2_tL^{\infty}}\|\ep \p_z\sigma^R\|_{L_t^{\infty}H_{tan}^2}\|\p_z\sigma^R\|_{L_t^2H_{tan}^2}\\
    &\, +\bigg|\iint_{Q_t}\rho^{\f{\gamma+1}{2}} \ep\pt Z^{\alpha}u_3^R Z^{\alpha}\p_z \sigma^R\, \d x \d s \bigg|
    + \|Z^{\alpha}\p_z\sigma^R\|_{L_t^2L^2} \big(\|(\cF_1^{\alpha},\cF_2^{\alpha},\cF_3^{\alpha})\|_{L_t^2L^2}+\|(G_u,G_{\sigma})\|_{L_t^2H_{tan}^2} \big)
\end{aligned}
\eeq
where  $$\cF_1^{\alpha}=\ep^2[Z^{\alpha}, u^a+u^R]\cdot \na\p_z\sigma^R,\,\,\quad  \cF_2^{\alpha}=[Z^{\alpha}, \rho^{\gamma}]\p_z\sigma^R,\, \, \quad \cF_3^{\alpha}=[Z^{\alpha}, \rho^{\f{\gamma+1}{2}}]\ep\pt u_3^R.$$

Let us first deal with the integrable term, which necessitates the careful treatment since we only have the control of two time derivatives. Integrating by parts first in time and then in space, we find that 
\beq\label{es-integral}
\begin{aligned}
&\bigg|\iint_{Q_t}\rho^{\f{\gamma+1}{2}} \ep\pt Z^{\alpha}u_3^R\cdot Z^{\alpha}\p_z \sigma^R\,  
\bigg|\leq \ep \bigg|\int_{\Omega} \rho^{\f{\gamma+1}{2}} Z^{\alpha}u_3^R\cdot Z^{\alpha}\p_z \sigma^R \,\big|_0^t\bigg|\\
&\qquad+\bigg| \iint_{Q_t}(\ep\pt(\rho^{\f{\gamma+1}{2}}) ) Z^{\alpha}u_3^R\cdot Z^{\alpha}\p_z \sigma^R\,\bigg|+\bigg| \iint_{Q_t}\p_z (\rho^{\f{\gamma+1}{2}}  Z^{\alpha}u_3^R)\cdot \ep\pt Z^{\alpha}\sigma^R\,\bigg|\\
&\leq \f{\ep^2\tilde{\mu}}{2} \|Z^{\alpha}\p_z \sigma^R(t)\|_{L^2}^2+\|\ep \p_z\sigma^R(0)\|_{H_{tan}^2}^2+C \|u_3^R\|_{L_t^{\infty}H_{tan}^2}^2+\ep\|u_3^R\|_{L_t^2H_{tan}^2}\|\p_z\sigma^R\|_{L_t^2H_{tan}^2}\\
&\qquad + \ep \|\na(\sigma^a,\sigma^R)\|_{L_t^2L^{\infty}} \|u_3^R\|_{L_t^{\infty}H_{tan}^2} \|\p_z\sigma^R\|_{L_t^2H_{tan}^2}+\|(\div u^R, \na_h u^R,\ep\pt \sigma^R)\|_{L_t^2H_{tan}^2}^2\\
&\leq  \f{\ep^2\tilde{\mu}}{2} \|Z^{\alpha}\p_z \sigma^R(t)\|_{L^2}^2+\|\ep \p_z\sigma^R(0)\|_{H_{tan}^2}^2+C \|u_3^R\|_{L_t^{\infty}H_{tan}^2}^2+C\|(\div u^R, \na_h u^R)\|_{L_t^2H_{tan}^2}^2\\
&\qquad+\ep\|R_{\sigma}^a\|_{L_t^2H_{tan}^2}^2+(\ep+\eta)  \eta^2\Lambda(\cM_T^a,\cA_T)\cN_T^2.
\end{aligned}
\eeq
Note that in the last inequality we have used the equation $\eqref{CNS-remainder}_1$ to have that
\begin{align*}
 \|\ep\pt \sigma^R\|_{L_t^2H_{tan}^2}\lesssim \|\div u^R\|_{L_t^2H_{tan}^2}+(\ep^{\f12}+\eta)\eta\Lambda(\cM_T^a,\cA_T)\cN_T+\ep\|R_{\sigma}^a\|_{L_t^2H_{tan}^2}. 
\end{align*}
Next, 
concerning $\cF_1^{\alpha}-\cF_3^{\alpha},$ we use the definitions of $\cE_T$ and $\cA_T$ (see \eqref{def-cET},\eqref{def-cAT}) to 
 verify that
\beq\label{es-cF1}
\begin{aligned}
  \|\cF_1^{\alpha}\|_{L_t^2L^2}&\lesssim \ep^2\|(u^a, \p_z u_3^a)\|_{L_t^{\infty}W_{tan}^{2,\infty}} \|\p_z\sigma^R\|_{L_t^2{H}_{co}^2}+\|\ep Z(u_h^R,\p_z u_3^R)\|_{L_t^2L^{\infty}}\|\ep\p_z\sigma^R\|_{L_t^{\infty}\underline{H}_{co}^2}\\
&\qquad+\ep^2\|u_h^R\|_{L_t^{2}W_{tan}^{2,\infty}}\|\p_z\sigma^R\|_{L_t^{\infty}H_{tan}^{1}}+\ep^2\|u_3^R\|_{L_t^{2}W_{tan}^{2,\infty}} \|\p_z^2\sigma^R\|_{L_t^{\infty}L^2}\\
  & \lesssim \ep^2\eta\Lambda(\cM_T^a)\cN_T+
  \big(
  \eta+
  \ep\eta\max\{1,({\ep^7}/{\nu^5})^{\f16}, (\ep^3/{\nu^2})^{\f12}\}\big)\eta\cN_T^2\, \\
  & \lesssim \big(\eta+\ep^{\f12})\eta(\Lambda(\cM_T^a)\cN_T+\cN_T^2\big),
\end{aligned}
\eeq
and
\beq\label{es-cF23}
\begin{aligned}
 \|(\cF_2^{\alpha},\cF_3^{\alpha})&\|_{L_t^2L^2}\lesssim
 \ep\|\sigma^R\|_{L_t^{\infty}H_{tan}^2}\|(\p_z\sigma^R,\ep\pt u_3^R)\|_{L_t^2L^{\infty}}\\
 & \quad+
(\ep+\|\ep Z\sigma^R\|_{L_{t,x}^{\infty}})\Lambda\big(\|\sigma^a\|_{L_t^{\infty}W_{co}^{2,\infty}}+\|\ep Z\sigma^R\|_{L_{t,x}^{\infty}}\big) \|(\p_z\sigma^R, \ep\pt u_3^R)\|_{L_t^{2}H_{tan}^1}\\
 &\lesssim (\ep+\eta) \eta\Lambda(\cM_T^a,\cA_T)\cN_T.
\end{aligned}
\eeq
It now remains to control $\|(G_u,G_{\sigma})\|_{L_t^2H_{tan}^2}.$ Checking each term in the definition \eqref{defGusigma}, one can find that 
\beq\label{es-Gusigma}
\begin{aligned}
 \|G_u\|_{L_t^2H_{tan}^2}\lesssim (\ep +\ep\nu^{-\f14}\eta)\eta\Lambda(\cM_T^a)(\cN_T+\cN_T^2)+\ep \|R_u^a\|_{L_t^2H_{tan}^2}+\eta \Lambda(\cM_T^a),\\
  \|G_{\sigma}\|_{L_t^2H_{tan}^2}\lesssim 
   ( \eta+\ep^{\f{3-\kappa}{4}})\Lambda(\cM_T^a)(\cN_T+\cN_T^2)+\ep^2\|\p_z R_{\sigma}^a\|_{L_t^2H_{tan}^2}.
\end{aligned}
\eeq
For instance, in view of the definitions \eqref{def-cET}, \eqref{def-cAT} and the estimate \eqref{pzuh-dinfty},
it holds that 
\begin{align*}
    \ep^2 \|\p_z u^a\cdot\na \sigma^R\|_{L_t^2H_{tan}^2}\lesssim \|\ep^2\p_z u^a \|_{L_t^{\infty}W_{tan}^{2,\infty}}\|\na\sigma^R\|_{L_t^2H_{tan}^2}\lesssim (\ep^3/\nu)^{\f12}\Lambda(\cM_T^a)\eta\cN_T, 
    \end{align*}
    \begin{align*}
     \ep^2 \|\p_z u^R\cdot\na\sigma^R\|_{L_t^2H_{tan}^2}&\lesssim \|\ep^2\p_z u^R \|_{L_{t,x}^{\infty}}\|\na\sigma^R\|_{L_t^2H_{tan}^2}+\|\ep^2\p_z u_h^R \|_{L_{t}^{2}W_{tan}^{1,\infty}}\|\sigma^R\|_{L_t^{\infty}H_{tan}^{2}}\\
     &\quad +\|\ep\p_z u_3^R \|_{L_{t}^{2}W_{tan}^{1,\infty}}\|\ep\p_z\sigma^R\|_{L_t^{2}H_{tan}^{1}}+\|\ep^2\p_z u^R \|_{L_{t}^{\infty}H_{tan}^2}\|\na\sigma^R\|_{L_t^2L^{\infty}}\\
     &\lesssim \big(\eta+\ep^{\f{3-\kappa}{4}}+(\ep^2/\nu)^{\f34}\eta+\ep\eta(\ep^{-\f14}+\nu^{-\f14})\big) \eta\cN_T^2
     \lesssim  
    ( \eta+\ep^{\f{3-\kappa}{4}})\eta\cN_T^2. 
\end{align*}
Collecting \eqref{es-integral}-\eqref{es-Gusigma} 
and using Young's inequality, we get from \eqref{EI-sigma} that 
\begin{align*}
   & \ep^2 \|\p_z\sigma^R\|_{L_t^{\infty}H_{tan}^2}^2+\|\p_z\sigma^R\|_{L_t^2H_{tan}^2}^2\lesssim  \ep^2 \|\sigma^R(0)\|_{H_{tan}^2}^2+\ep\|(R_{\sigma}^a, R_u^a, \ep\p_zR_{\sigma}^a)\|_{L_t^2H_{tan}^2}^2\\
    &+\|u_3^R\|_{L_t^{\infty}H_{tan}^2}^2+\|(\div u^R, \na_h u^R)\|_{L_t^2H_{tan}^2}^2+(\ep^{\min\{\f{3-\kappa}{4}, \f12\}}+\eta)\eta^2\Lambda(\cM_T^a)(\cN_T^2+\cN_T^4),
\end{align*}
which, combined with \eqref{ES-hightan}, leads to \eqref{es-nasigma-tan}. 
\end{proof}
We also need the following estimate for $\ep^{\f12}\|\p_z\div u^R\|_{L_t^2\underline{H}_{tan}^2},$ which will be useful in the next step for the estimate of the vorticity.
\begin{lem}
Under the same assumption as in Lemma \ref{lem-hightan}, 
it holds that for any $(\ep,\nu)\in A_0,$  any $t\in[0,T],$  
\beq\label{pzdiv-tangential}
\begin{aligned}
  &  \ep\big(\|\div  u^R\|_{L_t^{\infty}\underline{H}_{tan}^2}^2+\|\p_z\div u^R\|_{L_t^2\underline{H}_{tan}^2}^2\big) \\
    & \lesssim \eta^2\cN_0^2+  \eta^2\Lambda(\cM_T^a) +\big(T^{\f12}+
    \ep^{\min\{\f{1}{20},\f{\kappa}{3},\f{3-\kappa}{4}\}}+\eta\big)\eta^2\Lambda(\cM_T^a)(\cN_T^2+\cN_T^4). 
    \end{aligned}
    \eeq
\end{lem}
\begin{proof}
To obtain such an estimate, we use the equation for $u_3^R:$ 
\begin{align*}
 & \big(\displaystyle\pt+ (u^{a}+u^R)\cdot\na\big) u_3^{R}+u^R\cdot\na u_3^a+
\f{\rho^{\f{\gamma-1}{2}}\p_z \sigma^{R}}{\ep}-\f{1}{\rho}\, \big(\tilde{\mu}\p_z\div u^R+\mu_1\Delta_h u_3^{R}-\mu_1\nu\p_z\div_h u_h^R\big)+(R^a_{u})_3\\
&=\big(\f{1}{\rho}-1\big)(\div_{\nu} S u^a)_3\,. 
\end{align*}
 Let $\alpha$ be a multi-index with $\alpha_0\leq 1, \alpha_3=0.$
 Taking $Z^{\alpha}$ on the above equation and taking the $L^2(Q_t)$ product of the above equation with  
 $-\ep\p_zZ^{\alpha}\div u^R,$ we obtain after suitable integration by parts and the application of the Young's inequality that   
\begin{align*}
\ep \big(\|\div u^R\|_{L_t^{\infty}\underline{H}_{tan}^2}^2&+\|\p_z\div u^R\|_{L_t^2\underline{H}_{tan}^2}^2\big) \lesssim \|\div u^R(0)\|_{\underline{H}_{tan}^2}^2 +\|\div_h u_h^R\|_{L_t^2\underline{H}_{tan}^3}^2+ \|(\p_z\sigma^R, \div u^R)\|_{L_t^2\underline{H}_{tan}^2}^2\\
&+ \ep \|((u^{a}+u^R)\cdot\na  u_3^{R}, u^R\cdot\na u_3^a, \Delta_h u_3^R, \nu\p_z\div_h u_h^R)\|_{L_t^2\underline{H}_{tan}^2}^2+\|[Z^{\alpha}, \rho^{\f{\gamma-1}{2}}]\p_z\sigma^R\|_{L_t^2\underline{H}_{tan}^2}^2\\
&+\ep \|(R^a_{u})_3\|_{L_t^2\underline{H}_{tan}^2}^2+\ep\|\big({1}/{\rho}-1\big)(\div_{\nu} S u^a)_3\|_{L_t^2\underline{H}_{tan}^2}^2. 
\end{align*}
We can readily check that the second line can be controlled by $(\ep+\eta) \eta^2\Lambda(\cA_T, \cM_T^a)(\cN_T^2+\cN_T^4),$ while the last term is controlled by $\ep \eta^2\Lambda(\cM_T^a, \cA_T)(1+\cN_T^2).$ The estimate \eqref{pzdiv-tangential} then follows from \eqref{ES-hightan} and \eqref{es-nasigma-tan}.

\end{proof}

\subsection{Energy estimates for the weight vorticity and horizontal components of the vorticity $\omega_h$}
In this subsection, we prove some estimates 
for the vorticity, including the uniform estimates
for $\phi \,\omega^R$  and non-uniform estimates for $\omega_h^R.$

First, we recover the uniform estimate for the vorticity $\omega^R$ with weight, which gives the (uniform) estimate for $Z_3 u_h^R.$

\begin{lem}\label{lem-vor-uniform}
    Under the same assumption as in Lemma \ref{lem-hightan}, it holds that for any $t\in[0, T],$
 \beq\label{es-vor-weight}
\begin{aligned}
\|\phi\, \omega^R\|_{L_t^{\infty}H_{co}^2}^2+ \|\phi\na_{\nu}\omega^R\|_{L_t^{2}H_{co}^2}^2&\lesssim
 \|\phi\, \omega^R(0)\|_{L_t^{\infty}H_{tan}^2}^2+ 
 \eta^2
 \Lambda(\cM_T^a)\\
     &\quad +\eta^2(T^{\f12}+\ep^{\min\{\f{1}{40},\f{\kappa}{3},\f{3-\kappa}{4}\}})
)\Lambda\big(\cM_T^a,\cA_T\big)\big(\cN_T^2+\cN^4_T\big),
\end{aligned}
\eeq
where we denoted for simplicity $\na_{\nu}=(\p_{x_1},\p_{x_2}, \nu^{\f12}\p_z ).$
\end{lem}
\begin{proof}
    Taking the curl of the equation $\eqref{CNS-remainder}_2$ for the velocity, we find that $\omega^R=\curl u^R$ solves the equation
\begin{align}\label{eq-omega}
\big(\pt +(u^a+u^R)\cdot\na\big)\omega^R-\f{\mu_1}{\rho}\big(\Delta_h+\nu\p_z^2\big)\omega^R
=F-\curl R_u^a
\end{align}
where $F=F_1+F_2+F_3+F_4$ with
\beq\label{def-F}
\begin{aligned}
    &F_1=-\omega^a\div u^R+\omega^a\cdot\na u^R-u^R\cdot\na\omega^a+\omega^R\cdot \na u^a-\omega^R\div u^a, \\
 & F_2=\omega^R\cdot \na u^R+\omega^R\div u^R , \qquad F_3=\ep^{-1}\na (\rho^{-\f{\gamma-1}{2}})\times \na\sigma^R+\na(\rho^{-1})\times \div_{\nu} S u^R,\, \\
 & F_4=\na\rho \times \div_{\nu} S u^a+\mu_1(\rho-1)(\Delta_h+\nu\p_z^2)\omega^a.
\end{aligned}
\eeq
Let $\alpha\in \mathbb{N}^4$ be a multi-index with $|\alpha|\leq 2.$ Applying $ Z^{\alpha}$ on the equation \eqref{eq-omega}, we obtain the following system
\begin{align*}
&\big(\pt +(u^a+u^R)\cdot\na\big)(Z^{\alpha}\omega^R)-
\mu_1\big(\div_h Z^{\alpha} (\rho^{-1}\na_h \omega^R)+\nu Z^{\alpha}\p_z (\rho^{-1}\p_z \omega^R)\big)
\\
&=Z^{\alpha}F+D_{1}^{\alpha}+D_{2}^{\alpha}, 
\end{align*}
where 
\begin{align}\label{def-D12alpha}
  &  D_1^{\alpha}=-[Z^{\alpha}, (u^a+u^R)\cdot\na]\omega^R,\quad 
 D_2^{\alpha}=-\mu_1 Z^{\alpha}\big(\big(\na_h (\rho^{-1})\cdot\na_h+\nu\p_z(\rho^{-1})\p_z\big)\omega^R\big).
\end{align}
Taking the scalar product of the above system with $\phi^2 Z^{\alpha}\omega^R$ and integrating over $Q_t=[0,t]\times \Omega,$ we find the following energy equality 
\beq\label{EI-vor}
\begin{aligned}
&\f12\int_{\Omega} |\phi \, Z^{\alpha}\omega^R|^2(t)\, \d x+\mu_1 \iint_{Q_t} \rho^{-1}\big( |\phi \, \na_h  Z^{\alpha}\omega^R|^2+\nu |\phi \, \p_z Z^{\alpha}\omega^R|^2\big)\, \d x\d s \\
&=\f12 \int_{\Omega} |\phi \, Z^{\alpha}\omega^R|^2(0)\, \d x+ \f12 \iint_{Q_t} \big(\div(u^R+u^a)+2 \p_z\phi (u_3^a+u_3^R)/\phi \big)|\phi \, Z^{\alpha}\omega^R|^2\, \d x\d s\\
&-\mu_1\nu \iint_{Q_t} \rho^{-1}\big(2\phi \p_z\phi\, Z^{\alpha}\p_z\omega^R\cdot Z^{\alpha}\omega^R
+\phi^2[Z^{\alpha},\p_z]\omega^R\cdot \, \p_zZ^{\alpha}\omega^R\big)\, \d x \d s\\
&\, +\mu_1\nu \iint_{Q_t} \phi^2 [Z^{\alpha},\p_z](\rho^{-1}\p_z\omega^R)\cdot Z^{\alpha}\omega^R\, \d x \d s -\mu_1\iint_{Q_t} \phi^2 [Z^{\alpha}, \rho^{-1}]\na_{\nu} \omega^R \cdot \na_{\nu} Z^{\alpha} \omega^R \, \d x\d s\\
&\, + \iint_{Q_t} \phi^2 \big(Z^{\alpha}F+Z^{\alpha}\curl R_u^a+D_{1}^{\alpha}+D_{2}^{\alpha} \big) Z^{\alpha}\omega^R\d x \d s := \sum_{k=1}^6 \cJ_k^{\alpha}.
\end{aligned}
\eeq

The second term in the right hand side $\cJ_2^{\alpha}$ can be estimated simply by 
\begin{align*}
  |\cJ_2^{\alpha}|\lesssim T^{\f12} \|(\div u^a , \div u^R, \p_z u_3^a, \p_z u_3^R)\|_{L_t^2L^{\infty}} \|\phi\,\omega^R\|_{L_{t}^{\infty}H_{co}^2}^2\lesssim T^{\f12}\eta^2 (\nu^{-\f14}\eta\cN_T+\cM_T^a)\cN_T^2.
\end{align*}
For the next term, we use the following algebraic identity 
\begin{align}\label{algebraicid}
[Z_{\tan}, \p_z]=0, \qquad [Z_3,\p_z]=-\p_z\phi \p_z, \qquad [Z_3^2,\p_z]=-2\p_z\phi\p_z Z_3+\big((\p_z\phi)^2-Z_3\p_z\phi\big)\p_z
\end{align}
and the 
Young's inequality to find 
\begin{align}\label{es-cJ2}
    |\cJ_3^{\alpha}|\leq \f{\mu_1\nu}{100}\|\rho^{-\f12} Z_3 Z^{\alpha}\omega^R\|_{L_t^2L^2}^2+ C\nu \big(\|Z^{\alpha}\omega^R\|_{L_t^2L^2}^2+\mathbb{I}_{\{\alpha_3\geq 1\}}\sum_{\tilde{\alpha}_3\leq \alpha_3-1}\|Z_3Z^{\tilde{\alpha}}\omega^R\|_{L_t^2L^2}^2\big).
\end{align}
Similarly, for the term $\cJ_4^{\alpha},$ we use the identity  \eqref{algebraicid} and integration by parts to obtain 
that 
\begin{align*}
      |\cJ_4^{\alpha}|&\leq \f{\mu_1\nu}{100}\| Z_3 Z^{\alpha}\omega^R\|_{L_t^2L^2}^2 + C \nu \mathbb{I}_{\{\alpha_3\geq 1\}}\sum_{\tilde{\alpha}_3\leq \alpha_3-1} \|\phi Z^{\tilde{\alpha}}(\rho^{-1}\p_z\omega^R)\|_{L_t^2L^2}^2 \\
    &  \leq \f{\mu_1\nu}{100}\| Z_3 Z^{\alpha}\omega^R\|_{L_t^2L^2}^2+ C \nu \mathbb{I}_{\{\alpha_3\geq 1\}}\sum_{\tilde{\alpha}_3\leq \alpha_3-1} \|Z_3 Z^{\tilde{\alpha}}\omega^R\|_{L_t^2L^2}^2+C\eta^2\big(\ep\Lambda(\cM_T^a)+\eta\cN_T\big)^2\cN_T^2 .
\end{align*}

We now deal with the next term $\cJ_5^{\alpha}.$ First, 
by counting the derivatives hitting on each element, one has,
for any $1\leq |\alpha|\leq 2,$
\begin{align*}
    \|\phi [Z^{\alpha}, \rho^{-1}]\na_{\nu} \omega^R \|_{L_t^2L^2}&\lesssim \big( \ep \|\phi \na_{\nu}\omega^R\|_{L_t^2H_{co}^{1}}\|\sigma^a\|_{L_t^{\infty}W_{co}^{2,\infty}} +   
    \ep \|\phi \na_{\nu}\omega^R\|_{L_t^2L^{\infty}}\|\sigma^R\|_{L_t^{\infty}H_{co}^2}\big)(1+\|(Z\sigma^a,\ep Z\sigma^R)\|_{L^{\infty}_{t,x}})  \\
    &\lesssim \eta\Lambda(\cM_T^a)\big(\ep \cN_T +\ep^{\f12}\nu^{-\f14}\eta\cN_T^2 \big).
\end{align*}
Note that by the Sobolev embedding,  
\begin{align*}
    \|\phi \na_{\nu}\omega^R\|_{L_t^2L^{\infty}}\lesssim \|Z_3 \na_{\nu}\omega^R\|_{L_t^2H_{co}^1}^{\f12}\|\phi\na_{\nu}\omega^R\|_{L_t^2H_{co}^2}^{\f12}\lesssim \ep^{-\f12}\nu^{-\f14}\eta \cN_T.
\end{align*}
It then follows again from the Young's inequality that
\begin{align*}
 |\cJ_5^{\alpha}|\leq \f{\mu_1}{100}\|\rho^{-\f12} \phi \na_{\nu} Z^{\alpha}\omega^R\|_{L_t^2L^2}^2+ C\ep \eta^2 \Lambda(\cM_T^a) (\cN_T^2+\cN_T^4).
\end{align*}
Finally, the last term $\cJ_6^{\alpha}$ can be bounded by
\beqs 
|\cJ_6^{\alpha}|\lesssim \big(\|R_u^a\|_{L_t^1\underline{H}_{co}^3}+ \|\phi(Z^{\alpha} F, D_1^{\alpha}, D_2^{\alpha} )\|_{L_t^1L^2}\big) \|\phi Z^{\alpha}\omega^R\|_{L_t^{\infty}L^2},
\eeqs
which, combined with \eqref{es-FD12} proved in the  Lemma \ref{lem-F-D} and the Young's inequality, 
gives that 
\beq
|\cJ_6^{\alpha}|\leq \f{1}{4}\|\phi Z^{\alpha}\omega^R\|_{L_t^{\infty}L^2}+\|R_u^a\|_{L_t^1\underline{H}_{co}^3}^2+ \eta^2 \cN_T^2 
 \Lambda(\cM_T^a, \cA_T)  \big(T^{\f12}+
(T^{\f12}+
    \ep^{\min\{\f{1}{40},\f{\kappa}{3}
    \}})\cN_T\big).
\eeq
Plugging the previous estimates for $\cJ_2-\cJ_6$ 
into the energy equality \eqref{EI-vor},
we find that, under the assumption \eqref{epsigmaLinfty} for any $|\alpha|\leq 2,$
\beq\label{es-vor-weight-1}
\begin{aligned}
   & \|\phi Z^{\alpha} \omega^R\|_{L_t^{\infty}L^2}^2+\|\phi\na_{\nu} Z^{\alpha}\omega^R\|_{L_t^2L^2}^2\\
    &\lesssim 
     \|\phi Z^{\alpha} \omega^R(0)\|_{L^2}^2+\|R_u^a\|_{L_t^1\underline{H}_{co}^3}^2+
     (T^{\f12}+
    \ep^{\min\{\f{1}{40},\f{\kappa}{3},\f{3-\kappa}{4}\}})\Lambda(\cM_T^a, \cA_T)\big(\cN_T^2+\cN_T^4\big) \\
     &\quad +  \nu \big(\|Z^{\alpha}\omega^R\|_{L_t^2L^2}^2+\mathbb{I}_{\{\alpha_3\geq 1\}}\sum_{\tilde{\alpha}_3\leq \alpha_3-1}\|Z_3Z^{\tilde{\alpha}}\omega^R\|_{L_t^2L^2}^2\big).
\end{aligned}
\eeq
Since we do not have the a-priori control of $\|Z^{\alpha}\omega^R\|_{L_t^2L^2}$ unless $\alpha_3=0,$ 
to close the estimate, we do inductions for the amount of weighted normal derivative $Z_3.$ First, applying the above estimate for those $\alpha$ with $\alpha_3=0,$ we get thanks to \eqref{ES-hightan} that 
\begin{align*}
     \|\phi \, \omega^R\|_{L_t^{\infty}H_{tan}^2}^2+\|\phi\na_{\nu} \omega^R\|_{L_t^2H_{tan}^2}^2&\lesssim 
     \|\phi\, \omega^R(0)\|_{L_t^{\infty}H_{tan}^2}^2+\|R^a\|_{L_t^1\underline{H}_{co}^3}^2\\
     &\quad +\eta^2(T^{\f12}
   +\ep^{\min\{\f14,\f{5\kappa}{12},\f{3-\kappa}{4}\}})\Lambda\big(\cM_T^a,\cA_T\big)\big(\cN_T^2+\cN^4_T\big).
\end{align*}
Note that from this estimate we get the control $\nu^{\f12}\|Z_3 \omega^R\|_{L_t^2H_{tan}^2},$ which enables us to apply \eqref{es-vor-weight-1}  to obtain the estimate of $\|\phi Z_3 \omega^R\|_{L_t^{\infty}H_{tan}^1}, \|\nu^{\f12}Z_3^2 \omega^R\|_{L_t^2H_{co}^2}.$ Applying  \eqref{es-vor-weight-1} again gives the control of $\|\phi Z_3^2\omega^R\|_{L_t^{\infty}L^2}, \|\na_{\nu}Z_3^2 \omega^R\|_{L_t^2L^2}.$ We thus finish the proof of \eqref{es-vor-weight}.
\end{proof}

Next, we aim to prove some non-uniform estimates 
for horizontal components of the vorticity $\omega_h^R:$ 
\begin{lem}
    Under the same assumption as in Lemma \ref{lem-hightan}, we have that for any $t\in[0,T],$
\beq\label{es-omegah-weight}
\begin{aligned}
&\min\{(\ep^2\nu)^{\f23}, \ep\nu\}\big(\|\omega_h^R\|_{L_t^{\infty}H_{co}^1}^2+\|\na_{\nu}\omega_h^R \|_{L_t^{2}{H}_{co}^1}^2\big)+\ep^2\nu\big(\|\omega_h^R\|_{L_t^{\infty}\underline{H}_{co}^2}^2  + \|\na_{\nu}\omega_h^R \|_{L_t^{2}\underline{H}_{co}^2}^2\big)\\
&\lesssim \eta^2\big( \cN_0^2+ 
 \Lambda(\cM_T^a)\big) +\eta^2(T^{\f12}+\ep^{\min\{\f{1}{40},\f{\kappa}{3},\f{3-\kappa}{4}\}})\Lambda\big(\cM_T^a,\cA_T\big)\big(\cN_T^2+\cN^4_T\big),
\end{aligned}
\eeq
where we denote as before $\na_{\nu}=(\p_{x_1},\p_{x_2}, \nu^{\f12}\p_z).$ 
\end{lem}
\begin{proof}
The proof follows from direct energy estimates for $\omega_h$ which solve the first two equations of \eqref{eq-omega}. By using the equation for the velocity $u_h$, the boundary condition $u_h^R|_{\p\Omega}=u_h^a|_{\p\Omega}=0,$
we have the following boundary condition for $\omega_h:$
\begin{align}\label{bdrycond-omegah}
\mu_1\nu\p_z(\omega_h^R)^{\perp}=\rho^{-\f{\gamma-1}{2}} \na_h{\sigma}^R/\ep-\tilde{\mu}\na_h\div u^R+(R_{u}^{a})_h \qquad \text{ on }\, \p\Omega
\end{align}
where $\tilde{\mu}=\mu_1\nu+\mu_2,\, (\omega_h^R)^{\perp}=(\omega_2^R,-\omega_1^R).$ 

Let $\alpha\in \mathbb{N}^4$ be a multi-index with $\alpha_0\leq 1, |\alpha|\leq 2,$ and denote $\beta=\beta(\alpha)=\min\{(\ep^2\nu)^{\f13}, (\ep\nu)^{\f12}\}$ if $|\alpha|\leq 1$ and $
\beta=\ep \nu^{\f12}$ if $|\alpha|=2.$ 
Taking $Z^{\alpha}$ on the first two equations of \eqref{eq-omega}, testing by $\beta^2 Z^{\alpha}\omega_h^R$ and using the boundary condition \eqref{bdrycond-omegah}, we get the following identity which is similar to  \eqref{EI-vor},
\beq\label{EI-vor-h}
\begin{aligned}
&\f{\beta^2}{2}\int_{\Omega} | Z^{\alpha}\omega_h^R|^2(t)\, \d x+\beta^2 \mu_1 \iint_{Q_t} \rho^{-1} |\na_{\nu}  Z^{\alpha}\omega_h^R|^2
\, \d x\d s \\
&=\f{\beta^2}{2} \int_{\Omega} |Z^{\alpha}\omega_h^R|^2(0)\, \d x+ \f{\beta^2}{2} \iint_{Q_t} \div(u^R+u^a)| Z^{\alpha}\omega_h^R|^2\, \d x\d s\\
&\,\,-\beta^2 \mu_1\nu \iint_{Q_t} \rho^{-1}
[Z^{\alpha},\p_z]\omega_h^R\cdot \, \p_zZ^{\alpha}\omega_h^R \, \d x \d s  +\beta^2\mu_1\nu \iint_{Q_t}  [Z^{\alpha},\p_z](\rho^{-1}\p_z\omega_h^R)\cdot Z^{\alpha}\omega_h^R\, \d x \d s\\
&\, \, -\beta^2\mu_1\iint_{Q_t} [Z^{\alpha}, \rho^{-1}]\na_{\nu} \omega_h^R \cdot \na_{\nu} Z^{\alpha} \omega_h^R \, \d x\d s +\beta^2\iint_{Q_t}  \big(Z^{\alpha}F_h+Z^{\alpha}(\curl R_u^a)_h+D_{1h}^{\alpha}+D_{2h}^{\alpha} \big) \cdot Z^{\alpha}\omega_h^R\d x \d s\\
&\, \,+\beta^2\mu_1\nu \mathbb{I}_{\{\alpha_3=0\}} \int_0^t\int_{\p\Omega} \rho^{-1} Z^{\alpha}\p_z\omega_h^R\cdot Z^{\alpha}\omega_h^R \, \d x_h\, \d s  := \sum_{k=1}^7 \cG_k^{\alpha}
\end{aligned}
\eeq
where we denote $f_h$ the first two components of a three dimensional vector $f.$ 
Following the similar (and easier) computations as we performed in the proof of Lemma \ref{lem-vor-uniform}, 
one can verify that 
\begin{align*}
\sum_{k=1}^6 |\cG_k^{\alpha}|&\leq \f{\mu_1}{100} \beta^2(\alpha) \|\na_{\nu}Z^{\alpha}\omega_h^R\|_{L_t^2L^2}^2+
\beta^2(\alpha) \|Z^{\alpha}\omega_h^R(0)\|_{L_t^2L^2}^2+C\beta^2(\alpha)\|\na_{\nu} \omega_h^R\|_{L_t^2H_{co}^{|\alpha|-1}}^2 \\
&+\beta^2(\alpha)\|\na R_{u}^a\|_{L_t^1{H}_{co}^{|\alpha|}}^2+\eta^2(T^{\f16}+\ep^{\f12} +\nu^{\f{1}{12}})\Lambda\big(\cM_T^a,\cA_T\big)\big(\cN_T^2+\cN^4_T\big).
\end{align*}
For instance, when $|\alpha|=2,$ we can deal with $\cG_5^{\alpha}$ as 
\begin{align*}
    \ep^2\nu \mu_1\iint_{Q_t} &[Z^{\alpha}, \rho^{-1}]\na_{\nu} \omega_h^R \cdot \na_{\nu} Z^{\alpha} \omega_h^R \, \d x\d s \leq \f{\mu_1}{200}\ep^2\nu \|\na_{\nu}Z^{\alpha}\omega_h^R\|_{L_t^2L^2}^2 \\
    & +C\ep^2\Lambda(\cM_T^a, \cA_T)\|\sigma^R\|_{L_t^{\infty}H_{co}^2}^2\|\ep\nu^{\f12} \na_{\nu} \omega_h^R\|_{L_t^2L^{\infty}}^2+C\ep^2\nu \Lambda(\cM_T^a,\cA_T) \|\na_{\nu}\omega_h^R\|^2_{L_t^2H_{co}^{1}} \\
    &\leq  \f{\mu_1}{200}\ep^2\nu \|\na_{\nu}Z^{\alpha}\omega_h^R\|_{L_t^2L^2}^2+ C(\ep^2+\nu^{\f13})\eta^2 \Lambda(\cM_T^a,\cA_T) (\cN_T^2+\cN_T^4).
\end{align*}
Note that it holds 
by the definition of the $\cA_T$ (see \eqref{def-cAT}) that 
\begin{align*}
    \|\ep\nu^{\f12} \na_{\nu} \omega_h^R\|_{L_t^2L^{\infty}}\lesssim \|\ep\nu\p_z\omega_h^R\|_{L_t^2L^{\infty}}+\|\ep\nu^{\f12}\omega_h^R\|_{L_t^2W_{tan}^{1,\infty}}\lesssim\nu^{-\f14}\eta \cA_T\lesssim  \nu^{-\f14}\eta \cN_T.
\end{align*}
Moreover, it follows from Lemma \ref{lem-F-D} that the term $\cG_6^{\alpha}$ can be bounded as
\begin{align*}
    \cG_6^{\alpha}&\lesssim \|\beta \omega_h^R\|_{L_t^{\infty}H_{co}^2}\, \beta\|(Z^{\alpha}(F_h, \curl R_u^a), D_{1h}^{\alpha}, D_{2h}^{\alpha})\|_{L_t^1L^2}\\
    &\leq \f{1}{2}\|\beta \omega_h^R\|_{L_t^{\infty}H_{co}^2}^2+C\beta^2\|\na R_u^a\|_{L_t^1H^{|\alpha|}}^2+\eta^2 \cN_T^2 \Lambda(\cM_T^a, \cA_T)  \big(T^{\f12}+(T^{\f12}+
    \ep^{\min\{\f14,\f{5\kappa}{12},\f{3-\kappa}{4}\}})\cN_T\big).
\end{align*}
It now remains to bound the boundary term $\cG_7^{\alpha},$ which we need to use the boundary condition \eqref{bdrycond-omegah}.
On the one hand, when $|\alpha|\leq 1,$ 
we use Cauchy-Schwarz  inequality and the trace inequality to find that
\begin{align*}
   | \cG_7^{\alpha}|&\lesssim \beta^2 \ep^{-1}|Z^{\alpha}\omega_h^R|_{L_t^2L^2(\p\Omega)} |\ep\nu \p_z\omega_h^R|_{L_t^2H^{1}(\p\Omega)} \lesssim  \beta^{\f32} (\ep^{2}\nu)^{-\f12} \|\beta\nu^{\f12}Z^{\alpha}\na \omega_h^R\|_{L_t^2L^2}^{\f12}\|\nu^{\f12}\omega_h^R\|_{L_t^2H_{tan}^1}^{\f12}  \\
   & \cdot\big( \Lambda(\cM_T^a, \cA_T) T^{\f14}\|(\sigma^R, \ep\div u^R)\|_{L_t^{\infty}H_{tan}^2}^{\f12}\|\na(\sigma^R, \ep\div u^R)\|_{L_t^2H_{tan}^2}^{\f12}+\ep \|R_u^{a}\|_{L_t^{2}H_{tan}^1}^{\f12}\|\na R_u^{a}\|_{L_t^{2}H_{tan}^1}^{\f12}\big)\\
    &\leq  
    \f{\mu_1}{4} \|\beta\nu^{\f12}Z^{\alpha}\na \omega_h^R\|_{L_t^2L^2}^2+\|\nu^{\f12}\na u^R\|_{L_t^2H_{tan}^1}^2+\ep^2\|R_u^{a}\|_{L_t^{2}H_{tan}^1}\|\na R_u^{a}\|_{L_t^{2}H_{tan}^1}+ 
    T^{\f14}\eta^2 \Lambda(\cM_T^a, \cA_T) \cN_T^2
\end{align*}
since $\beta\leq (\ep^{2}\nu)^{\f13}.$
On the other hand, when $|\alpha|=2,$ we use the duality between $H^{\f12}(\p\Omega)$ and $H^{-\f12}(\p\Omega)$ to obtain that 
\begin{align*}
 | \cG_7^{\alpha}|&\lesssim \ep\nu|Z^{\alpha}\omega_h^R|_{L_t^2H^{\f12}(\p\Omega)} |\ep\nu \p_z\omega_h^R-\ep(R_{u}^a)_h|_{L_t^2H^{\f32}(\p\Omega)}+ \ep^2\nu|Z^{\alpha}\omega_h^R|_{L_t^2L^2(\p\Omega)}|(R_{u}^a)_h|_{L_t^2H^2(\p\Omega)}\\
 &\leq \f14 \|\ep\nu\na Z^{\alpha}\omega_h^R\|_{L_t^2L^2}^2+(\|\na(\sigma^R, \ep\div u^R)\|_{L_t^2\underline{H}_{tan}^2}^2+\|\nu^{\f12}\na u^R\|_{L_t^2\underline{H}_{tan}^2}^2+\ep^3\nu^{\f12} \|R_u^{a}\|_{L_t^{2}\underline{H}_{tan}^2}\|\na R_u^{a}\|_{L_t^{2}\underline{H}_{tan}^2}).
\end{align*}
Plugging the estimates for $\cG_1^{\alpha}-\cG_7^{\alpha}$ into \eqref{EI-vor-h}, and utilizing the estimate \eqref{ES-hightan} for $\nu^{\f12}\|\na u^R\|_{L_t^{2}\underline{H}_{tan}^2},$ \eqref{es-nasigma-tan} for $\|\na\sigma^R\|_{L_t^{2}\underline{H}_{tan}^2},$
we get \eqref{es-omegah-weight} by using Young's inequality and induction.

\end{proof}

\begin{lem}\label{lem-F-D}
Let $F, D_1^{\alpha}, D_2^{\alpha}$ be defined in \eqref{def-F},\eqref{def-D12alpha}. Recall that $\beta(\alpha)=\min\{(\ep^2\nu)^{\f13}, (\ep\nu)^{\f12}\}$ if $|\alpha|\leq 1$ and $
\beta=\ep \nu^{\f12}$ if $|\alpha|=2.$
Under the same assumption as in  Lemma \ref{lem-hightan}, it holds that for $|\alpha|\leq 2,$
\beq\label{es-FD12}
\begin{aligned}
&\|\phi F\|_{L_t^1\underline{H}_{co}^2}+\|\phi (D_1^{\alpha}, D_2^{\alpha})\|_{L_t^1L^2} 
+\beta(\alpha)\| (Z^{\alpha} F_h, D_{1h}^{\alpha}, D_{2h}^{\alpha})\|_{L_t^1L^2}\\
&\lesssim \eta \cN_T \Lambda(\cM_T^a, \cA_T)  (T^{\f12}+(T^{\f12}+ \ep^{\min\{
\f{1}{40},\f{\kappa}{3}\}})
\cN_T).
\end{aligned}
\eeq
\end{lem}
\begin{proof}
We will give the proof of $\|\phi F\|_{L_t^2H_{co}^2}$ $\|\phi (D_1^{\alpha}, D_2^{\alpha})\|_{L_t^1L^2},$ the other terms can be 
bounded in similar manners (by essentially replacing $\phi\omega^R$ by $\beta\omega_h^R$ ).
We first control $F_1-F_3$ defined in \eqref{def-F}. 
By using $u_3\p_z=\f{u_3}{\phi}Z_3$ and Hardy inequality,
the term $F_1$ can be controlled easily by 
\begin{align*}
    \|\phi F_1\|_{L_t^1H_{co}^2}&\lesssim T^{\f12} \big(\|(\div u^R, \na_h u_h^R, Z_3 u^R, u^R)\|_{L_t^2H_{co}^2} 
    \|\phi \omega_h^a, \omega_3^a\|_{L_t^{\infty}W_{co}^{3,\infty}}\\
   & \qquad\qquad  +\|(\phi \omega_h^R, \omega_3^R)\|_{L_t^2H_{co}^2}\|(Z u^a,\div u^a)\|_{L_t^{\infty}W_{co}^{2,\infty}}\big)\\
   &\lesssim T^{\f12}\eta \Lambda(\cM_T^a) \cN_T.
\end{align*}
Next, to estimate $F_2,$ we first  use the Sobolev embedding to get that 
\begin{align}\label{vor-1infty}
    \|(\phi \omega_h^R,\omega_3^R)\|_{L_t^{2}W_{co}^{1,\infty}}\lesssim \|\na (\phi \omega_h^R,\omega_3^R) \|_{L_t^{2}H_{co}^2}^{\f12}
    \|(\Id, \na_h^2) (\phi \,\omega_h^R,\omega_3^R) \|_{L_t^{2}H_{co}^1}^{\f12}
    \lesssim 
   (\nu^{-\f14}\eta)\eta \cN_T,
\end{align}
using which we can estimate $F_2$ in the following way
\begin{align*}
 \|\phi F_2\|_{L_t^1H_{co}^2}&\lesssim  \|(\phi \omega_h^R,\omega_3^R)\|_{L_t^{2}H_{co}^2}\|(Z u^R, \div u^R)\|_{L_t^2L^{\infty}}+\|(\phi \omega_h^R,\omega_3^R)\|_{L_t^{2}W_{co}^{1,\infty}} \|(Z u^R, \div u^R)\|_{L_t^2{H_{co}^2}}\\
 &\lesssim 
 (\nu^{-\f14}\eta)\eta\cN_T^2 \lesssim (\ep^{\f14}+\nu^{\f12})\eta\cN_T^2.
\end{align*}
We now deal with $F_3.$ First, we use the identity $f\times f=0$ to have that 
\begin{align*}
  \ep^{-1} \|\phi \na (\rho^{-\f{\gamma-1}{2}})\times \na\sigma^R\|_{L_t^1H_{co}^2}\lesssim T^{\f12}\|\na\sigma^R\|_{L_t^2H_{co}^2} \Lambda(\|\sigma^a\|_{L^{\infty}W_{co}^{3,\infty}})\lesssim T^{\f12}\eta\Lambda(\cM_T^a)\cN_T.
\end{align*}
For the next term, we denote $g=(\phi\na_h u^R, \nu\p_z u^R, \div u^R)$ and
count the derivatives hitting on each element and apply \eqref{rewrite-nasigmaR} to find  
\begin{align*}
\|\na(\rho^{-1})\times \phi \,& \div_{\nu} S u^R\, \|_{L_t^1H_{co}^2}\lesssim T^{\f12}\Lambda\big( \|(\sigma^a, \na\sigma^a)\|_{L_t^{\infty}W_{co}^{2,\infty}} \big) \|\ep g \|_{L_t^2\underline{H}_{co}^3}\\
& + \big(\|\na\sigma^R\|_{L_t^2H_{co}^2}+\|\sigma^R\|_{L_t^{\infty}H_{co}^2}\big) \Lambda\big(\ep \|\na\sigma^R\|_{L_t^2L^{\infty}}+\ep\| Z(\sigma^R,\sigma^a)\|_{L_{t,x}^{\infty}}\big) \|\ep (\na_h, Z_3)g\|_{L_t^2L^{\infty}}\\
&+ \|\na\sigma^R\|_{L_t^2W_{co}^{1,3}}\Lambda\big(\ep\|Z(\sigma^R,\sigma^a)\|_{L_{t,x}^{\infty}}\big)\|\ep (\na_h, Z_3)g\|_{L_t^2W_{co}^{1,6}}+\|\na \sigma^R\|_{L_t^2L^{\infty}}\|\ep (\na_h, Z_3)g\|_{L_t^2H_{co}^2}\\
&\lesssim \eta\cN_T \big( T^{\f12}\Lambda(\cM_T^a)+(\eta+\ep^{\f14}) \Lambda(\cM_T^a+\cA_T) \cN_T\big).
\end{align*}
The last term $F_4$ is similar and easier to control
\begin{align*}
    \|\phi F_4\|_{L_t^1H_{co}^2}&\lesssim \Lambda(\cM_T^a)\|\ep \phi (\, \div_{\nu}S u^a, \nu\p_z^2\omega^a, \Delta_h\omega^a)\|_{L_t^1H_{co}^2}+ T^{\f12}\Lambda(\cM_T^a,\cA_T)\|(\na\sigma^R, \sigma^R)\|_{L_t^2H_{co}^2}\\
  &\lesssim \eta \Lambda(\cM_T^a)+ T^{\f12}\eta\Lambda(\cM_T^a,\cA_T)\cN_T.
\end{align*}
Let us now estimate  $\|\phi (D_1^{\alpha}, D_2^{\alpha})\|_{L_t^1L^2}$ where $D_1^{\alpha}, D_2^{\alpha}$ are defined in \eqref{def-D12alpha}. The term $\phi\, D_2^{\alpha}$ can be bounded in exactly the same way as $\|\na(\rho^{-1})\times \phi \, \div_{\nu} S u^R\, \|_{L_t^1H_{co}^2},$ we thus omit the detail. Concerning the term $\phi\, D_1^{\alpha},$ we use the identity $u_3\p_z=\f{u_3}{\phi}Z_3$ and the Hardy inequality to find that 
\begin{align*}
\|\phi\, D_{1}^{\alpha}\|_{L_t^1L^2}&\lesssim 
T^{\f12}\|(u^a, \p_z u_3^a)\|_{L_t^{\infty}W_{co}^{2,\infty}} \|\phi Z\omega^R\|_{L_t^2H_{co}^1}+\|(u_h^R,\p_z u_3^R)\|_{L_t^2H_{co}^2}\|\phi Z\omega^R\|_{L_t^2L^{\infty}}\\
& +\| u_h^R\|_{L_t^2W_{co}^{1,\infty}} \|\phi \na_h\omega^R\|_{L_t^2H_{co}^{1}}+\min\{\|Z \p_z u_3^R\|_{L_t^2L^{\infty}} \|\phi Z_3\omega^R\|_{L_t^2H_{co}^{1}}, \|\p_z u_3^R\|_{L_t^2W_{co}^{1,6}}\|\phi Z_3\omega^R\|_{L_t^2W_{co}^{1,3}} \}.
\end{align*}
Thanks to \eqref{vor-1infty}, the inequalities \eqref{ineq-319} \eqref{ineq-320} (with `tan' changed to `co') as well as  
\begin{align*}
   \|\phi Z\omega^R\|_{L_t^2(L_x^{\infty}\cap W_{co}^{1,3})} 
  \lesssim \|\phi \na \omega^R\|_{L_t^2H_{co}^2}^{\f{1}{2}}\|\phi\omega^R\|_{L_t^2H_{co}^2}^{\f12} \lesssim 
  \nu^{-\f14} \eta \cN_T, 
\end{align*}
we obtain by using $\min\{ \nu^{-\f{5}{12}}, \ep^{-(\f12+\f{1}{16})}+\nu^{-\f14}\} \eta\lesssim \ep^{\f{1}{40}}+\nu^{\f13}$ 
\begin{align*}
    \|\phi\, D_{1}^{\alpha}\|_{L_t^1L^2}&\lesssim \big(T^{\f12}+
   \ep^{\min\{\f{1}{40},\f{\kappa}{3}\}}
  \big) \eta\cN_T \big(\Lambda(\cM_T^a)+\cN_T\big).
\end{align*}
We thus finish the proof.
\end{proof}

\subsection{Some estimates for the compressible part}
 In this substep, we aim to bound $\ep \|(\na\sigma^R, \div u^R)\|_{L_t^{\infty}\underline{H}_{co}^2}$, $\ep\|\na\div u^R\|_{L_t^2\underline{H}_{co}^2}$ which will be useful to control uniformly the $L_t^2H_{co}^2$ norm of $(\na\sigma^R, \div u^R)$ in the next step.   
\begin{lem}
    Under the same assumption as in Lemma \ref{lem-hightan}, we have that for any $t\in[0,T],$ 
    \beq
 \begin{aligned} \label{es-nadiv}
    \ep^2 \|(\na\sigma^R, \div u^R)\|_{L_t^{\infty}\underline{H}_{co}^2}^2+ \ep^2\|\na\div u^R\|_{L_t^2\underline{H}_{co}^2}^2 &\lesssim \ep^2 \|(\na\sigma^R, \div u^R)(0)\|_{\underline{H}_{co}^2}^2+ \eta^2
 \Lambda(\cM_T^a) \\
 &+\eta^2(T^{\f13}+
 \ep^{\min\{\f{1}{40},\f{\kappa}{3},\f{3-\kappa}{4}\}})\Lambda\big(\cM_T^a,\cA_T\big)\big(\cN_T^2+\cN^4_T\big).
 \end{aligned}
 \eeq
\end{lem}
\begin{proof}
Let $\alpha\in \mathbb{N}^4$ be a multi-index with $\alpha_0\leq 1.$ Taking $(Z^{\alpha}\na, Z^{\alpha})$ on the equations \eqref{CNS-remainder}, we find the following equation 
 \beq \label{eq-nasi-divu}
\left\{
\begin{array}{l}
 \big(\displaystyle\pt+ (u^{a}+u^R)\cdot\na\big)Z^{\alpha}\na \sigma^{R}+\f{\rho^{\f{\gamma-1}{2}}Z^{\alpha}\na\div u^{R}}{\ep}=-Z^{\alpha} \na R^a_{\sigma}-\cJ_{\sigma,1}^{\alpha}- \cJ_{\sigma,2}^{\alpha}
 \\[8pt]
 \pt Z^{\alpha} u^R+ \f{\rho^{\f{\gamma-1}{2}}Z^{\alpha}\nabla \sigma^{R}}{\ep}-\f{\tilde{\mu}}{\rho}Z^{\alpha}\na\div u^R=-Z^{\alpha}R^a_{u}
-\cJ_{u,1}^{\alpha}-\cJ_{u,2}^{\alpha}
-\cJ_{u,3}^{\alpha}-\cJ_{u,4}^{\alpha}
\end{array}
\right.
\eeq
where $\tilde{\mu}=\mu_2+\mu_1\nu>0$ and
\beq\label{def-cJ}
\begin{aligned}
  & \cJ_{\sigma,1}^{\alpha}=Z^{\alpha}\big( \na (u^R\cdot\na\sigma^a)+\na(u^a+u^R)\cdot\na \sigma^R+\na (\ep^{-1}\rho^{\f{\gamma-1}{2}})\, \div u^R\big), \quad \cJ_{\sigma,2}^{\alpha}= [Z^{\alpha}, \ep^{-1}\rho^{\f{\gamma-1}{2}}]\na\div u^R,\\
&  \cJ_{u,1}^{\alpha}=  Z^{\alpha}\big((u^{a}+u^R)\cdot\na u^{R}+u^R\cdot\na u^a\big), \qquad  \cJ_{u,2}^{\alpha}=[Z^{\alpha}, \ep^{-1}\rho^{\f{\gamma-1}{2}}]\na \sigma^R-\tilde{\mu}[Z^{\alpha},  \rho^{-1}]\na\div u^R, \\
&  \cJ_{u,3}^{\alpha}=-Z^{\alpha} \big(\big({1}/{\rho}-1\big)\div_{\nu} S u^a\big), \qquad \cJ_{u,4}^{\alpha}=-\mu_1 Z^{\alpha}\big( \rho^{-1} \nu\big(  \p_3 (\omega_2^R, -\omega_1^R, 0)^t -\na\div_h u_h^R \big)+\rho^{-1}\Delta_h u^R\big).
\end{aligned}
\eeq
Taking the inner product of the above equations with $\ep^2(Z^{\alpha}\na\sigma^R, -Z^{\alpha}\na\div u^R)$ in $L^2(Q_t),$ 
we obtain, after suitable integration by parts the following energy equality
\beqs
\begin{aligned}
&\ep^2\int_{\Omega} (|Z^{\alpha}\na\sigma^R|^2+|Z^{\alpha}\div u^R|^2)(t) \, \d x\, + \ep^2 \tilde{\mu}  \iint_{Q_t} \rho^{-1}|Z^{\alpha}\na\div u^R|^2\, \d x \d s \\
&= \ep^2\int_{\Omega} (|Z^{\alpha}\na\sigma^R|^2+|Z^{\alpha}\div u^R|^2)(0) \, \d x +\ep^2  \iint_{Q_t} \div (u^a+u^R) |Z^{\alpha}\na\sigma^R|^2\, \d x\d s \\
&\qquad -\ep  \iint_{Q_t} \ep\pt Z^{\alpha} u_3^R \cdot [Z^{\alpha},\p_z]\div u^R +\ep\pt [Z^{\alpha},\div] u^R  \cdot Z^{\alpha}\div u^R\d x\d s \\
&\qquad + \ep^2  \iint_{Q_t} -(Z^{\alpha} \na R^a_{\sigma}+\cJ_{\sigma,1}^{\alpha}+ \cJ_{\sigma,2}^{\alpha})\cdot Z^{\alpha}\na\sigma^R
+\big(Z^{\alpha} R^a_{u}+\sum_{k=1}^4 \cJ_{u,k}^{\alpha}\big)\cdot Z^{\alpha}\na\div u^R \, \d x \d s.
\end{aligned}
\eeqs
The second term in the right hand side can be bounded directly by 
\beq \label{term1}
T^{\f12}  \|\div(u^a, u^R)\|_{L_t^2L^{\infty}}\|\ep \na\sigma^R\|_{L_t^{\infty}H_{co}^2}^2\lesssim T^{\f12}\Lambda(\cM_T^a,\cA_T)\eta^2\cN_T^2.
\eeq
By the virtue of identities \eqref{algebraicid} and integration by parts, we can control the  second line  in the right hand side by 
\begin{align}\label{term2}
  \ep  \|\p_z u_3^R\|_{L_t^2H_{co}^2}\|\div u^R\|_{L_t^2H_{co}^2}\lesssim \ep \eta^2\cN_T^2.
\end{align}
By using the Young's inequality, we bound the last line by 
\begin{align*}
   & \ep^2 (\|\na R_{\sigma}^a\|_{L_t^2H_{co}^2}+\|(\cJ_{\sigma,1}^{\alpha},\cJ_{\sigma,2}^{\alpha})\|_{L_t^2L^2})\|\na\sigma^R\|_{L_t^2H_{co}^2}\\
    &+\f{\ep^2\tilde{\mu}}{100}\|Z^{\alpha}\na\div u^R\|_{L_t^2L^2}^2+C\|\ep R_{u}^a\|_{L_t^2H_{co}^2}^2+\sum_{k=1}^4\|\ep \cJ_{u,k}^{\alpha}\|_{L_t^2L^2}^2 .
\end{align*}
By the definition of $\cM_T^a$ in \eqref{def-cM}, it holds that 
$$\ep^2\|\na R_{\sigma}^a\|_{L_t^2H_{co}^2}\lesssim (\ep+(\ep^3/\nu)^{\f12})\eta\cM_T^a,\quad  \|\ep R_{u}^a\|_{L_t^2H_{co}^2}\lesssim \ep\eta\cN_T.$$
We then apply \eqref{es-cJsigmau} to find that the last line can be controlled further by 
\begin{align*}
    \f{\ep^2\tilde{\mu}}{100}\|Z^{\alpha}\na\div u^R\|_{L_t^2L^2}^2+\|\ep Z^{\alpha}\Delta_h u^R\|_{L_t^2L^2}^2+\|\ep\nu\p_z\omega_h^R\|_{L_t^2\underline{H}_{co}^2}^2\\
+\eta^2\Lambda(\cM_T^a)+\big(\eta+\ep^{\f{3-\kappa}{2}}
    +T \big)\eta^2\Lambda(\cM_T^a,\cA_T)(\cN_T^2+\cN_T^4).
\end{align*}
This, together with \eqref{term1} \eqref{term2}, \eqref{ES-hightan}, \eqref{es-omegah-weight}, enables us to find \eqref{es-nadiv}.
\end{proof}
\begin{lem}
    Let $\cJ_{\sigma,1}^{\alpha},\cJ_{\sigma,2}^{\alpha},\cJ_{u,1}^{\alpha}-\cJ_{u,4}^{\alpha} $ be defined in \eqref{def-cJ}. It holds that, under the same assumption as in Lemma \ref{lem-hightan}, for any $|\alpha|\leq 2, \alpha_0\leq 1,$ $t\in[0,T],$
    \beq \label{es-cJsigmau}
    \begin{aligned}
     &  \ep^2 \|(\cJ_{\sigma,1}^{\alpha},\cJ_{\sigma,2}^{\alpha})\|_{L_t^2L^2}\lesssim \big(\ep^{\f{3-\kappa}{4}}+\eta\big)\eta\Lambda(\cM_T^a, \cA_T)(\cN_T+\cN_T^2), \\ 
      & \sum_{k=1}^4 \ep\|\cJ_{u,k}^{\alpha}\|_{L_t^2L^2}\lesssim \|\ep Z^{\alpha}\Delta_h u^R\|_{L_t^2L^2}+\|\ep{\nu} \p_z \omega_h^R\|_{L_t^2\underline{H}_{co}^2}+\ep\eta \Lambda(\cM_T^a)\\
      &\qquad\qquad \qquad \qquad+\big(\eta+\ep^{\min\{\f{3-\kappa}{4},\f14, \f{\kappa}{2}\}}
    +T^{\f12}  \big)\eta\Lambda(\cM_T^a,\cA_T)(\cN_T+\cN_T^2).
    \end{aligned}
    \eeq
\end{lem}
\begin{proof}
Let us begin with the estimates for $\cJ_{\sigma,1}^{\alpha}, \cJ_{\sigma,2}^{\alpha}.$ For the term $\cJ_{\sigma,1}^{\alpha},$ let us
 detail the estimate of $Z^{\alpha}(\na u^R\cdot\na\sigma^R),$ the other ones are easier. 
 We count the derivative hitting on each element to have that 
 \begin{align*}
\ep^2\|Z^{\alpha}(\na u^R\cdot\na\sigma^R)\|_{L_t^2L^2}&\lesssim \ep^2\|\na u_h^R\|_{L_t^{\infty}(L_x^{\infty}\cap W_{co}^{1,3})}\|\na\sigma^R\|_{L_t^2\underline{H}_{co}^2}+\ep\|\na u_h^R\|_{L_t^2{\underline{H}}_{co}^2}  \|\ep \na_h \sigma^R\|_{L_{t}^{\infty}L_x^{\infty}}\\
&\,\,\quad +\|\ep\na u_3^R\|_{L_t^2W_{co}^{1,\infty}}   \|\ep\p_z\sigma^R\|_{L_t^{\infty}\underline{H}_{co}^2}+\ep\|\ep \na u_3^R\|_{L_t^{\infty}\underline{H}_{co}^2}\|\p_z \sigma^R\|_{L_t^2L^{\infty}}.
 \end{align*}
 Applying \eqref{pzuh-dinfty} and \eqref{rewrite-nasigmaR}, we find 
 \beqs
\ep^2\|Z^{\alpha}(\na u^R\cdot\na\sigma^R)\|_{L_t^2L^2}\lesssim \big(
\eta+\ep^{\f{3-\kappa}{4}}+\ep\nu^{-\f12}\eta\big) \eta\cN_T^2\lesssim  \big(
\eta+\ep^{\f{3-\kappa}{4}}\big) \eta\cN_T^2.
\eeqs
%
Next, since 
\begin{align*}   \|\ep\na\div u^R\|_{L_t^2L^3}\lesssim \|\ep\na\sigma^R\|_{L_t^2W_{co}^{1,3}}+\ep^2 \|\na(R_{\sigma}^a, (u^a+u^R)\na\sigma^R, u^R\cdot\na\sigma^a)\|_{L_t^2W_{co}^{1,3}}\lesssim \cN_T+\Lambda(\cM_T^a),\end{align*}
the term $\ep^2\cJ_{\sigma,2}^{\alpha}$ can be bounded as 
\begin{align*}
    \ep^2\|\na\div u^R\|_{L_t^2H_{co}^2}\Lambda(\cM_T^a)+\ep\Lambda(\cM_T^a, \cA_T) \|\sigma^R\|_{L_t^2W_{co}^{2,6}}\|\ep\na\div u^R\|_{L_t^2L^3}\lesssim \ep\eta (\cN_T+\cN_T^2) \Lambda(\cM_T^a, \cA_T).
\end{align*}

We can now switch to the estimate of $\ep \cJ_{u,k}^{\alpha}.$ 
For $\cJ_{u,1}^{\alpha},$
let us showcase the control of the term $\ep Z^{\alpha}(u^R\cdot\na u^R),$ the other terms are easier.  On the one hand, in view of \eqref{energynorm-uh}, 
\begin{align*}
  \ep  \| Z^{\alpha}(u_h^R\cdot\na_h u^R)\|_{L_t^2L^2}&\lesssim  \ep\|u_h^R\|_{L_t^{\infty}(L_x^{\infty}\cap W_{co}^{1,3})}\|\na_h u^R\|_{L_t^2H_{co}^2}+\| u_h^R\|_{L_t^{\infty}H_{co}^2}  \|\ep \na_h u^R\|_{L_{t}^{2}L^{\infty}}\\
&\lesssim \big(\ep\eta \max\{(\ep\nu)^{-\f14}, (\ep^2\nu)^{-\f16}\} \big)\eta\cN_T^2\lesssim \ep^{\f23}\eta\cN_T^2.
\end{align*}
On the other hand, by using the identity 
$u_3^R\p_3u_h^R=\f{u_3^R}{\phi}\phi(\omega_h^R)^{\perp}+u_3^R\na_h u_3^R,$
the Hardy inequality and the Fundamental Theorem of Calculus as well as the definitions \eqref{def-cET} \eqref{def-cAT}, we get that
\begin{align*}
\ep \| Z^{\alpha}(u_3^R\p_z u ^R)\|_{L_t^2L^2}&\lesssim  \|\phi \omega_h^R\|_{L_t^{\infty}H_{co}^2}\|\ep\p_z u_3^R\|_{L_t^2W_{co}^{1,\infty}}+\|\p_z u_3^R\|_{L_t^2H_{co}^2}\|\ep\phi \omega_h^R\|_{L_t^{\infty}L^{\infty}}\\
&\quad\,\, 
+\ep\|\na u_3^R\|_{L_t^2(L^{\infty}\cap W_{co}^{1,3})}\|\na u_3^R\|_{L_t^{\infty}H_{co}^1}+\ep \|\na u_3^R\|_{L_t^2H_{co}^2}\|u_3^R\|_{L_t^{\infty}L^{\infty}}\\
&\lesssim \big(\eta+\ep\eta (\ep\nu)^{-\f12} \big)\eta\cN_T^2\lesssim \big(\eta+\ep^{\f{3-\kappa}{4}}\big)\eta\cN_T^2.
\end{align*}
The above two estimates then give that 
\begin{align*}
    \ep \| Z^{\alpha}(u\cdot\nabla 
    u ^R)\|_{L_t^2L^2}\lesssim \big(\eta+\ep^{\min\{\f{3-\kappa}{4},\f14+\f{\kappa}{2}\}}\big)\eta\cN_T^2\,.
\end{align*}
Next, the term $\ep\cJ_{u,2}^{\alpha}$ can be estimated in the similar way as $\ep^2\cJ_{\sigma,2}^{\alpha},$ we thus omit the details. The term $\ep\cJ_{u,3}^{\alpha}$ is easily controlled by 
\begin{align*}
\|\ep \cJ_{u,3}^{\alpha}\|_{L_t^2L^2}&\lesssim \ep\Lambda(\|\sigma^a\|_{L_t^{\infty}W_{co}^{3,\infty}})\|\ep\div_{\nu} S u^a\|_{L_t^2H_{co}^2} 
+T^{\f12}\|\sigma^R\|_{L_t^{\infty}\underline{H}_{co}^2}\Lambda(\cM_T^a,\cA_T)\\
&\lesssim \eta \big( \Lambda(\cM_T^a)+ T^{\f12} \Lambda(\cM_T^a,\cA_T)\cN_T\big).
\end{align*}
Finally, in view of its definitions \eqref{def-cJ}, we get readily that 
\begin{align*}
\|\ep \cJ_{u,4}^{\alpha}\|_{L_t^2L^2}&\lesssim \|\ep\nu \nabla \omega_h^R\|_{L_t^2\underline{H}_{co}^2}+\|\ep \Delta_h u^R\|_{L_t^2\underline{H}_{co}^2}+
(\ep^{\f14}+\ep^{\f{\kappa}{2}})\eta(\cN_T+\cN_T^2),
\end{align*}
since $\ep\nu \|\na^2 u^R\|_{L_t^2L^{\infty}}\lesssim \nu^{-\f14}\eta \cN_T\lesssim (\ep^{\f14}+\ep^{\f{\kappa}{2}})\cN_T. $
\end{proof}

\subsection{ Uniform control of the of $(\na\sigma^R, \div u^R)$ } 
In this step, we aim to prove the uniform boundedness of $\|(\na\sigma^R, \div u^R)\|_{L_t^2H_{co}^2\cap L_t^{\infty}\underline{H}_{co}^1},$ which, due to the singular terms $\ep^{-1}(\div u^R, \na\sigma^R)$ in the equations \eqref{CNS-remainder}, is unlikely to be obtained by direct energy estimates. Since we have obtained in the second step the estimate for high order tangential derivatives, we could now use the equations to recover weighted normal derivatives.  

\begin{lem}\label{lem-nasigmadiv-uniform}
Under the same assumption as in Lemma \ref{lem-hightan}, we have that for any $t\in[0,T],$
\beq\label{es-nasigma-divuR}
\begin{aligned}
   & \|(\na\sigma^R, \div u^R)\|_{L_t^2H_{co}^2\cap L_t^{\infty}\underline{H}_{co}^1}^2+\|\ep\na\div u^R\|_{L_t^{\infty}\underline{H}_{co}^1}^2  
   \\
    & \lesssim \eta^2\big(\cN_0^2+\Lambda(\cM_T^a)\big)+(T^{\f12}+
    \ep^{\min\{\f{1}{40},\f{\kappa}{3},\f{3-\kappa}{4}\}})\eta^2\Lambda(\cM_T^a,\cA_T)(\cN_T^2+\cN_T^4).
    \end{aligned}
\eeq
\end{lem}
\begin{proof}
Let us rewrite the equation \eqref{CNS-remainder} into the following form 
\begin{align}\label{rewrite-eq}
    \div u^R=-\ep\pt \sigma^R
    -\cH_{1}, \quad 
    \p_z\sigma^R=-\ep\pt u_3^R-\cH_{2}
\end{align}
where 
\beq
\begin{aligned}\label{def-cH12}
    \cH_{1}&=\ep (\rho^{\f{\gamma-1}{2}}-1)\pt\sigma^R+\ep \rho^{\f{\gamma-1}{2}} \big(  \displaystyle  (u^{a}+u^R)\cdot\na \sigma^{R}+u^R\cdot\na\sigma^a+R^a_{\sigma} \big), \\
    \cH_{2}&=\ep (\rho^{\f{\gamma-1}{2}}-1)\pt u_3^R+\ep \rho^{\f{\gamma-1}{2}} \big(  (u^{a}+u^R)\cdot\na u_3^{R}+u^R\cdot\na u_3^a+(R^a_{u})_3 \big)\\
    &\qquad -\ep \tilde{\mu}\,\p_z\div u^R-\ep \mu_1(\Delta_h u_3^R-\nu\p_z\div_h u_h^R).
\end{aligned}
\eeq
By using the equation \eqref{rewrite-eq} twice, we get easily that 
\begin{align*}
    \|Z_3^2\, \div u^R\|_{L_t^2L^2}&\lesssim \|(\ep\pt)Z_3\p_z\sigma^R\|_{L_t^2L^2}+\|\cH_1\|_{L_t^2\underline{H}_{co}^2}\\
    &\lesssim \|(\ep\pt)^2(\div u^R, \div_h u_h^R)\|_{L_t^2L^2}+\|(\cH_1,\cH_2)\|_{L_t^2\underline{H}_{co}^2}, \\
   \|Z_3^2\, \p_z\sigma^R\|_{L_t^2L^2}&\lesssim \|(\ep\pt)Z_3\p_zu_3^R\|_{L_t^2\underline{H}_{co}^2}+\|\cH_2\|_{L_t^2L^2}\\
    &\lesssim \|(\ep\pt)^2\p_z \sigma^R\|_{L_t^2L^2}+\|Z_3 u_h^R\|_{L_t^2H_{tan}^2}+\|(\cH_1,\cH_2)\|_{L_t^2\underline{H}_{co}^2}.
\end{align*}
The norm $\|Z_3(\div u^R,\p_z\sigma^R)\|_{L_t^2H_{tan}^1}$ can be estimated in the same manner. 
As a result, we obtain that
\begin{align*}
\|(\div u^R,\p_z\sigma^R)\|_{L_t^2H_{co}^2}\lesssim 
\|(\div u^R, \p_z\sigma^R)\|_{L_t^2H_{tan}^2}+T^{\f12}\|(\phi \, \omega_h^R, \na_h u_3^R)\|_{L_t^{\infty}H_{tan}^2}+\|(\cH_1,\cH_2)\|_{L_t^2\underline{H}_{co}^2}.
\end{align*}
In view of the estimate \eqref{cH1-2-L2} for  $\cH_1, \cH_2,$
we obtain that 
\begin{align*}
    \|(\div u^R,\na \sigma^R)\|_{L_t^2H_{co}^2}&\lesssim \|(\div u^R, \na\sigma^R)\|_{L_t^2H_{tan}^2}+\ep\|\na\div u^R\|_{L_t^2\underline{H}_{co}^2}+\eta\Lambda(\cM_T^a)\\
    &\qquad + \eta \cN_T  \Lambda(\cM_T^a,\cA_T) \big( T^{\f12}+\ep+ (\ep+\eta)\cN_T \big).
\end{align*}
Let us now recover $L_{t}^{\infty}\underline{H}_{co}^1$ norms of $(\na\sigma^R, \div u^R).$ 
We first get from
the first equation of \eqref{rewrite-eq} that   
\begin{align*}
  \ep\|  \p_z\div u^R\|_{L_t^{\infty}\underline{H}_{co}^1}\lesssim \ep\|\na\sigma^R\|_{L_t^{\infty}\underline{H}_{co}^2}+ \ep\|\p_z \cH_1\|_{L_t^{\infty}\underline{H}_{co}^1},
\end{align*}
and use it together with the second equation 
 of \eqref{rewrite-eq} to find
 \begin{align*}
\|\p_z\sigma^R\|_{L_t^{\infty}\underline{H}_{co}^1}&\lesssim \ep\|(\div u^R, \div_h u_h^R,\na\sigma^R)\|_{L_t^{\infty}\underline{H}_{co}^2}
+\ep \|u_3^R\|_{L_t^{\infty}\underline{H}_{tan}^3}+\ep\nu\|\omega_h^R\|_{L_t^{\infty}H_{co}^2}+ \|\ep\p_z\cH_1, \cH_{2,1}\|_{L_t^{\infty}\underline{H}_{co}^1} 
 \end{align*}
where we denoted $\cH_{2,1}$ as the first line in the definition of $\cH_2.$ From the second equation of \eqref{rewrite-eq} it also stems 
\begin{align*}
    \|Z_3(\ep\pt)\sigma^R\|_{L_t^{\infty}L^2}\lesssim \|u_3^R\|_{L_t^{\infty}H_{tan}^2}+\ep\|\div u^R\|_{L_t^{\infty}\underline{H}_{co}^2}+\ep \|u_3^R\|_{L_t^{\infty}\underline{H}_{tan}^3}+\ep\nu\|\phi\omega_h^R\|_{L_t^{\infty}H_{tan}^2}+\|\phi \cH_{2,1}\|_{L_t^{\infty}H_{tan}^1}
\end{align*}
while from the
first equation one obtains
\begin{align*}
\|\div u^R\|_{L_t^{\infty}\underline{H}_{co}^1}\lesssim \|\sigma^R\|_{L_t^{\infty}\underline{H}_{tan}^2}+\|Z_3(\ep\pt)\sigma^R\|_{L_t^{\infty}L^2}+\|\cH_1\|_{L_t^{\infty}\underline{H}_{co}^1}.
\end{align*}
Combining the previous three estimates and using \eqref{Linfty-H1H2},
we deduce that 
\begin{align*}
  \|(\p_z\sigma^R,\div u^R)\|_{L_t^{\infty}\underline{H}_{co}^1}&\lesssim 
  \|(\sigma^R, u_3^R)\|_{L_t^{\infty}H_{tan}^2}+\ep\|(\div u^R,\na\sigma^R)\|_{L_t^{\infty}\underline{H}_{co}^2}+\eta\Lambda(\cM_T^a)\\
  & \qquad + \eta\cN_T\Lambda(\cM_T^a,\cA_T) \big( \ep+ (\ep+\eta)\cN_T \big).
\end{align*}
Thanks to the estimates \eqref{es-nasigma-tan}, \eqref{ES-hightan}, \eqref{es-nadiv}, we derive \eqref{es-nasigma-divuR}.
\end{proof}
We state in the next lemma some estimates for $\cH_1,\cH_2$ used in the proof of the above lemma.
\begin{lem}
Let $\cH_1,\cH_2$ be defined in \eqref{def-cH12}.
Under the same assumption as in Lemma \ref{lem-hightan}, we have that for any $t\in[0,T],$ 
\begin{align}\label{cH1-2-L2}
    \|(\cH_1,\cH_2)\|_{L_t^2\underline{H}_{co}^2}\lesssim \eta \cN_T  \Lambda(\cM_T^a,\cA_T) \big( \ep+ (\ep+\eta)\cN_T \big)+\ep \|(R_{\sigma}^a, R_u^a)\|_{L_t^2H_{co}^2}+\ep\|\na\div u^R\|_{L_t^2\underline{H}_{co}^2}, 
\end{align}
\begin{align}\label{Linfty-H1H2}
    \|(\ep\p_z\cH_1, \cH_1 , \cH_{2,1})\|_{L_t^{\infty}\underline{H}_{co}^1}\lesssim \big( \ep^{\min\{\f13,\f{3-\kappa}{2}\}}+\eta\big) \eta  \Lambda(\cM_T^a,\cA_T) (\cN_T+\cN_T^2) +\ep\|(\Id, \ep\p_z)R_{\sigma}^a, (R_u^a)_3)\|_{L_t^{\infty}\underline{H}_{co}^1}.
\end{align}
\end{lem}
The proof of this lemma follows from direct estimation of each term by using the norms involved in $\cN_T,$ since it is very similar to what has been done before, we omit the proof.

\subsection{Second normal derivative of $\sigma^R$}
In this short subsection, we state an estimate for the second normal derivative of $\sigma^R,$ which is useful to control $\|\p_z\sigma^R\|_{L_t^{2}L^p}$ with $p>2.$
\begin{lem}
Let $\tau:=\min\{(\nu^{5}/\ep^7)^{\f16}, (\nu^2/\ep^3)^{\f12}, 1 \}.$
Under the same assumption as in Lemma \ref{lem-hightan}, we have that for any $(\ep,\nu)\in A_0,$ any  $t\in[0,T],$ 
\begin{align*}
 &
 \tau^2\big(\|\ep\p_z^2\sigma^R \|_{L_t^{\infty}H_{co}^1}^2+\|\p_z^2\sigma^R\|_{L_t^2H_{co}^1}^2\big)\\
 &\lesssim \cN^2(0)+\eta^2 \Lambda(\cM_T^a)
 + (T+\ep)^{d_1(\kappa)}\eta^2 \Lambda(\cM_T^a,\cA_T)(\cN_T^2+\cN_T^4),
\end{align*}
for some $0<d_1(\kappa)\lesssim \ep^{\min\{\f{1}{40},\f{\kappa}{3},\f{3-\kappa}{4}\}}.$
\end{lem}
\begin{proof}
The proof is similar to that of Lemma \ref{lem-pzsig-tan}, we thus only sketch the proof.
We apply $Z^{\alpha}\p_z $ $(|\alpha|=1 \text{ or } 0) $ on the equation \eqref{eq-pzsigma} and taking the inner product of the resultant equation with $\tau^2 Z^{\alpha}\p_z^2\sigma^R ,$ we find after  applying Young's inequality that
\begin{align*}
(\tau \ep)^2 \|\p_z^2\sigma^R\|_{L_t^{\infty}H_{co}^1}^2+\tau^2\| \p_z^2\sigma^R\|_{L_t^2H_{co}^1}^2& \lesssim    (\tau \ep)^2 \|\p_z^2\sigma^R(0)\|_{H_{co}^1}^2+\tau^2 \|(\rho^{\f{\gamma+1}{2}}\ep\pt u_3^R,\p_z G_u, \p_z G_{\sigma})\|_{L_t^2H_{co}^1}^2 \\
&+\ep\|\div u\|_{L_t^{2}L^{\infty}}\|\tau \p_z^2\sigma^R\|_{L_t^2H_{co}^1}\|\tau\ep\p_z^2\sigma^R\|_{L_t^{\infty}H_{co}^1} \\
&+\tau^2 \sum_{|\alpha|=0,1}\big(\|[Z^{\alpha}\p_z,\rho^{\gamma}]\p_z\sigma^R\|_{L_t^2L^2}^2\|\ep^2[Z^{\alpha}\p_z, u\cdot\na]\p_z\sigma^R \|_{L_t^2L^2}^2\big).
\end{align*}
It is not hard to check that the second line can be bounded by $\ep^{d_1(\kappa)} \eta^2 \Lambda(\cM_T^a)(\cN_T^2+\cN_T^4).$ For instance,  it holds that
\begin{align*}
   & \tau \ep^2 \|[Z^{\alpha}\p_z, u\cdot\na]\p_z\sigma^R \|_{L_t^2L^2}\lesssim  \|\ep^2\p_z u 
    \|_{L_{t}^{\infty}(L^{\infty}\cap W_{co}^{1,3})}\|\tau\na\p_z\sigma^R\|_{L_t^2H_{co}^1}\\
    &\qquad\qquad\qquad\qquad\qquad+\|\tau\ep \na\p_z\sigma^R\|_{L_t^{\infty}H_{co}^1}\|\ep \p_z u_3\|_{L_t^2W_{co}^{1,\infty}}\lesssim \ep^{
    \f13} \eta^2 \Lambda(\cM_T^a)(\cN_T^2+\cN_T^4).
\end{align*}
Concerning the estimates of $\|\p_z(G_u, G_{\sigma})\|_{L_t^2H_{co}^1},$ let us detail two terms appearing in $G_{\sigma}$ namely $\tau\ep^2\| \p_z^2 u^a \cdot \na\sigma^R\|_{L_t^2H_{co}^1}, \tau\ep^2\|\p_z^2 u^R \cdot \na\sigma^a\|_{L_t^2H_{co}^1}$ whose estimates explain why we have to put some weights $\tau$ in the estimate of $\p_z^2\sigma^R.$ As we have that $\|
\p_z^2 (u_h^a, \p_z u_3^a)\|_{L_t^{\infty}W_{co}^{1,p}}\lesssim (\ep\nu)^{-1+\f{1}{2p}}\Lambda(\cM_T^a),$ 
it holds that 
\begin{align*}
    \tau\ep^2\| \p_z^2 u^a \cdot \na\sigma^R\|_{L_t^2H_{co}^1}&\lesssim \|\tau\ep^2 \p_z^2 (u_h^a,\p_z u_3^a)\|_{L_t^{\infty}W_{co}^{1,3}}\|\na\sigma^R\|_{L_t^2H_{co}^2}+\|\tau\ep^2\p_z^2(u_h^a,\p_z u_3^a)\|_{L_t^{\infty}H_{co}^1}\|Z\sigma^R\|_{L_t^2L^{\infty}}\\
    &\lesssim \eta^2\Lambda(\cM_T^a)+\|\na\sigma^R\|_{L_t^2H_{co}^2}^2+\ep^{\min\{\f{1+\kappa}{12},\f16\}}\eta\Lambda(\cM_T^a)\cN_T.
\end{align*}
Moreover, we readily control the second one by 
\begin{align*}
    \tau\ep^2\| \p_z^2 u^R & \cdot \na\sigma^a\|_{L_t^2H_{co}^1}\lesssim \|\tau\ep^2 \p_z^2 u^R\|_{L_t^{2}H_{co}^{1}}\|\na\sigma^a\|_{L_t^{\infty}W_{co}^{1,\infty}}\\
    &\lesssim \tau \ep^2 \big(
\beta_1^{-1}\|\beta_1\p_z\omega_h^R\|_{L_t^2H_{co}^1}+
    \|\p_z u_3^R\|_{L_t^2H_{co}^1}+\|\p_z(\div u^R, \div_h u_h^R)\|_{L_t^2H_{co}^1}\big)
    \Lambda(\cM_T^a) \\
    &\lesssim  \|\beta_1\p_z\omega_h^R\|_{L_t^2H_{co}^1}\Lambda(\cM_T^a)+\tau\ep^{\f12}\eta \Lambda(\cM_T^a)\cN_T. 
\end{align*}
where $\beta_1=\min\{\ep^{\f23}\nu^{\f56}, \ep^{\f12}\nu\}.$

The other terms appearing in $\p_z(G_u, G_{\sigma})$ indeed enjoy better property, we omit the details. The desired estimates follows from above estimates and \eqref{es-omegah-weight},\eqref{es-nasigma-divuR}.
\end{proof}

\section{Proof of Theorem \ref{thm1}}\label{sec-proofofthm1}
In this section, we will give the proof of Theorem \ref{thm1} based on  Proposition \ref{prop-unifromes}. Before that, we will first 
finish the proof of Proposition \ref{prop-unifromes} by providing the estimates of $L_x^{p}\, (p=3,6,+\infty)$ type norms listed in the definition of $\cA_T$ in \eqref{def-cAT}.
\begin{prop}\label{prop:est-A-t}
    Under the same assumption as in Proposition \ref{prop-unifromes}, we have that for any $(\ep,\nu=\ep^{\kappa})\in A_0,$ 
    any $t\in[0,T],$ there exists $d_2(\kappa)>0$ such that 
    \begin{align*}
    \cA_T\lesssim \cN_0+\eta\Lambda(\cM_T^a) (T+\ep)^{d_2(\kappa)}(\cN_T +\cN_T^2). 
    \end{align*}
\end{prop}
\begin{proof}
Indeed, most of quantities included in $\cA_T$
are easily estimated by $\mathcal{E}_{T}$ and the Sobolev embedding 
\begin{align*}
    \|f(t)\|_{L_x^{\infty}}\lesssim \|\p_z f (t)\|_{H_{tan}^1}^{\f12}\|f(t)\|_{H_{tan}^{1+\delta}}^{\f12}, \quad \|f(t)\|_{L_x^{3}}\lesssim \|\p_z f (t)\|_{L_x^2}^{\f12}\|f(t)\|_{L_x^2}^{\f12}, \,\quad  \|f(t)\|_{L_x^6}\lesssim \|\na f(t)\|_{L_x^2},
\end{align*}
for any $\delta>0.$ 
%
For instance, it holds that 
\begin{align*}
  \|Z_3 u_h^R\|_{L_t^2{W}_{tan}^{2,3}}&\lesssim \|\na  Z_3 u_h^R\|_{L_t^2H_{tan}^2}^{\f12}\|Z_3 u_h^R\|_{L_t^2H_{tan}^2}^{\f12}\\
 & \lesssim \big(\|Z_3 \,\omega_h^R\|_{L_t^2H_{tan}^2} +\|\nabla u^R\|_{L_t^2\underline{H}_{tan}^3}\big)^{\f12} \big(\|\phi \, \omega_h^R\|_{L_t^2H_{tan}^2}+\|u_3^R\|_{L_t^2\underline{H}_{tan}^3}\big)^{\f12},\\ 
 \nu^{\f14}\|u^R\|_{L_t^2W_{tan}^{2,\infty}}& \lesssim \|\nu^{\f12}\p_z u^R\|_{L_t^2\underline{H}_{tan}^3}^{\f12}\|\na_h u^R\|_{L_t^2\underline{H}_{tan}^3}^{\f12},\\
(\ep\nu)^{\f12}\|u_h\|_{L_t^2W_{co}^{1,\infty}}&\lesssim \|\ep\nu(\omega_h^R,\na_h u_3^R)\|_{L_t^{2}H_{co}^2}^{\f12}\|u_h^R\|_{L_t^{2}\underline{H}_{co}^3}^{\f12},\\
\tau^{\f12}\|\na\sigma^R\|_{L_t^2W_{co}^{1,3}}&\lesssim \|\tau \na^2 \sigma^R\|_{L_t^2H_{co}^1}^{\f12}\|\na\sigma^R\|_{L_t^2H_{co}^1}^{\f12}, 
\end{align*}
and 
\beq\label{pzuRLinfty}
(\beta_1\nu)^{\f12}\|\p_z u^R\|_{L_t^2L_x^{\infty}}\lesssim \|\beta_1\nu^{\f12}\p_z^2 u^R\|_{L_t^2H_{co}^1}^{\f12}\|\nu^{\f12}\p_z u^R\|_{L_t^2H_{co}^2}^{\f12} \lesssim \eta\cE_T,
\eeq
where $\beta_1=\min\{\ep^{\f23}\nu^{\f13}, (\ep\nu)^{\f12}\}, \tau=\min\{(\nu^{5}/\ep^7)^{\f16}, 1 \}.$ Moreover, 
\begin{align*} 
\|\p_z u_h^R\|_{L_{t,x}^{\infty}}&\lesssim \|\p_z^2u^R\|_{L_t^2H_{tan}^1}^{\f12}\|(\Id, |\na_h|^{2\delta})\p_z u^R\|_{L_t^2H_{tan}^1}^{\f12}\\
 &\lesssim  \min\{\|\na^2 u_h^R\|_{L_t^{\infty}H_{co}^1}^{\f12+2\delta} \|\na u_h^R\|_{L_t^{\infty}H_{co}^1}^{\f12-2\delta},\|\na^2 u_h^R\|_{L_t^{\infty}H_{co}^1}^{\f12} \|\na u_h^R\|_{L_t^{\infty}H_{co}^1}^{\f12}\}\\
 &\lesssim (\ep\nu)^{-(\f12+2\delta)}\beta_1^{-(\f12-2\delta)}\eta\cE_T\lesssim 
 (\ep^{-\f{3+4\delta}{2}}+(\ep\nu)^{-(\f34+\delta)})\eta\cE_T.
\end{align*}

In the remaining part of this section, we prove the corresponding estimates for $\|\div u^R\|_{Y}, \|\na\sigma^R\|_{L_t^2L_x^{\infty}}$ where $Y=L_t^2L_x^{\infty}\cap L^2_tW_{co}^{1,6}.$ To control $\div u^R,$ we notice that
\beno
 \big( \pt+ (u^{a}+u^R)\cdot\na\big) \sigma^{R}+u^R\cdot\na\sigma^a+\f{\rho^{\f{\gamma-1}{2}}\div u^{R}}{\ep}+R^a_{\sigma}=0
\eeno
which implies that
\begin{align*}
\|\div u^R\|_{Y}&\lesssim \|\ep\pt \sigma^R\|_{Y}+\ep\|u^R\cdot\na\sigma^R\|_{Y}+\ep\Lambda(U^a) \|(u^R, Z\sigma^R)\|_{Y}+\ep\|R_{\sigma}^a\|_{Y}\\
&\lesssim \|\na\sigma^R\|_{L_t^2H_{co}^2}+\ep \|(u_h^R, u_3^R/\phi)\|_{L_{T}^{\infty}(L_x^{\infty}\cap W_{co}^{1,6})}\|Z\sigma^R\|_{Y}+\ep^{\f{4-\kappa}{4}}\, \eta \, \cN_T+\ep\,\Lambda(\cM_T^a)
\\
&\lesssim \|\na\sigma^R\|_{L_T^2H_{co}^2} +\big(
\eta+\ep^{\min\{\f23,\f{3-\kappa}{4}\}}\big)\eta\,  (\cN_T+\cN^2_T)+\eta\,\Lambda(\cM_T^a).
\end{align*}
It then follows from Lemma \ref{lem-nasigmadiv-uniform} that 
\begin{align*}
\|\div u^R\|_{Y}\lesssim \eta \big(\cN_0+\Lambda(\cM_T^a)\big)+  \big(T^{\f12}+ \ep^{\min\{\f{1}{40},\f{\kappa}{3},\f{3-\kappa}{4}\}}\big)\eta\, (\cN_T+\cN^2_T) .
\end{align*}
We are now left to control $\|\na\sigma^R\|_{L_{t}^2L_x^{\infty}}.$
The tangential derivative can be controlled directly by using the Sobolev embedding 
\beqs 
\|\na_h\sigma^R\|_{L_{t}^2L_{x}^{\infty}}\lesssim \|\na\sigma^R\|_{L_t^2H_{co}^2}.
\eeqs
To control the normal derivative, we study the equation \eqref{eq-pzsigma}.
In order not to lose derivatives, it is convenient to use  the Lagranian coordinates. Define the unique flow $X_t(x)=X(t,x)$
associated to $u$ 
 \begin{equation}
 \left\{
  \begin{array}{l}
   \partial_t X(t,x)=u(t,X(t,x))\\[3pt]
   X(0,x)=x\in\Omega.
  \end{array}  
  \right.
 \end{equation}
Note that since $u_3|_{\partial{\Omega}}=0,$ and $u\in \Lip([0,T]\times\Omega)$ (although the norm is not necessarily uniform in $\ep,\nu$), we have for each $t\in[0,T],$ $X_t:\Omega\rightarrow\Omega$
 is a diffeomorphism. Consequently, it holds that 
 \begin{align*}
 (\p_z \sigma^R)(t, X_t(x))=e^{-\Gamma(t,x)} (\p_z \sigma^R)(0)+\int_0^t e^{-\Gamma(t-s,x)}\f{1}{\epsilon^2\tilde{\mu}}\big(G_u+G_{\sigma}\big)(s, X_s(x))\,
 \d s
 \end{align*}
 where $\Gamma(t,x)=\frac{1}{\varepsilon^2\tilde{\mu}}\int_0^t \rho^{\gamma}(s, X_s(x))\d s\geq \frac{4^{-\gamma} t}{\epsilon^2\tilde{\mu}}$ since $\rho(t,x)\geq \f14$ for any $(t,x)\in [0,T]\times \Omega.$ Taking the supreme in $x\in \Omega$ on  both 
 sides, and using that $X(t,\cdot)(0\leq t\leq T)$ is a diffeomorphism from $\Omega$ to $\Omega,$ we deduce
\begin{equation}\label{pzsigmaLinfty}
 \|\p_z\sigma^R(t,\cdot)\|_{L_x^\infty}\lesssim e^{-\f{4^{-\gamma}t}{\ep^2\tilde{\mu}}}\|\p_z\sigma^R(0)\|_{L_x^{\infty}(\Omega)}+\int_0^t
 e^{-\frac{4^{-\gamma}(t-s)}{(2\mu+\lambda)\epsilon^2}}\frac{1}{\epsilon^2\tilde{\mu}} \|(G_u+G_{\sigma})(s,\cdot)\|_{L^{\infty}_x}\,\d s\,.
 \end{equation}
We then apply the convolution inequality after zero extension of the above functions to $[T,+\infty]$ to get
\begin{align*} \|\p_z\sigma^R(t,\cdot)\|_{L_t^2L_x^\infty}\lesssim \ep \|\p_z\sigma^R(0)\|_{L_x^{\infty}} +\|(\ep\pt u^R, G_u, G_{\sigma})\|_{L_t^2L_x^{\infty}}.
\end{align*}
In view of the definitions of $G_{u}, G_{\sigma}$ in \eqref{defGusigma} and the definition of $\cA_T$ in \eqref{def-cAT}, we readily get  that 
\begin{align*}
   \|(G_u, G_{\sigma})\|_{L_T^2L_x^{\infty}}\lesssim  \eta \Lambda(\cM_T^a) \big(1+(\eta+ \ep^{\f{3-\kappa}{4}}) (\cN_T +\cN^2_T)\big)+\ep^2(\beta_1\nu)^{-\f12}\|(\beta_1\nu)^{\f12}\p_z u^R\|_{L_t^2L^{\infty}}\|\na\sigma^a\|_{L^{\infty}_{t,x}}.
\end{align*}
The desired estimate for $\p_z\sigma^R$ then follows from the estimate \eqref{pzuRLinfty}
the fact $\ep^2(\beta_1\nu^{\f12})^{-\f12}\lesssim \max\{1, (\ep^7/\nu^3)^{\f14}\}.$
\end{proof}

 The Proposition \ref{prop-unifromes} is then a 
 consequence of  Proposition \ref{prop:est-E-T} and Proposition \ref{prop:est-A-t}. 
 
 We are now in position to prove Theorem \ref{thm1}, which, in turn implies the results stated in Theorem \ref{thm-remainder}.
 To achieve this, we require the following local well-posedness result, which can be derived using arguments similar to those employed in the proof of the local well-posedness of the compressible Navier-Stokes system with Dirichlet boundary conditions, see for instance \cite{loacl-CNS-low,local-CNS}.
 
\begin{thm}\label{classical-local}
    Assume that $(\sigma_0^R,u_0^R)\in L^2(\Omega)$ and that $\sigma_0^R$ is such that 
$$\rho_0(x)=\bigg(\f{\gamma-1}{2}\ep(\sigma_0^{a}+\sigma_0^R)+1\bigg)^{\f{2}{\gamma-1}}\in [1/2, 2],\quad  \forall \, x\in \Omega.$$ 
There exists $T^{\ep,\nu}>0,$ such that \eqref{CNS-remainder} has a strong unique solution in $C([0,T^{\ep,\nu}],H^2(\Omega)).$ Moreover, the velocity $u^R\in L^2([0,T^{\ep,\nu}],H^3(\Omega))$ and the density $\rho(t,x)$ satisfies
    \beq\label{epsigmaLinfty-2}
 \rho(t,x) \in [1/4,4], \qquad \forall \, (t,x)\in [0,T_0]\times \Omega .
 \eeq
\end{thm}

\medskip

{\bf Proof of Theorem \ref{thm1}:} Since $(\sigma_0^R, u_0^R)\in \cN_0,$ a space with higher regularity, by standard propagation of regularity arguments (for example based on applying  finite difference instead of derivatives)
 in the estimates of Section 3, we find that 
 the estimates of  Proposition \ref{prop-unifromes} hold. 
 Specifically, there are polynomials $\Lambda_1, \Lambda_2$ such that for any $T\leq T^{\ep,\nu},$
\begin{align}\label{unifrom-final}
  \mathcal{N}_{T}^2
  \leq C \mathcal{N}_{0}^2+\Lambda_1(\cM_T^a)
  + (T+\ep)^{d({\kappa})}\Lambda_2\big(\cM_T^a,
  \cA_{T}
  \big) \big(\cN_T^2+\cN_T^4\big).
\end{align}
Define
\beqs
T^{\ep,\nu}_{*}=\sup\{T\big| (\sigma^{R},u^{R})\in C([0,T],H^2), u^{R}\in L^2([0,T],H^3)\},
\eeqs
\beqs
\begin{aligned}
T_0^{\ep,\nu}=\sup\big\{T\leq T^{\ep,\nu}_{*}&\big| \cN_{T}\leq 2 M, \, \rho(t,x)\in [1/4,4]
\quad \forall (t,x)\in [0,T]\times\Omega \big\},
\end{aligned}
\eeqs
where $M\geq C\cN_0^2+\sup_{T\leq T^{\ep,\nu}_{*}}\Lambda_1(\cM_T^a).$ 
Denote also $M_1=\sup_{t\leq T^{\ep,\nu}_{*}, x\in\Omega } |\sigma^a(t,x)|.$
 We choose successively two  constants $\ep_0\leq 1$ and $T_0>0$ (uniform in $\ep\in(0,\ep_0]$) which are small enough, such that
 \beqs 
 \ep(2M +M_1)\leq {\f{2}{\gamma-1}} \min\{ (4^{\f{\gamma-1}{2}}-1), (1-1/4^{\f{2}{\gamma-1}})\},
 \eeqs
 \beqs
4({T_0}+\ep_0)^{d(\kappa)}\Lambda_2\big(\cM_T^a, 
  2M\big) \big(1+4M^2\big)<1.
 \eeqs
In order to prove Theorem \ref{thm1}, it suffices for us to show that $T_0^{\ep,\nu}\geq {T_0}$ for any $0<\ep\leq \ep_0, \nu=\ep^{\kappa} (0<\kappa<3).$ Suppose otherwise $T_0^{\ep,\nu}<{T_0}$ for some $0<\ep\leq \ep_0, \nu=\ep^{\kappa},$ then in view of inequalities 
\eqref{unifrom-final} 
we have by the definition of $\ep_0$ and $T_0$ that
\beq\label{sec6:eq3}
\cN_{T}\leq 2M, \qquad \forall\, T\leq \tilde{T}:=\min\{T_0, T^{\ep,\nu}_{*}\},
\eeq
\beq\label{sec6:eq4}
\f14\leq \rho(t,x)\leq 4 \quad \forall (t,x)\in [0,\tilde{T}]\times\Omega.
\eeq
We will prove that $\tilde{T}=T_0\leq T_{*}^{\ep,\nu}.$ This fact, combined with the definition of $T_0^{\ep,\nu}$ and estimates \eqref{sec6:eq3}, \eqref{sec6:eq4},
yields $T_0^{\ep,\nu}\geq T_0,$ which is a contradiction with assumption $T_0^{\ep,\nu}<T_0.$ To continue, we shall need the claim stated and proved below.
 Indeed, once the following claim holds, we have by \eqref{sec6:eq3} that 
 $\|(\sigma^{R},u^{R})(T_0)\|_{H^2(\Omega)}<+\infty,$ which, combined with the local existence result stated in 
 Theorem \ref{classical-local}, yields that $T_{*}^{\ep,\nu}\geq T_0.$

 $\textbf{Claim.}$ 
For all $(\ep, \nu)\in A_0,$
if $\cN_{T}<+\infty,$ then $(\sigma^{R},u^{R})\in C([0,T], H^2),$
 $u^{R}\in L^2([0,T],H^3).$
 \begin{proof}[Proof of claim]
We obtain from the definition of $\cN_{T}$ and the equation of the velocity 
that
$$\ep\nu^{\f32} u^{R}\in 
L^2([0,T],H^3),\quad \ep\nu^{\f12}\pt u^{R}\in L^2([0,T],H^1), \quad \ep \tau\sigma^{R} \in L^{\infty}([0,T],H^2), \, $$
where $\tau=
\min\{(\nu^{5}/\ep^7)^{\f16}, (\nu^2/\ep^3)^{\f12}, 1 \}.$
One deduces from the first two facts and the interpolation that $\ep\nu u^{R}\in C([0,T],H^2).$
Moreover, carrying out direct energy estimates for $\sigma^{R}$ in $H^2(\Omega),$
one gets that:
\beq\label{ineq-gronwall}
\pt g(t)
\leq  K\big(g(t)+f(t)\big)
\eeq
where  $K=\Lambda( \|(u^a, \na\sigma^a, \na u^a)\|_{L_T^{\infty}W^{2,\infty}}+\|(\sigma^R, u^R)\|_{L_T^{\infty}W^{1,\infty}})<+\infty$ 
and $$g(t)=\|\sigma^{R}(t)\|_{H^2}^2,\quad f(t)=\| u^{R}(t)\|_{H^3}\|\sigma^{R}(t)\|_{H^2}\in L^1([0,T]).$$
Inequality \eqref{ineq-gronwall} and the boundedness of $\|\ep\tau g(\cdot)\|_{L^{\infty}([0,T])}$
lead to the fact that $g(\cdot)\in C([0,T]),$ which further yields that $\ep \tau\sigma^R\in C([0,T],H^2).$ We thus finish the proof of this claim.
Note that at this stage we do not require the norm $\|(\sigma^{R},u^{R})\|_{C([0,T], H^2)}$ to be bounded uniformly.

\end{proof}

\section{Remarks for the case $\kappa\geq 3.$}\label{sec-remark-othercase}
In this section, we make some comments on how to prove incompressible and vanishing vertical viscosity limits when the vertical viscosity parameter $\nu$
is much smaller than $\ep$ in the sense that $\nu=\ep^{\kappa}$ with $\kappa\geq 3.$ In this regime, it is indeed simpler to construct the approximation solution in another way in which the error between the approximated solution and the real solution is smaller.
That is, we directly do Taylor expansion for the system \eqref{CNS-O} in $\nu.$
The approximate solution  $U^a=(\rho^a, u^a)$ reads
\begin{align*}
  U^a= \sum_{l=0}^j \sqrt{\nu}^j\big(U^{e,j}(t,x)+ U^{p,j}(t,y, \f{z}{\sqrt{\nu}})+\tilde{U}^{p,j}(t,y, \f{a_3-z}{\sqrt{\nu}})\big),
\end{align*}
$U^e, U^{p,j}, \tilde{U}^{p,j}$ stand respectively for the interior part, lower and upper boundary layer profiles. The first interior profile
 $(\rho^e, u^e):=(\rho^{e,0}, u^{e,0})$ is the solution to the compressible Navier-Stokes with fast oscillation and  with vanishing vertical dissipation 
\begin{align}\label{eq-ue1}
\left\{
   \begin{array}{l}
         \pt \rho^{e}+\div(\rho^e u^e)=0, \\
         \rho^e (\pt +u^e\cdot\na) u^e+\f{\na P(\rho^e)}{\ep^2}-\f{\mu_1}{\rho^e}\Delta_h u^e-\f{\mu_2}{\rho^e}\na\div u^e=0,\\
         (\rho^e, u^e)|_{t=0}=(\rho, u)|_{t=0}; \quad u_3^e|_{\p\Omega}=0.
\end{array}
    \right.
\end{align}
The boundary layer profile  $U^{p,0}=(0, u^{p}_h, \sqrt{\nu}u_3^{p,1})^t$ associated to the lower boundary satisfies the Prandtl equation
\begin{align}\label{comp-Prandtl}
\left\{
 \begin{array}{l}
   \pt u_h^p+(\overline{u_h^e}+u_h^p)\na_h u_h^p+u_h^p\cdot\na_h \overline{u_h^e}+\big(\f{-\int_0^{+\infty}\div_h(\overline{\rho^e} u_h^p)\d \theta}{\overline{\rho^e}}+u_3^{p,1}+\theta\overline{\p_z u_3^e}\big)\p_{\theta} u_h^p \\[5pt]
   \qquad\qquad\qquad\qquad\qquad\qquad-\f{\mu_1}{\overline{\rho^e}}(\p_{\theta}^2u_h^p+\Delta_h  u_h^p)+\mu_2\na_h (\ln \overline{\rho^e})\cdot u_h^p=0 \\
   \div_h (\overline{\rho^e} u_h^p)+\p_{\theta}(\overline{\rho^e}u_3^{p,1})=0,
   \qquad u_3^{p,1}\rightarrow_{z\rightarrow +\infty} 0 , \\
   u_h^p|_{z=0}=-u_h^e|_{z=0}.
 \end{array}
      \right.
\end{align}
The next order interior equation takes the form 
\begin{align}\label{eq-ue2}
\left\{
   \begin{array}{l}
         \pt \rho^{e,1}+\div(\rho^{e,0} u^{e,1}+\rho^{e,1}u^{e,0})=0, \\
        \pt u^{e,1}+u^{e,0}\cdot\na u^{e,1}+u^{e,1}\cdot\na u^{e,0}+\f{\na (g'(\rho^{e,0})\rho^{e,1})}{\ep^2}\\[3pt]
        \qquad   \qquad    \qquad     -\f{\mu_1}{\rho^e}\Delta_h u^{e,1}-\f{\mu_2}{\rho^e}\na\div u^{e,1}-\f{\mu_1}{\rho^{e,1}}\Delta_h u^{e,0}-\f{\mu_2}{\rho^{e,1}}\na\div u^{e,0}=0,\\[4pt]
         (\rho^{e,1}, u^{e,1})|_{t=0}=0, \quad u_3^{e,1}|_{z=0}=-u_3^{p,1},
\end{array}
    \right.
\end{align}
where $g'(x)=\f{P'(x)}{x}.$ 
One can also write down the equation satisfied by the second Prandtl boundary layer, 
but it has been already enough for us to do stability analysis of the remainder. 
Define 
\beqs 
U^a=
 U^{e,0}+(0, u_h^p, \sqrt{\nu} (u_{3}^{p,1}))^t+U^{e,1}+\sqrt{\nu}(0,  \mathfrak{R}_1, \mathfrak{R}_2, 0)^t
\eeqs
where  $\mathfrak{R}_j$ is chosen such that $\mathfrak{R}_j|_{\p\Omega}=
u_j^{e,1}|_{\p\Omega}$ with the property $\|\mathfrak{R}_j(t)\|_{H^4}\lesssim \nu^{\f14}.$
Let us comment that the uniform well-posedness of the system \eqref{eq-ue1}, \eqref{eq-ue2} in the usual Sobolev spaces are not an issue since there is no boundary layers and the singular terms are anti-symmetric. For the boundary layer equation \eqref{comp-Prandtl}, the existence of a solution in the conormal  Sobolev spaces on a time interval independent of $\ep,\nu$ is also not hard to attain, thanks to the horizontal dissipation appears in the system.

The approximate solution solves 
\begin{align*}
\left\{
   \begin{array}{l}
         \pt \rho^{a}+\div(\rho^a u^a)=R_{\rho}^a, \\[3pt]
         (\pt +u^a\cdot\na) u^a+\f{\na g(\rho^a)}{\ep^2}-\f{\mu_1}{\rho^a}(\Delta_h u^a+\nu\p_z^2 u^a)-\f{\mu_2}{\rho^a}\na\div u^a=R_{u}^a,\\[3pt]
         (\rho^a, u^a)|_{t=0}=(\rho, u)|_{t=0}; \quad u^a|_{z=0}=0,
\end{array}
    \right.
\end{align*}
where $\|(R_{\rho}^a, R_u^a)\|_{L_t^{\infty}H_{co}^3}\lesssim \nu^{\f34}.$
Denote $\sigma^R=\f{g(\rho)-g(\rho^e)}{\ep},$ then $\sigma^R$ satisfies
\beqs
\pt\sigma^R+\ep^{-1}g'(\rho)\big(\div (\rho^R u^R)+\div ({\rho^R}(u^a+u^R)) \big)=g'(\rho)R_{\rho}^a/\ep + \f{(g'(\rho)-g'(\rho^e))\pt \rho^e}{\ep}:=R_{\sigma}^a.
\eeqs
We could then write down the equation satisfied by the error $U^R=(\sigma^R, u^R)=(\sigma, u)-(\sigma^a, u^a)$ and prove some uniform regularity estimates for $U^R.$ 
The norms involved would be similar but slightly different as in the definition of $\cN_T.$ The possible changes are the estimates of the compressible part $(\na\sigma^R,\div u^R),$ we should instead prove that 
\begin{align*}
    (\ep\nu)^{\f14}\big(\|(\na \sigma^R,\div u^R)\|_{L_T^{\infty}\underline{H}_{co}^{2}}+ \|\na\div u^R\|_{L_T^{2}\underline{H}_{co}^{2}\cap L_T^{\infty}\underline{H}_{co}^{1}}\big)+(\nu/\ep^3)^{\f14}\|(\nabla\sigma^{R},\div u^R)\|_{L_T^2H_{co}^2\cap L_T^{\infty}\underline{H}_{co}^1}
\end{align*}
and thus the corresponding $L_x^p\,(p=3,6,+\infty)$ norm for $\sigma^R, u_3^R, \na\sigma^R, \div u^R$ should be updated correspondingly.
However, 
the main gains here is that the source term $R^a$ in the equation of $U^R$ has the size of  $\tilde{\eta} =\nu^{\f34}/\ep$ in the Sobolev norm  which is indeed very good when $\nu\leq \ep^{\kappa}$ with $\kappa\geq 3.$
One can use the similar arguments as in  Section 4-5  to close the estimate.
Note that to achieve such a goal, the crucial facts one needs are 

$\bullet$
$\nu^{-\f14}\tilde{\eta} \lesssim 1$ so that $\|u\|_{L_t^2L_x^{\infty}}\lesssim 1,$


$\bullet$ $\|\div u^R\|_{L_t^2L_x^{\infty}}\lesssim 
(\ep^3/\nu)^{\f14}\tilde{\eta}\lesssim ({\nu^2}/{\ep})^{\f14}\lesssim 1.$

$\bullet$
$(\ep\nu)^{\f12}\|\na u^R\|_{L_{t,x}^{\infty}}\lesssim (\ep\nu)^{\f12}\|\na^2 u^R\|_{L_t^{\infty}H_{co}^{1}}^{\f12+\delta}\|\na u^R\|_{L_{t}^{\infty}H_{co}^{1}}^{\f12-\delta}\lesssim (\ep\nu)^{\f12} (\ep\nu)^{-(\f12+\delta)} (\ep\nu)^{-(\f14-\f{\delta}{2})}\tilde{\eta}= {\nu^{\f{1-{\delta}}{2}}}/\ep^{\f{5+2\delta}{4}}\lesssim 1 , $ as long as $\delta$ is chosen small enough.

\appendix
\section{ Study of the group generated by $L.$}\label{app-A}
Recall the linear operator $L=-\left( \begin{array}{cc}
    0 &  \div \\
    \na  & 0 
    \end{array} 
    \right).$
Consider the evolution problem
\begin{align}\label{evolpb}
    \p_{\tau}U-LU=0, \quad U|_{\tau=0}=U^0, \quad U_3|_{\p\Omega}=0\, .  
\end{align}
Denote $V^0=\left\{ f=(f_0,\na q) \in (L^2(\Omega))^4\big|\, \p_3 q|_{\p\Omega}=0\right\}$ and  denote $\tilde{\cL}(\tau)$ the group generated by $L$ in $V^0$ which is an isometry with respect to $(L^2(\Omega))^4.$  To find the explicit form of the semigroup, we 
reduce the problem to the three dimensional torus $\mathbb{T}_a^3=\mathbb{T}_{a_1,a_2}^2\times [-a_3,a_3]=\mathbb{T}_{a_1,a_2}^2\times \mR/(2a_3)$ by using the even and odd extensions:
\beqs 
(e_0 f )(x_h,z)= f(x_h,-z); \quad (e_1 f )(x_h,z)=- f(x_h, -z); \quad \forall\, z<0\,.
\eeqs
Upon defining the extension
$EU=(e_0 U_0, e_0 U_1, e_0 U_2, e_1 U_3)^t,$ we have :
\beqs 
(\p_{\tau}-L)(EU)=0, \quad EU|_{\tau=0}=EU^0\,.
\eeqs
 It is direct to compute that the eigenvalues and eigenfunctions (in $V^0$) of $L$ are:
 \beqs 
\lambda_{k}^{\pm}=\pm |k|, \quad e_{k}^{\pm}=\f{1}{\sqrt{2|\mathbb{T}_a^3|}|k|} e^{ik\cdot x}\left(\begin{array}{c}
    \mp |k|  \\
     k
\end{array}
\right)
 \eeqs
where $k$ is a three-elements vector whose components are defined by 
\beq\label{rela-k'-k}
k_j=\f{2\pi}{a_j}k'_j, (j=1,2);\quad  k_3=\f{\pi}{a_3}k_3'\,,
\eeq
and $k'=(k'_1,k'_2,k'_3)\in \mathbb{Z}^3$ and  
$|\mathbb{T}_a^3|=2 a_1a_2a_3.$  As $\{e_{k}^{\pm}\}$ form the orthogonal  bases in the space $V^0\,,$ 
we  decompose   
$$EU^0(x)=\sum_{k'\in\mathbb{Z}^3}\sum_{\alpha\in\{\pm\}} b_k^{\alpha}e_{k}^{\alpha}, \quad EU(\tau, x) =\sum_{k'\in\mathbb{Z}^3}\sum_{\alpha\in\{\pm\}}b_k^{\alpha}e^{i\alpha|k|\tau}e_{k}^{\alpha}.$$ 
More explicitly, by denoting $EU=(\sigma, u)^t,$ it holds that
\beqs
\sigma=\f{c_{*}}{2}\sum_{k'\in\mathbb{Z}^3}\sum_{\alpha\in\{+,-\}}
-\alpha b_k^{\alpha}e^{i\alpha|k|\tau} e^{ik\cdot x}, \quad u=\f{c_{*}}{2}\sum_{k'\in\mathbb{Z}^3}\sum_{\alpha\in\{+,-\}}b_k^{\alpha}e^{i\alpha|k|\tau} \f{k}{|k|}\,e^{ik\cdot x} , \, \quad \bigg(c_{*}=\f{1}{\sqrt{a_1a_2a_3}}\bigg).
\eeqs
For any three-elements vector $f,$ let the operator $S$ be defined as $Sf=(f_1,f_2,-f_3)^t.$ 
Since $EU$ is supposed to be real and $u(x)=Su(Sx), \sigma(x)=\sigma(S x),$ we readily find that
\beq\label{id-2}
b_k^{+}=-\overline{b_{-k}^{-}}\, , \qquad b_{k}^{\pm}=b_{Sk}^{\pm}\,, 
\eeq
from which we derive 
that
\beq\label{formula-EU}
EU(\tau, t, x)= \sum_{k'\in\mathbb{Z}^3}
\sum_{\alpha\in\{\pm\}} b_k^{\alpha}(t)N_{k}^{\alpha} e^{i\alpha|k|\tau}\, ,
\eeq
where
\beqs 
N_{k}^{\alpha}=\f{c_{*}}{2
|k|}\left(\begin{array}{c}
    -\alpha|k|\cos(k_3 z)  \\
    k_h \cos(k_3 z)\\
    ik_3\sin(k_3 z)
\end{array}\right)e^{ik_h\cdot x_h}.
\eeqs
For instance, the first element can be checked in the following way:
\beqs
\begin{aligned}
\sigma(\tau,t,x)&=\f{c_{*}}{2}\sum_{k'\in \in\mathbb{Z}^3 }\big(b_{Sk}^{-}(t)e^{-i|k|\tau}-b_{Sk}^{+}(t)e^{i|k|\tau}\big)e^{ik\cdot x}\\
&=\f{c_{*}}{2} \sum_{k'\in \in\mathbb{Z}^3 }\big(b_{k}^{-}(t)e^{-i|k|\tau}-b_{k}^{+}(t)e^{i|k|\tau}\big)e^{iSk\cdot x}\\
&= \f{c_{*}}{2}\sum_{k'\in \in\mathbb{Z}^3 } \big(b_{k}^{-}(t)e^{-i|k|\tau}-b_{k}^{+}(t)e^{i|k|\tau}\big)(e^{ik\cdot x}+ e^{iSk\cdot x})/{2}\, .
\end{aligned}
\eeqs
We notice also that $\{N_{k}^{\pm}\}_{k'\in \mathbb{Z}^2\times\mathbb{N}}$ are eigenfunctions of the operator $L$ in $V^0$ and  form the orthogonal basis of the space 
 $V_{\sym}^s, (s\geq 0),$ where the high order symmetric Sobolev space
\beq \label{def-vsym}
V_{\sym}^s=\big\{  f=(f_0,\na q) \in (\dot{H}^s(\Omega))^4\,|\, \p_z^{2k+1}f_0=0, \p_z^{2\ell+1}q=0 \text{ on } \Omega, 2k+1<s-\f12,\, 2\ell<s-\f12  \big\}\, .
\eeq
Restricting the formulae \eqref{formula-EU} to $z>0,$ we find that the solution to the evolution problem \eqref{evolpb} in $V_{\sym}^s$ is given by 
 \beqs 
U(\tau, x)=\sum_{k'\in\mathbb{Z}^3}
\sum_{\alpha\in\{+,-\}} b_k^{\alpha}N_{k}^{\alpha} e^{i\alpha|k|\tau}\, ,
 \eeqs
 where $\{b_{k}^{\alpha}(t)\}$ are the coefficients of initial condition $U^0(x)$ associated to the bases $\{N_k^{\alpha}\}:$
$$ U^0(x)=\sum_{k'\in\mathbb{Z}^3}\sum_{\alpha\in\{+,-\}} b_k^{\alpha}N_{k}^{\alpha} \, .$$ 
It is direct to verify that the semigroup $\cL(\tau)$ generated by $L$ is an isometry in $V_{\sym}^s$ for any $s\geq 0$ and that 
\begin{align*}
\|U(\tau)\|_{V_{\sym}^s}^2=\|U(0)\|_{V_{\sym}^s}^2=\f12\sum_{k'\in \mathbb{Z}^3}\sum_{\alpha\in\{+,-\}}|b_{k}^{\alpha}|^2 |k|^{2s} . 
\end{align*}
For the later use, we also note that since $b_{k}^{\alpha}=b_{S k}^{\alpha},$ 
\begin{align*}
\|U(0)\|_{V_{\sym}^s}^2\approx \sum_{k'\in \mathbb{Z}^2\times \mathbb{N}}\sum_{\alpha\in\{+,-\}}|b_{k}^{\alpha}|^2 |k|^{2s} . 
\end{align*}

\section{Study of the Filtered quantity $W=\cL(-\tau)U^{I,0}.$} \label{app-B}
In this appendix, we study the filtered quantity $W=\cL(-\tau)U^{I,0},$ which solves the equation \eqref{eq-W-1}.
Let $\mathbb{Q}=\left( \begin{array}{cc}
 \Id  & 0 \\
  0 & -\na (-\Delta_N)^{-1}\div \end{array}\right)$ be the projection from $(L^2(\Omega))^4$ to $V^0$ 
 and write $W=W_{osc}+(0,v^{\INS})^t$ with 
 \beqs 
 W_{osc}=\mathbb{Q} W=\sum_{k'\in \mathbb{Z}^3}\sum_{\alpha\in\{+,-\}}  b_k^{\alpha}(t)
 N_k^{\alpha}, \qquad (0,v^{\INS})^t=(\Id -\mathbb{Q})W.
 \eeqs
 \subsection{Derivation of the equations for $W$}
 In order to determine the equations governing $W,$ we still need to compute the dissipation term $\overline{\Delta} W$  and the nonlinear term $\overline{Q}(W,W)$ defined in 
 \eqref{res-dissipation}, \eqref{res-nonlinear}.
By careful studies on the resonant systems 
we find the following:
\begin{prop}
    It holds that
    \begin{align}\label{limit-dissp}
        \overline{\Delta} W=\f12\big(\mu_1\Delta_h+\mu_2\Delta \big)W_{osc}+\mu_2(0,\Delta_h v^{\INS})^t
    \end{align}
and 
\beq\label{nonlinearterm}
\overline{Q}(W,W)=\left( \begin{array}{c}
     0  \\
     v^{\INS}\cdot\na v^{\INS}
\end{array}\right) + \overline{Q}_1(W_{osc},W_{osc})+ \overline{Q}_2\big((0,v^{\INS})^t,W_{osc}\big)\, ,
\eeq
where the limit of  the oscillating-oscillating interactions 
\beq\label{def-osc-osc}
\overline{Q}_1(W_{osc},W_{osc})= \f{\gamma+1}{8}c_{*}\sum_{m'\in \mathbb{Z}^3
}\sum_{
\alpha\in\{+,-\}}i|m| d_m^{\alpha} N_m^{\alpha}\, ,
\eeq
with the coefficient $d_m^{\alpha}$ defined as
\beqs 
d_m^{\alpha}=d_{Sm}^{\alpha}=\sum_{\substack{k+l=m,\\ |k|+|l|=|m|,\\ k_3,l_3\geq 0}}b_{k}^{\alpha}b_l^{\alpha}-\sum_{\substack{k+Sl=m,\\ |k|-|l|=|m|,\\ k_3,l_3\geq 0}}b_k^{\alpha}b_l^{-\alpha}-\sum_{\substack{l+Sk=m,\\|l|-|k|=|m|,\\ k_3,l_3 \geq 0}}b_k^{-\alpha}b_l^{\alpha},  \qquad 
(m_3\geq 0), 
\eeqs
and the limit of mean flow-oscillating interactions 
\beq\label{cross-inter}
\overline{Q}_2\big((0,v^{\INS})^t,W_{osc}\big)=\sum_{m'\in \mathbb{Z}^3
}\sum_{\alpha\in\{+,-\}}i m_h\cdot \bigg(\f{1}{|\Omega|}\int_{\Omega} v_h^{\INS}\d x\bigg) b_m^{\alpha} N_m^{\alpha}\,.
\eeq
\end{prop}

\begin{proof}
  Let us sketch the proof of \eqref{nonlinearterm}, the proof of \eqref{limit-dissp} is simpler and will thus be omitted.  
We split 
\begin{align*}
    \overline{Q}(W,W)&= \overline{Q}\big((0,v^{\INS})^t,(0, v^{\INS})^t\big)+\overline{Q}\big(W_{osc},W_{osc}\big)+
\bigg(\overline{Q}\big(W_{osc},(0,v^{\INS})^t\big)+\overline{Q}\big((0,v^{\INS})^t, W_{osc}\big)\bigg), \\
&:=\mathrm{I}+ \mathrm{II}+ \mathrm{III},
\end{align*}
where 
\begin{align*}
    \overline{Q}(U,V):= \lim_{\tau\rightarrow +\infty}\f{1}{\tau} \int_0^{\tau} \cL(-\tau') Q\big(\cL(\tau')U, \cL(\tau')V\big) \,\d \tau' . 
\end{align*}
We refer to \eqref{def-quadratic} for the definition of the quadratic term $Q(U,V).$
Since $(0,v^{\INS})^t$ belongs to the kernel of $L,$ it is direct to check that 
\begin{align*}
   \mathrm{I}=(\Id-\mathbb{Q})Q\big((0,v^{\INS})^t,(0, v^{\INS})^t\big)=\big(0, \mathcal{P}(v^{\INS}\cdot \na v^{\INS}) \big)^t .
\end{align*}
Let us now turn to the computation of the second term which is the most involved. Using the identities 
$(\na\psi\cdot \na)\na\psi=\f12 \na |\na\psi|^2,$ 
$\sigma\cdot\na\sigma=\f12\na \sigma^2,$
one can verify that $(\Id-\mathbb{Q})\mathrm{II}=0,$ which is consistent with the fact that interactions between the oscillating (compressible) components do not contribute to the mean incompressible flow. We then split $\mathrm{II}=\mathrm{II}_1+\mathrm{II}_2$ with 
\begin{align*}
   \mathrm{II}_1=  \lim_{\tau\rightarrow +\infty}\f{1}{\tau} \int_0^{\tau} \mathbb{Q} \cL(-\tau') \big((U_{osc}'\cdot\na)  U_{osc} \big)(\tau')\,\d \tau' , \\
   \mathrm{II}_2= \f{\gamma-1}{2} \lim_{\tau\rightarrow +\infty}\f{1}{\tau} \int_0^{\tau} \mathbb{Q} \cL(-\tau') \big( U_{osc}^0\, L U_{osc} \big)\,\d \tau', 
\end{align*}
where we denote $U_{osc}=\cL(\tau)W_{osc}=(U_{osc}^0, U_{osc}').$ Let us detail the calculation of $\mathrm{II}_1.$ To continue, it is convenient to define the normalized bases $\tilde{N}_k$ and the modified coefficients  $\tilde{b}_k$
\begin{align*}
\tilde{N}_k^{\pm}=\displaystyle \left\{ \begin{array}{cc}
      2 N_k^{\pm}  & \text{ if } k_3>0 \\
    \sqrt{2}N_k^{\pm} & \text{ if } k_3=0
    \end{array},  \right. \qquad \tilde{b}_k^{\pm}= \displaystyle \left\{ \begin{array}{cc}
     b_k^{\pm}  & \text{ if } k_3>0 \\
    \sqrt{2}^{-1} b_k^{\pm} & \text{ if } k_3=0
    \end{array} \right.
\end{align*}
so that 
$\|\tilde{N}_k^{\pm}\|_{(L^2(\Omega))^4}=1$ and $U_{osc}=\sum_{k'\in \mathbb{Z}^2\times\mathbb{N}}\sum_{\alpha\in\{+,-\}}\tilde{b}_k^{\alpha}\tilde{N}_k^{\alpha}e^{i\alpha|\tau|}.$ 
Direct computations show that 
\begin{align*}
    \mathbb{Q} \cL(-\tau') \big((U_{osc}'\cdot\na)  U_{osc} \big)(\tau')=\sum_{k',l',m',\in \mathbb{Z}^2\times\mathbb{N}}\sum_{\alpha,\beta,\nu\in\{+,-\}}
    g_{k l m}^{\alpha\beta\nu} \tilde{N}_{m}^{\nu}
    e^{i\tau'(\alpha |k|+\beta|l|-\nu|m|)}
\end{align*}
where $ g_{k l m}^{\alpha\beta\nu}:={i} \, \tilde{b}_k^{\alpha} \tilde{b}_{l}^{\beta}\f{c_{*}^2}{2}\bigg\langle \f{1}{|k||l|} \left( \begin{array}{c}
     -\beta |l| F_{kl}^1     \\[1pt]
     l_h  F_{kl}^1   \\[1pt]
      i\,l_3  F_{kl}^2 
   \end{array} \right)e^{i (k_h+l_h)\cdot x_h}  
   ,  \tilde{N}_{m}^{\nu} \bigg\rangle $ 
with 
\begin{align*}
F_{kl}^1=[(k\cdot l) \cos((k_3+l_3)z)+(k\cdot Sl) \cos ((k_3-l_3)z)] \chi(k_3)\chi(l_3),\, \\
F_{kl}^2=[(k\cdot l) \sin((k_3+l_3)z)-(k\cdot Sl)\sin ((k_3-l_3)z)] \chi(k_3)\chi(l_3). 
\end{align*} 
Here the function $\chi: \mathbb{N}\rightarrow \mathbb{R} $ is such that $\chi(0)=\f{1}{\sqrt{2}}, \chi(k_3)=1$ for any $k_3>0.$ The resonant system satisfied by $(k,l,\alpha,\beta)$ that has nontrivial contribution to $I$ is such that 
$$\alpha |k|+\beta |l|=\nu |m|; \quad k_h+l_h=m_h; \quad k_3+l_3=m_3 \text{ or } |k_3-l_3|=m_3$$ which gives the resonant set $\cA_{m}^{\nu}=\cA_{m,1}^{\nu}+\cA_{m,2}^{\nu}+\cA_{m,3}^{\nu}$ where 
\begin{align*}  
&\cA_{m,1}^{\nu}=\{ k,l\in \mathbb{Z}^2\times\mathbb{N}|\, k+l=m, 
k=\lambda m\, (\lambda>0 \text{ if } m_3> 0), \,\, \alpha=\beta=\nu \}, \\
&\cA_{m,2}^{\nu}=\{ k,l\in \mathbb{Z}^2\times\mathbb{N}|\, k+Sl=m, \, k\cdot Sl=-|k||l|, \,\, \alpha=-\beta=\nu \},\\
&\cA_{m,3}^{\nu}=\{ k,l\in \mathbb{Z}^2\times\mathbb{N}|\, l+Sk=m, \, Sk\cdot l=-|k||l|, \,\, -\alpha=\beta=\nu \}.
\end{align*}
Therefore, by using the symmetry of $A_{m}^{\nu},$ further calculations lead to that 
\begin{align*}
\mathrm{II}_1&=\sum_{m',\nu}\tilde{N}_{m}^{\nu}\sum_{(k,l,\alpha,\beta)\in A_{m}^{\nu}} \f12 \big(g_{k l m}^{\alpha\beta\nu}+  g_{lk m}^{\beta\alpha\nu} \big)=
\sum_{m',\nu}{N}_{m}^{\nu} \sum_{(k,l,\alpha,\beta)\in A_{m}^{\nu}} \f{i c_{*}}{4}  b_{k}^{\alpha} b_l^{\beta} \alpha\beta |m|,
\end{align*}
where hereafter, we use the shorthand notation $\sum_{m',\nu}=\sum_{m'\in\mathbb{Z}^2\times \mathbb{N},\nu\in\{+,-\}}.$
Following the similar computations, we find that
\begin{align*}
\mathrm{II}_2&=
\sum_{m',\nu}{N}_{m}^{\nu} \sum_{(k,l,\alpha,\beta)\in A_{m}^{\nu}} \f{i (\gamma-1) c_{*}}{8}  b_{k}^{\alpha} b_l^{\beta} \alpha\beta |m|.
\end{align*}
The above two identities then give that 
\begin{align*}
\mathrm{II}=
\sum_{m',\nu}{N}_{m}^{\nu} \sum_{(k,l,\alpha,\beta)\in A_{m}^{\nu}} \f{i (\gamma+1) c_{*}}{8}  b_{k}^{\alpha} b_l^{\beta} \alpha\beta |m|,
\end{align*}
which is \eqref{def-osc-osc}. 

It remains to compute the last term $\mathrm{III}.$ Again, it can be verified that $(\Id-\mathbb{Q})\mathrm{III}=0.$ Since $ v^{\INS}$ is divergence-free, it can be written in the form 
$$v^{\INS}=\sum_{k'\in \mathbb{Z}^2\times  \mathbb{N}}\bigg( \begin{array}{c}
    v_h^{k_3}(k_h)\cos (k_3 z)   \\[2pt]
      -ik_h\cdot v_h^{k_3}(k_h)\f{\sin(k_3 z)}{k_3}
\end{array}\bigg), \qquad (k_h\cdot v_h^{0}=0)$$
One can then compute that 
\begin{align*}
     \mathbb{Q} \cL(-\tau')\big(v^{\INS}\cdot \nabla U_{osc}+U_{osc}'\cdot\nabla (0, v^{\INS})^t \big)=\sum_{m',l'}\sum_{\beta,\nu} \tilde{N}_m^{\nu} e^{i(\beta|l|-\nu|m|)\tau} \tilde{b}_l^{\beta} \sum_{k'} h_{lmk}^{\beta\nu},
\end{align*}
where $h_{lmk}^{\beta\nu}= \f{C_{*}}{2|l|} \left\langle 
\left( \begin{array}{c}
     -\beta |l| B_{kl}^1     \\[1pt]
     l_h  B_{kl}^1+ v_h^{k_3}C_{kl}^1   \\[1pt]
      i\,l_3  B_{kl}^2 - \f{ik_h}{k_3}\cdot v_h^{k_3} C_{kl}^2
   \end{array} \right)e^{i (k_h+l_h)\cdot x_h} ,\tilde{N}_m^{\nu}\right\rangle, $
   with  
   $$B_{kl}^1=i \big[\, (d \cdot l) \cos ((k_3-l_3)z)+(d\cdot Sl) \cos((k_3+l_3)z)\big] \chi(l_3), \quad \big(d=(v_h^{k_3}, (k_h\cdot v_h^{k_3})/{k_3})^t\big),$$
    $$B_{kl}^2= \big[\, (d \cdot Sl) \sin ((k_3+l_3)z)-(d\cdot Sl) \sin((k_3-l_3)z)\big] \chi(l_3), $$
   $$C_{kl}^1=F_{kl}^1/\chi(k_3), \quad C_{kl}^2=[(k\cdot l) \sin((k_3+l_3)z)+(k\cdot Sl)\sin ((k_3-l_3)z)] \chi(l_3).  $$
The resonant system is $$|l|=|m|, \quad \beta=\nu, \quad  k+l=m \text{ or } k+Sl=m \text{ or } Sk+l=m  $$ from which we find the resonant set $B_m^{\nu}=B_{m,1}^{\nu}+B_{m,2}^{\nu},$ where 
\begin{align*}
    B_{m,1}^{\nu}=\{ l=m, k=0, \beta=\nu \}, \qquad  B_{m,2}^{\nu}=\{ -S l={k}/{2}=m, \, \beta=\nu \}.
\end{align*}
However, direct computations show that all the coefficients $h_{lmk}^{\beta\nu}$ vanish on  $B_{m,2}^{\nu}$ so that
\begin{align*}
    \mathrm{III}=\sum_{m',\nu}  \tilde{N}_m^{\nu}
\tilde{b}_{m}^{\nu} h_{mm0}^{\nu\nu}=\sum_{m',\nu} \tilde{N}_m^{\nu} 
\tilde{b}_{m}^{\nu} (i m_h\cdot v_h^0(0)), 
\end{align*}
which is \eqref{cross-inter}.

\end{proof}

\begin{rmk}
 Note that the interactions of the mean flow $v^{^{\INS}}\cdot\na v^{\INS}$ have no effects on the limit of the oscillating part. 
Consequently, taking the projection $\mathbb{Q}, \Id-\mathbb{Q}:=\diag(0, \mathcal{P})\, $ on \eqref{eq-W-1},
 we arrive that
 \beq \label{eq-Wosc}
\pt W_{osc}+ \overline{Q}_1(W_{osc},W_{osc})+ \overline{Q}_2\bigg((0,v^{\INS})^t,W_{osc}\bigg)-\f12\big(\mu_1\Delta_h+\mu_2\Delta \big)W_{osc}+\sqrt{\f{\nu}{\ep}}\overline{S}=0\, ,
\eeq
\beq\label{eq-vINS}
\pt v^{\INS}+\mathcal{P}\big(v^{\INS}\cdot \na v^{\INS}-{\mu_1}\Delta_h v^{\INS} \big)=0\,.
\eeq
\end{rmk}

\medskip

\subsection{Study of $W_{osc}-$ well-posedness and damping}
By \eqref{defS-barS},
the system \eqref{eq-Wosc} for  $W_{osc}$ is equivalent to the following infinite collections of coupled (damped Burgers) equations for $b_{m}^{\alpha}:$
\beqs \label{bmpm}
(\text{DB})_m^{\alpha}: \quad  \, \pt b_m^{\alpha}+\f{\gamma+1}{8}c_{*} i|m|d_m^{\alpha}+im_h\cdot A(v_h^{\INS}) b_m^{\alpha}+\f12\big(\mu_1|m_h|^2+\mu_2|m|^2\big)b_m^{\alpha}+ \sqrt{\f{\nu}{\ep}}( -\lambda_{m,1}^{\alpha})b_m^{\alpha}=0 
\eeqs
where $\lambda_{m,1}^{\pm}=-\f{2(1\pm i)}{a_3}\sqrt{\f{\mu_1}{2}}  \f{|m_h|^2}{|m|^{{3}/{2}}}$ and the real constant $A(v_h^{\INS})=\int_{\Omega} v_h^{\INS}\,\d x$ denotes the zero mode of $v_h^{\INS}.$  Note that the interaction between the mean flow and 
oscillating part is very weak and has no effects on the energy estimates of oscillating part. 
In the following, we only establish the a prior estimates, which, together with some regularization arguments, leads to the local well-posedness of system \eqref{bmpm}.
\begin{lem}\label{well-posedness-W}
Assume that $
W_{osc}(0)\in V_{\sym}^s (s\geq 5/2).$
    There exists a $T>0,$ the system \eqref{bmpm} with the initial condition \eqref{initial-w} admits a solution in $C([0,T], H^s(\Omega))$  which satisfies
    \beqs
    \|W_{osc}\|_{L^{\infty}({[0,T]}, H^s)}\lesssim  \|W_{osc}(0)\|_{H^s}\, .
    \eeqs
\end{lem}

\subsubsection{A priori estimates in $L^2$}
Multiplying $(\text{DB})_m^{\alpha}$ by $\overline{b_{m}^{\alpha}},$ summing up in $m'\in \mathbb{Z}^2\times \mathbb{N}$ and $\alpha\in\{\pm\}$ and taking the real part we get that
\beq\label{EnergyId-bm}
  \sum_{m', \alpha}\bigg(\f{\pt}{2}+ 
 \sqrt{\f{2\mu_1}{a_3^2}}\sqrt{\f{\nu}{\ep}}  \f{|m_h|^2}{|m|^{{3}/{2}}}+\f12\big(\mu_1|m_h|^2+\mu_2|m|^2\big)\bigg) |b_m^{\alpha}|^2=-\f{\gamma+1}{8}c_{*}\Re \bigg(\sum_{m', \alpha} i|m|d_m^{\alpha} \overline{b_m^{\alpha}}\bigg)
=0. 
\eeq
Let us check the right hand side. Indeed, 
it follows from the relation \eqref{id-2} that
\beqs
\begin{aligned}
   &\qquad  \sum_{m', \alpha} i|m|d_m^{\alpha} \overline{b_m^{\alpha}}
   \\
    &= \sum_{m',\alpha} i|m|\bigg[
    -\sum_{\substack{k+l=m,\\ |k|+|l|=m,\\ k_3,l_3 \geq 0}}b_{k}^{\alpha}b_l^{\alpha} b_{-Sm}^{-\alpha}+\sum_{\substack{k+Sl=m,\\ |k|-|l|=|m|,\\ k_3,l_3 \geq 0}}b_k^{\alpha}b_l^{-\alpha}b_{-Sm}^{-\alpha}+\sum_{\substack{l+Sk=m,\\|l|-|k|=|m|,\\ k_3,l_3 \geq 0}}b_k^{-\alpha}b_l^{\alpha}b_{-Sm}^{-\alpha}\bigg]\\
    &= \sum_{m',\alpha} \sum_{\substack{k+l=m,\\ |k|+|l|=m,\\ k_3,l_3 \geq 0}}(|k|+|l|-|m|)b_{k}^{\alpha}b_l^{\alpha} b_{-Sm}^{-\alpha}=0.
\end{aligned}
\eeqs
Note that we have redenoted $(k,l,m)$ by $(-Sm, l,-Sk)$ and $(k,-Sm,-Sl)$ in the last two summations.

\subsubsection{A priori estimates in $H^s\, (s>0)$}
Similar to the $L^2$ estimates,
for any $\beta\in\mathbb{N}^3, |\beta|=s,$ 
we have that 
\beqs\label{EnergyId-bm}
\begin{aligned}
  \sum_{m', \alpha}\bigg(\f{\pt}{2}+ 
 \sqrt{\f{2\mu_1\nu}{a_3^2\ep}}  \f{|m_h|^2}{|m|^{{3}/{2}}}\bigg) |m^{\beta}b_m^{\alpha}|^2&=-\f{\gamma+1}{8}c_{*}\Re \bigg(\sum_{m', \alpha} i|m|m^{2\beta}d_m^{\alpha} \overline{b_m^{\alpha}}\bigg).
 \end{aligned}
\eeqs
Since 
\begin{align*}
\sum_{m', \alpha}|m|m^{2\beta}d_m^{\alpha} &=\sum_{\substack{k+l=m, \\|k|+|l|=|m|}}(|m|m^{2\beta}-|k|(-Sk)^{2\beta}-|l|(-Sl)^{2\beta})b_k^{\alpha} b_l^{\alpha}\\
&\lesssim \sum_{\substack{k+l=m}}|b_k^{\alpha}||b_l^{\alpha}| |m|^s(|k|^s|l|+|l|^s|k|),
\end{align*} 
the right hand side can be controlled as
\beqs 
\sum_{m', \alpha} |b_m^{\alpha}|^2|m|^{2s} \sup_{k,\alpha}|b_k^{\alpha}||k|\lesssim \|W_{osc}\|_{H^s}^3
\eeqs
as long as $s>{5}/{2}.$ 
We thus obtained the desired a-prior estimates $$\|W_{osc}(t)\|_{H^s}^2\lesssim \|W_{osc}(0)\|_{H^s}^2+ t \sup_{0\leq t' \leq t}\|W_{osc}(t')\|_{H^s}^3. $$ 

\subsubsection{The damping} The following result concerns the damping of the interior oscillating part $W_{osc}:$
\begin{lem}
    It holds that for any $r\geq 0,$
\beqs 
\|\na_h W_{osc}\|_{L_t^2H^r}\lesssim (\f{\ep}{\nu})^{\f{1}{4}}\|W_{osc}|_{t=0}\|_{H^{r+\f{3}{4}}}.
\eeqs
\end{lem}
\begin{proof}
It results from \eqref{EnergyId-bm} that 
\beqs 
 \sum_{\substack{m'\in \mathbb{Z}^2\times \mathbb{N}},\\ \alpha\in\{+,-\}}  \sqrt{\f{2\mu_1}{a_3^2}}\sqrt{\f{\nu}{\ep}}  \f{|m_h|^2}{|m|^{{3}/{2}}} \int_0^t |b_m^{\alpha}|^2 \leq \sum_{\substack{m'\in \mathbb{Z}^2\times \mathbb{N}},\\ \alpha\in\{+,-\}} |b_m^{\alpha}|^2(0)<+\infty.
\eeqs
\end{proof}
\subsubsection{The boundedness of the time derivative.}
The following estimates for $\pt b_k^{\alpha}$ is useful in the control of the $\cB_h'$ defined in \eqref{defcB'h}.
\begin{lem}\label{lem-elem}
For any $s\geq 0,$ it holds that
\beq
\sum_{\alpha=\{+,-\}}\bigg(\langle m\rangle^s\int_0^t |\pt b_m^{\alpha}|\bigg)_{\ell_m^2}\lesssim \|W_{osc}(0)\|_{H^{s}}+(1+t)\|W_{osc}\|_{L_t^{\infty}H^{s+1}}^2.
\eeq
As a consequence,
\beq \label{es-ptb}
\sum_{m',\alpha}\langle m\rangle ^s \int_0^t  |\pt b_m^{\alpha}(t')| \d t' \lesssim \|W_{osc}(0)\|_{H^{s+(3/2)^{+}}}+(1+t)\|W_{osc}\|_{L_t^{\infty}H^{s+(5/2)^{+}}}^2.
\eeq
\end{lem}
\begin{proof}
This estimate follows directly from the explicit formulae
\beqs 
b_m^{\alpha}(t)=e^{r(m)t\sqrt\f{\nu}{{\ep}}}b_m^{\alpha}(0)
-\int_0^t e^{r(m)(t-s)\sqrt\f{\nu}{{\ep}}} \big(\f{\gamma+1}{8}c_{*} i|m|d_m^{\alpha}\big)\, \d s
\eeqs
where $r_m=\lambda_{m,1}^{\alpha}-\f12({\mu_1|m_h|^2 +\mu_2|m|^2})\sqrt{\f{\ep}{\nu}}+i m_h\cdot A(v_h^{INS})\sqrt{\f{\ep}{\nu}}.$ 
\end{proof}
As a consequence of the estimate \eqref{es-ptb},
we also have the following estimate for the error term $\cB_h^{'0}$ defined in \eqref{defcB'h}:
\beq\label{prop-bhp0}
\|{\cB'}_h^0\|_{L_t^1H_{co}^s}+ (\ep\nu)^{\f{1}{2}} \|\p_z{\cB'}_h^0\|_{L_t^1H_{co}^{s-1}}\lesssim (\ep\nu)^{\f{1}{4}} (\|W_{osc}(0)\|_{H^{s}}+\|W_{osc}(0)\|_{H^{s+(5/2)^{+}}}^2).
\eeq   

\section{Study of the approximate solution}
\subsection{The oscillating  profiles}
\subsubsection{The interior profiles
$U^{I,0}, S_{osc}^{I,1}, U_{osc}^{I,1} $} 
Let $W$ solves the equation \eqref{eq-W-1} with the initial condition 
 \beqs \label{initial-w}
W(0)
=\sum_{k'\in \mathbb{Z}^3
}\sum_{\alpha\in\{+,-\}}  b_k^{\alpha}(0)
 N_k^{\alpha}+
 (0,v_0^{\INS})^t:=W_{osc}(0)+(0,v_0^{\INS})^t\, .
\eeqs
We first have the following property for the first interior profile $U^{I,0}=\cL(\tau)W$ which is a consequence of the well-posedness of $W$ stated in Lemma \ref{well-posedness-W} and the isometry of $\cL(\tau)$ in $V_{\sym}^{m}$ defined in \eqref{def-vsym} and the standard well-posedness of the incompressible Navier-Stokes equations \eqref{eq-vINS}.
\begin{thm}\label{lem-UI0}
  Assume that $W_{osc}(0)\in V_{\sym}^{m} (m\geq 0),$  
  then it holds that, for any $t\leq T,$
\beq  
\|U^{I,0}\|_{L_t^{\infty} H^{m}}\lesssim \|W(0)\|_{H^{m}}
\,,
\eeq
\beq 
\|\na_h U_{osc}^{I,0} \|_{L_t^2H^{m-\f34}}\lesssim ({\ep}/{\nu})^{\f{1}{4}}\|
W(0)\|_{H^{m}}.
\eeq
\end{thm}

Moreover,
as a direct consequence of 
the explicit expression \eqref{defS-barS}, we have the following result for $S_{osc}^{I,1}:$
\begin{thm}
For any $s\geq 0, 1\leq p \leq +\infty,$  it holds that
    \beq\label{SoscI1}
    \|S_{osc}^{I,1}\|_{L_t^{p}H^s}\lesssim \|U_{osc}^{I,0}\|_{L_t^{p}H^{s+\f12}}\, .
\eeq
\end{thm}
We now study the second interior profile $U_{osc}^{I,1}:=V-\sqrt{\f{\nu}{\ep}}S_{osc}^{I,1},$ where $V$ is defined in \eqref{def-V}. Based on the above estimate for $S_{osc}^{I,1},$ it is enough for us to show the following result on $V:$
\begin{lem}\label{lem-V}
Assume that $a_1, a_2, a_3$ are such that the small divisor estimate \eqref{no-resonantasp} holds for some $r_0>0,$ then for any $s\geq 5/2,$
\beq\label{es-V}
\begin{aligned}
\|\ep V\|_{L_t^{\infty}H^s}&\lesssim \ep(1+t) \bigg(\|U^{I,0}\|_{L_t^{\infty}H^{s+r_0+2}}^2+\|U^{I,0}\|_{L_t^{\infty}H^{s+r_0+2}}^3\bigg)\\
&\lesssim \ep(1+t)\big(\|W(0)\|_{H^{s+r_0+2}}^2+\|W(0)\|_{H^{s+r_0+2}}^3\big)\, .
\end{aligned}
\eeq
\end{lem}
\begin{proof}
    Let us give the estimate of the first term in  \eqref{def-V}, namely the nonlinear term 
    \beqs
   \cL(\tau) \int_0^{\tau}\bigg(-\cL(-\tau') Q\big(
U^{I,0},U^{I,0}\big)
+\overline{Q}(W,W)\bigg) \d \tau',
\eeqs
the other term concerning the horizontal dissipation is easier to control whose estimate is thus omitted. 
Since $\cL(\tau)$ is unitatry in $V_{\sym}^s,$ it suffices to show that  $\int_0^{\tau}\bigg(-\cL(-\tau') Q\big(
U^{I,0},U^{I,0}\big)
+\overline{Q}(W,W)\bigg) \d \tau'$ is uniformly bounded in $H^s$ and can be controlled by the right hand side of \eqref{es-V}.
By the definition of $\overline{Q}(W,W),$ it could be written in the form
\begin{align*}
\sum_{\omega_{kl}^{\alpha\beta}\neq 0} N_l^{\beta} \int_0^t
F_{1,l}^{\beta}\bigg(\left(\begin{array}{c}
   0  \\ 
   v^{\INS}
\end{array}\right), b_k^{\alpha} N_{k}^{\alpha} \bigg)  e^{i\omega_{kl}^{\alpha\beta}\f{t'}{\ep}}\f{\d t'}{\ep}  +\text{symmetric term} \\
+\sum_{\omega_{klm}^{\alpha\beta\gamma}\neq 0} N_m^{\gamma}  \int_0^t F_{2,m}^{\gamma}\bigg( b_k^{\alpha} N_{k}^{\alpha}, b_l^{\beta} N_{l}^{\beta}\bigg)   e^{i\omega_{klm}^{\alpha\beta\gamma}\f{t'}{\ep}}\f{\d t'}{\ep}:= J_1(t)+J_2(t), 
\end{align*} 
where $\omega_{kl}^{\alpha\beta}=\alpha |k|-\beta |l|,$  
$\omega_{klm}^{\alpha\beta\gamma}=\alpha |k|+\beta |l|-\gamma |m|,$  and $ F_{1,l}^{\beta}, F_{2,m}^{\gamma}$ are some polynomials with degree 2. 
As $|\omega_{kl}^{\alpha\beta}|\geq 1$ if it does not vanish, we can first integrate by parts in time and then take the 
$H^s$ norm. By using the convolution inequality, we find, for $s\geq 5/2,$
\beqs 
\|J_1(t)\|_{H^s}\lesssim \|W\|_{L_t^{\infty}H^s}^2+ \| v^{\INS}\|_{L_t^{\infty}H^s}
\bigg(\langle k\rangle^s\int_0^t |\pt b_k^{\alpha}|\bigg)_{\ell_k^2}
+\|\pt v^{\INS}\|_{L_t^{2}H^s}\| W_{osc}\|_{L_t^{2}H^s}\, .
\eeqs
Note that thanks to Lemma \ref{lem-elem}, it holds that 
$$
\bigg(\langle k\rangle^s\int_0^t |\pt b_k^{\alpha}|\bigg)_{\ell_k^2}\lesssim \|W_{osc}(0)\|_{H^{s}}+(1+t)\|W_{osc}\|_{L_t^{\infty}H^{s+1}}^2.$$

The next term is controlled in a similar manner. One has, 
if the small divisor estimate \eqref{no-resonantasp} holds true, 
\beqs 
\|J_2(t)\|_{H^s}\lesssim \|W\|_{L_t^{\infty}H^{s+r_0}}^2+ \|W_{osc}\|_{L_t^{\infty}H^{s+r_0}}\bigg(\langle k\rangle^{s+r_0+1}\int_0^t |\pt b_k^{\alpha}|\bigg)_{\ell_k^2}\, . 
\eeqs
\end{proof}

\subsubsection{The boundary layers oscillating boundary layer $U_{osc}^{B,0}=(0,\cB_h^0, \sqrt{\ep\nu}\cB_3^1)^t$}
\begin{lem}
    Let $\cB_h^0, \cB_3^1$ be defined in \eqref{def-cB_h^0}, \eqref{bo-bl-3} it holds that
    \begin{align}\label{cBh-Linfty}
      \|
     \cB_h^0 \|_{L_t^{\infty}H_{co}^s}+\|\cB_3^1\|_{L_t^{\infty}H_{co}^{s-1/2}} + \sqrt{\ep\nu}  (\|\p_z\cB_h^0
     \|_{L_t^{\infty}H_{co}^{s-1}}+\|\p_z\cB_3^1 
     \|_{L_t^{\infty}H_{co}^{s-3/2}})\lesssim (\ep\nu)^{1/4}\|W_{osc}\|_{L_t^{\infty}H^{s+\f14}}.
    \end{align}
 Moreover, it may behave better when taking $L_t^2$ type norm: 
\begin{align}\label{cBh-L2}
  \|\cB_h^0\|_{L_t^2H_{co}^s}  + \sqrt{\ep\nu}  \|\cB_h^0\|_{L_t^2H_{co}^{s-1}}\lesssim \min\{(\ep\nu)^{1/4}, \ep^{\f12}\}\|W_{osc}\|_{L_t^{\infty}H^{s+1}\cap L_t^2H^{s+2^{+}}}.
\end{align}
\begin{rmk}
   In view of \eqref{cBh-L2}, we have indeed established a stronger estimate than required (specifically, the factor on the right-hand side being $(\ep\nu)^{\f14}$) when $\nu < \ep$. This improvement is attributed to the damping effect caused by the strong boundary layers.
\end{rmk}

\end{lem}
 \begin{proof}
     We shall present only the proof of \eqref{cBh-L2} for $\cB_h^0$ and for the case $s=0,$ the estimate \eqref{cBh-Linfty} is easier and the case of $s>0$ is similar.
     By the explicit formulae for $\cB_h^0$ in \eqref{def-cB_h^0}, 
we compute
\beqs 
\int_0^t  \|\cB_h^0\|_{L^2(\Omega)}^2 \, \d s=a_3^2\sum_{\substack{k', l' \in \mathbb{Z}^2\times \mathbb{N}, \\k_h'=l'_h\neq 0}}\sum_{\alpha,\beta \in\{+,-\}} \f{|k_h|^2}{|k||l|} \int_0^{a_3}e^{-\f{z}{\sqrt{\ep\nu}}d(k,l)}\d z \int_0^t b_k^{\alpha} \overline{b_l^{\beta}} e^{\phi(k,l){s}/{\ep}}\,\d s
\eeqs
 where $d(k,l)=\sqrt{|k|}(1+\alpha i)+\sqrt{|l|}(1-\beta i), \,\phi(k,l)= \alpha |k|-\beta|l|.$  
 We split the above summation into two pieces, depending on $\phi(k,l)$ vanishes or not and denote them as $I_1$ and $I_2.$ For the first piece, the non-zero contribution comes merely from the summation where
 $\alpha=\beta, k=l.$ It thus holds  that
 \beqs 
|I_1|\lesssim \sqrt{\ep\nu}\sum_{k',\alpha}a_3^2 \f{|k_h|^2} {|k|^{5/2}} \int_0^t |b_k^{\alpha}|^2 (s)\,\d s\lesssim \sqrt{\ep\nu}\min\{ 1, \sqrt{\ep/\nu} \} \|W_{osc}(0)\|_{L^2}^2 .
 \eeqs
 Note that in the last inequality, we have used the damping of $b_k^{\alpha}$ proved in \eqref{EnergyId-bm}.

 As for the second piece $I_2$, we take benefits of the fast oscillations in the time integral. 
More precisely, by integrating by parts in time, one has 
\beqs 
|I_2|\leq \ep \sqrt{\ep\nu}\sum_{\substack{k',l',\alpha,\beta, \\\phi\neq 0}}\sum_{} \f{|k_h|^2}{|k||l|} 
\f{1}{| \phi d(k,l)|}\bigg(\bigg| \int_0^t \p_s(b_k^{\alpha}\overline{b_l^{\beta}})
\d s\bigg|+ \bigg|( b_k^{\alpha}\overline{b_l^{\beta}} e^{is\phi/\ep} )\big|_{0}^t\bigg|\bigg):= I_{21}+ I_{22}.
\eeqs
Note that once $\phi(k,l)\neq 0,$ it holds that $|\phi(k,l)|\gtrsim (|k|+|l|) \mathbf{I}_{\{\alpha\neq\beta\}}+\f{k_3^2-l_3^2}{|k|+|l|}\mathbf{I}_{\{\alpha=\beta,k_3\neq l_3\}},$ this together with 
$|d(k,l)|\gtrsim \sqrt{|k|+|l|},
$ yields that
\begin{align*}
    I_{21}&\lesssim \ep\nu \int_0^t\sum_{k', l', k_h'=l_h'} \sum_{\alpha,\beta} \f{|k_h|^2}{|k||l|} \f{\sqrt{|k|+|l|}}{|k_3|+|l_3|+1}  |\sqrt{\ep/\nu}\pt b_k^{\alpha}||b_l^{\beta}|\, \d s \\
    & \lesssim \ep\nu \int_0^t \sum_{k',\alpha} (1+|k_3|)^{0^{+}} |\sqrt{\ep/\nu}\pt b_k^{\alpha}|^2 
    + \sum_{l', \beta}(1+|k_3|)^{0^{+}} 
   |k_h| |b_l^{\beta}|^2 \,\d s\\
    & \lesssim \ep\nu \big(\|\sqrt{\ep/\nu}\pt W_{osc}\|_{L_t^2 H^{0^{+}}}^2+ \|
    W_{osc}\|_{L_t^2 H^{{\f{1}{2}}^{+}}}^2\big)\lesssim \ep\nu .
\end{align*}
The second term $I_{22}$ can be controlled in a similar way
\beqs 
 I_{22}\lesssim \ep \sqrt{\ep\nu}  \| W_{osc}\|_{L_t^{\infty} H^{{\f{1}{4}}^{+}}}^2 .
\eeqs
 We thus proved that $\|\cB_0^h\|_{L_t^2L^2(\Omega)}\lesssim \min\big\{ {\ep\nu}^{\f{1}{4}}, \ep^{\f12}\big\},$ the case for $s>0$ being similar. 
 \end{proof}
\subsection{The non-oscillating profiles}
\subsubsection{Well-posedness of the boundary layer $v^{p,0}$}
To state the result, we introduce some notations. 
Denote $\tilde{\Omega}=\mathbb{T}_{a_1,a_2}^2\times \mR_{+}$ and let $\delta, \lambda$ be two positive numbers.
For any $t\in (0, 1/\lambda],$ define the norm
\begin{align*}
\|f(t)\|_{H_{\delta,\lambda}^s}^2:=\sum_{j+k+l\leq s}\int_{\tilde{\Omega}} e^{2(2-\lambda t)z^2}|\p_{x_1}^{j}\p_{x_2}^{k} (\delta\theta\p_\theta)^{l} f|^2\, \d x_h \d \theta .
\end{align*}
\begin{prop} \label{prop-well-vhp}
 Let $s\geq 11$ and $$\underline{U_{osc}^{I,0}}, \,\underline{v_h^{\INS}}\in L_{T_2}^{\infty}H^{s-\f12}(\mathbb{T}_{a_1,a_2}^2), \quad \underline{\pt v_h^{\INS}}\in L_{T_2}^{\infty}H^{s-\f52}(\mathbb{T}_{a_1,a_2}^2).$$ There exist $\delta>0$ small enough, $\lambda>0$ large enough and
   $T_4\in (0,\min\{T_2,  1/\lambda\}]$ independent of $\ep,\nu,$ such that the equation \eqref{prandtl-bo} admits a unique solution in $C([0,T_4], H_{\delta,\lambda}^{s-1}(\tilde{\Omega}))\cap L^2([0,T_4], H_{\delta,\lambda}^{s}(\tilde{\Omega})).$ 
\end{prop}
The above Proposition is proved by establishing the a-priori estimates on $v_h^p$ which is achieved by 
performing energy estimate on the new unknown 
$$U:= v_h^p-e^{(2-\lambda t)z^2}\cF_{k_h\rightarrow x_h}\big(e^{-z|k_h|}\cF_{x_h}(\underline{v_h^{\INS}})(k_h)\big),$$
The computations are similar to what have been performed in
\cite{MR3843301}, which is the situation when $\underline{U_{osc}}=0,$ we thus omit the details.

\subsubsection{Well-posedness of the non-oscillating interior part $V^{I,1}$}
\begin{prop}\label{prop-well-vI1}
    Let $s\geq 11$ and $\underline{v_3^{p,1}}\in L_{T_4}^{2}H^{s-3/2}(\mathbb{T}_{a_1,a_2}^2)$
    and $U^{I,0}\in  L_{T_4}^{\infty}H^{s} .$ Then there exists $T_5\in (0,T_4]$ which is uniform in $\ep,\nu,$ such that 
    there is a unique solution of \eqref{eq-ve1} in $
    C([0,T_5], \oH_{co}^{s-5}).$  
\end{prop}
Since  the time derivative of $ \underline{v_3^{I,1}}$ is uniformly bounded in $L_{T_1}^{\infty}H^{s-7/2}(\mathbb{T}_{a_1,a_2}^2)$,  one can always take a divergence-free vector to lift the boundary condition  up without destroying the symmetric structure $(\div v^{I,1}, \na\sigma^{I,1})/\ep.$ The a-priori estimates then follows from the usual conormal energy estimates. 

\section*{Acknowledgement}
N. Masmoudi is
supported by Tamkeen under the NYU Abu Dhabi Research Institute grant CG002.  
C. Sun is partially supported by ANR 
project ANR-24-CE40-3260. C. Wang is supported by NSF of China under Grant 12071008.
Z. Zhang is partially supported by NSF of China under Grant 12171010 and
12288101.

Part of this work was conducted during the second author's visits to NYU Abu Dhabi and Peking University. The warm hospitality and support of these institutions are sincerely acknowledged. The second author would also like to thank Fr\'ed\'eric Rousset for some useful discussions at the early stage of this work.


\bibliographystyle{plain}
\nocite{*}
\bibliography{ref}
\end{document}